%% file: main_qualitative_quasiinvariance.tex
\documentclass[12pt,reqno]{amsart}

\usepackage[a4paper, margin=1in]{geometry}

\usepackage{hyperref}
\hypersetup{
	colorlinks=true,
	linkcolor=blue,
	pdfpagemode=FullScreen,
}

\input{commands}

\title[Quasi-invariance Gaussian measures for NLS]{Qualitative quasi-invariance of low regularity Gaussian measures for the 1D quintic nonlinear Schrödinger equation}

\author{Alexis Knezevitch}

\begin{document}
	
	\address{ENS de Lyon site Monod
		UMPA UMR 5669 CNRS
		46, allée d’Italie
		69364 Lyon Cedex 07, FRANCE}
	
	\email{alexis.knezevitch@ens-lyon.fr}
	
	\begin{abstract}
		We consider the 1d quintic nonlinear Schrödinger equation (NLS) on the torus with initial data distributed according to the Gaussian measures with covariance operator $(1-\Delta)^{-s}$, and denoted $\mu_s$. For the full range $s>\frac{9}{10}$, we prove that these Gaussian measures are quasi-invariant along the flow of (NLS), meaning that the law of the solution at any time is absolutely continuous with respect to the initial Gaussian measure. Moreover, the condition $s>\frac{9}{10}$ corresponds to the threshold where the Sobolev space $H^{\frac{2}{5}+}(\mathbb{T})$ is of $\mu_s$-full measure (it is of zero $\mu_s$-measure otherwise). This is the lower regularity Sobolev space where we currently know that (NLS) is globally well-posed, thanks to the work of LI-WU-XU in \cite{Li_Wu_Xu_global}. The present work is partially an extension of \cite{knez24transportlowregularitygaussian} : we extend the quasi-invariance from the range $s>\frac{3}{2}$ to the range $s>\frac{9}{10}$, but we do not obtain here quantitative results on the Radon-Nikodym derivatives generated by the quasi-invariance. Our approach is based on the work of Sun-Tzvetkov in \cite{sun2023quasi}, combining a Poincaré-Dulac normal form reduction with energy estimates. However, our main tool to obtain these energy estimates differs: we use the Boué-Dupuis variational formula instead of Wiener Chaos.
	\end{abstract}
	
	\maketitle
	
	\tableofcontents
	
	\section{Introduction}\label{section Introduction}\input{introduction}
	
	\section{Preliminaries}\label{section Preliminaries}\input{Preliminaries}

	\section{Poincaré-Dulac normal form reduction and energy estimates}\label{section truncated system Poinacaré-Dulac and energy estimates}\input{truncated_system_normal_form_energy_estimates}

	\section{Transport of Gaussian measures by the truncated flow}\label{section Transport of GM by the truncated flow}
	\input{truncated_transport}

	\section{Proof of the quasi-invariance}\label{section proof of the quasi-inv}\input{proof_quasi_inv}
	
	\section{Lower bound on the resonant function and counting bounds}\label{section lower bound omega and counting bounds}
	\input{lower_bound_res_func_and_counting_bounds}

	\section{Deterministic estimates}\label{section Deterministic estimates}\input{Deterministic_estimates}

	\section{Exponential integrability of the energy functionals}\label{section Exponential integrability}
	\input{Exponential_integrability}

	\appendix 
	
	\section{Construction of a global flow and approximation}\label{appendix construction of a global flow and approximation}\input{appendix_approximation_of_the_flow_by_the_truncated_flow}

	\bibliographystyle{siam}
	\bibliography{refs_qualitative_quasiinv}

\end{document}

%% file: commands.tex
\usepackage{amssymb,amsfonts}
\usepackage[all,arc]{xy}
\usepackage{enumerate}
\usepackage{mathrsfs,stmaryrd}
\usepackage{relsize} 
\usepackage{enumitem}

\usepackage{pgf, tikz}

\newtheorem{thm}{Theorem}[section]
\newtheorem{cor}[thm]{Corollary}
\newtheorem{prop}[thm]{Proposition}
\newtheorem{lem}[thm]{Lemma}

\theoremstyle{definition}
\newtheorem{defn}[thm]{Definition}

\newtheorem{notn}[thm]{Notation}
\newtheorem{notns}[thm]{Notations}

\theoremstyle{definition}
\newtheorem{rem}[thm]{Remark}

\numberwithin{equation}{section}

\newcommand{\hsp}{\hspace{0.1cm}}

\newcommand{\cjg}[1]{%
  \overline{#1}%
  }                     
  
\newcommand{\tld}{\widetilde}

\newcommand{\lra}{\longrightarrow}
\newcommand{\ra}{\rightarrow}

\newcommand{\tendsto}[1]{%
\underset{#1}{\lra}%
}

\renewcommand{\Re}{\textnormal{Re}}
\renewcommand{\Im}{\textnormal{Im}}

\newcommand{\C}{\mathbb{C}}		
\newcommand{\E}{\mathbb{E}}		

\newcommand{\N}{\mathbb{N}}		

\renewcommand{\P}{\mathbb{P}}	
\newcommand{\R}{\mathbb{R}}		
\newcommand{\T}{\mathbb{T}}

\newcommand{\Z}{\mathbb{Z}}		

\usepackage{dsfont}
\renewcommand{\1}{\mathds{1}}			

\newcommand{\cA}{\mathcal A}		
\newcommand{\cB}{\mathcal B}		
\newcommand{\cC}{\mathcal C}	
\newcommand{\cD}{\mathcal D}
	
\newcommand{\cF}{\mathcal F}		
\newcommand{\cL}{\mathcal L}		
\newcommand{\cM}{\mathcal M}	
\newcommand{\cN}{\mathcal N}	
\newcommand{\cS}{\mathcal S}
\newcommand{\cT}{\mathcal T}		

\newcommand{\frkF}{\mathfrak{F}}

\newcommand{\eps}{\varepsilon}

\newcommand{\norm}[1]{%
	\lnorm #1 \rnorm%
}
\newcommand{\triplenorm}[1]{{\left\vert\kern-0.25ex\left\vert\kern-0.25ex\left\vert #1 
		\right\vert\kern-0.25ex\right\vert\kern-0.25ex\right\vert}}

\newcommand{\hsigball}[1]{%
	B_{#1}^{H^{\sigma}}%
}                   

\newcommand{\p}{\partial}

\newcommand{\lnorm}{\left\lVert} 
\newcommand{\rnorm}{\right\rVert} 

\newcommand{\hsig}{H^{\sigma}}
\newcommand{\hsigt}{\hsig(\T)}
\newcommand{\hsignorm}[1]{%
	\lnorm #1 \rnorm_{\hsig}%
}

\newcommand{\linc}{k_1-k_2+...-k_6=0}
\newcommand{\Omgz}{\Omega(\vec{k}) = 0}
\newcommand{\Omgk}{\Omega(\vec{k})}
\newcommand{\Omgeqkp}{\Omega(\vec{k}) = \kappa}
\newcommand{\Omgnz}{\Omega(\vec{k}) \neq 0}
\newcommand{\lincba}{p_1-p_2+...+p_5 =k_1}

%% file: introduction.tex
We consider the 1d quintic nonlinear equation on the torus $\T := \R/2\pi\Z$ :

\begin{equation}\label{NLS}
	\begin{cases}
		i\p_t u + \p_x^2 u = |u|^4u, \hspace{.5cm} (t,x) \in \R \times \T \\
		u|_{t=0} = u_0 \in \hsigt
	\end{cases}
	\tag{NLS}
\end{equation}
with initial data lying in the Sobolev space (for $\sigma \in \R$):
\begin{equation*}
	\hsigt := \{ w : \T \ra \C \hsp : \hsp \sum_{n \in \Z} \langle n \rangle^{2\sigma} |\widehat{w}(n)|^2 < +\infty\}
\end{equation*}
where $\widehat{w}$ denotes the Fourier transform, and $\langle n \rangle :=  (1+n^2)^{\frac{1}{2}}$ the japanese bracket. \\

\noindent \textbf{Well-posedness for NLS :} Let us start by reviewing some known results on the well-posedness for this equation. When $\sigma>\frac{1}{2}$, the algebra property of $\hsigt$ allow us to prove the existence of local solution to~\eqref{NLS} running a fixed point argument in the space $\cC([-T,T],\hsig)$, with $T>0$ a small time depending on the $\hsig$-norm of the initial data. When $\sigma \leq \frac{1}{2}$, the situation is more involved, especially for $\sigma$ small where it is not clear how to define the non-linearity $|u|^4u$ for a general function $u \in \cC([-T,T],\hsig)$, even in the sense of distributions. In~\cite{Bourgain1993}, Bourgain introduced appropriate spaces for the solution to exist and, indeed, proved  the local well-posedness of~\eqref{NLS} for all $\sigma >0$ using a fixed point argument with these new spaces. Applying this procedure, we recover the fact that the non-global solutions of~\eqref{NLS} must have a $\hsigt$-norm that blow up in finite time. Hence, when $\sigma \geq 1$, we obtain global solutions thanks to the conservation of the mass and of the Hamiltonian:
\begin{align*}
	M(u) &:= \int_\T |u|^2 dx & H(u) &:= \frac{1}{2} \int_\T |\p_x u|^2 + \frac{1}{6} \int_\T |u|^6 dx
\end{align*}
because it lead to a control on the $\hsig$-norm of the solutions. For $0<\sigma <1$, this argument cease to work because the Hamiltonian may be infinite. However, in~\cite{bourgain2004remark}, Bourgain combined the $I$-method with a normal form reduction to obtain the global well-posedness for all $\sigma > \sigma_*$ for some $\sigma_*<\frac{1}{2}$; and this result was then extended by LI-WU-XU in~\cite{Li_Wu_Xu_global} where the global-wellposedness is proven for all $\sigma>\frac{2}{5}$. Thus, for every $\sigma > \frac{2}{5}$, we are able to define the flow:
\begin{equation}\label{flow of nls}
	\begin{split}
		\Phi(t) \hsp :  \hsp & \hsigt \lra \hsigt \\
						& u_0 \longmapsto \textnormal{the solution of~\eqref{NLS} at time $t$}
	\end{split}
\end{equation}
for every $t \in \R$. To the knowledge of the author, $\sigma > \frac{2}{5}$ is the current threshold for the global well-posedness of~\eqref{NLS}. It is however worth noting that, thanks to~\cite{schippaGWP},~\eqref{NLS} is globally well-posed in $\hsigt$, for $\sigma>\frac{1}{3}$, under the condition that the $L^2(\T)$-norm of the initial data is small enough. \\

\noindent \textbf{Random Gaussian initial data and statement of the result :}
In this paper, the initial data are random (denoted now $\phi^\omega$ instead of $u_0$) and given by the random Fourier series:
\begin{equation}\label{phi omega}
	 \phi^\omega := \sum_{n \in \Z} \frac{g_n(\omega)}{\langle n \rangle^s}e^{inx}
\end{equation}
where $s \in \R$, and $(g_n)_{n \in \Z}$ are independent complex standard Gaussian measures on a probability space $(\Omega,\cF,\P)$. The law of these random variables defined the Gaussian measures $\mu_s$:
\begin{equation*}
	\mu_s = \textnormal{law} (\phi^\omega) = \phi^\omega_\# \P
\end{equation*}
where the symbol $\#$ denotes the push-forward measure ($\phi^\omega_\# \P(A) := \P((\phi^\omega)^{-1}(A))$). 
   More precisely, for $\sigma \in \R$, the series in~\eqref{phi omega} converges in the space $L^2(\Omega,\hsigt)$ if and only if $\sigma<s-\frac{1}{2}$, because:
\begin{equation*}
	\sum_{n \in \Z} \| \frac{g_n}{\langle n \rangle^s} e^{inx}\|^2_{L^2(\Omega,\hsig)} = \sum_{n \in \Z} \langle n \rangle^{2(\sigma -s )} < +\infty \iff \sigma < s -\frac{1}{2}
\end{equation*}
We obtain that for $\sigma<s-\frac{1}{2}$, the random variable $\phi^\omega$ is valued in $\hsigt$; and therefore $\mu_s$ is a probability measure on the Borel $\sigma$-algebra $\cB(\hsigt)$. It is known that almost-surely, 
\begin{align}\label{phi in h - but not in h}
	\phi^\omega &\in H^{s-\frac{1}{2}-}(\T) := \bigcap_{r<s-\frac{1}{2}} H^r(\T) & &\textnormal{and,}  & \phi^\omega &\notin H^{s-\frac{1}{2}}(\T) 
\end{align}
Consequently, to apply the global flow in~\eqref{flow of nls} to $\phi^\omega$, we need the condition $s-\frac{1}{2}>\frac{2}{5}$, that is $s>\frac{9}{10}$. We emphasize that throughout this paper, for $s> \frac{9}{10}$, we fix $\sigma<s - \frac{1}{2}$ close to $s - \frac{1}{2}$ (close enough to $s-\frac{1}{2}$ so that we have in particular $\sigma>\frac{2}{5}$), and see $\mu_s$ as a probability measure on $\cB(\hsigt)$. This choice of $\sigma$ is not particularly relevant, as we could have defined $\mu_s$ intrinsically on $H^{s-\frac{1}{2}-}(\T)$.\\

The random initial data $\phi^\omega$ generate solutions that are also random. Therefore, one can study the law of the solutions at any given time $t \in \R$, and compare it to the law of the initial data; in particular one may wonder if they are equal, absolutely continuous with respect to each other, or mutually singular. We can reformulate this in terms of transport of Gaussian measures by the flow. Indeed, the law of the solution at a time $t \in \R$ is:
\begin{equation*}
	\textnormal{law}(\Phi(t)\phi^\omega) = (\Phi(t) \circ \phi^\omega)_\# \P = \Phi(t)_\#({\phi^\omega}_\# \P) = \Phi(t)_\# \mu_s
\end{equation*} 
and one can compare the transported measure $\Phi(t)_\# \mu_s$ with the initial measure $\mu_s$. \\
The study of Hamiltonian PDEs with random initial data stems from physical motivation in the work of Lebowitz-Rose-Speer in~\cite{lebowitz1988statistical}, and was then developed by Bourgain in \cite{bourgain1994periodic}. In these works, the initial data are distributed according to Gibbs measures (denoted $Gb$ here), defined formally as $Gb = \frac{1}{Z} e^{-H(u)} du$, where $H$ is the Hamiltonian and $du$ is formally the Lebesgue measure (which does not exist on infinite dimensional vector space). Formally, the conservation of the Hamiltonian and the invariance of the Lebesgue measure under an Hamiltonian flow (Liouville's theorem) imply that:
\begin{equation*}
	\Phi(t)_\# Gb = \frac{1}{Z} e^{-H(\Phi(-t)u)} \Phi(t)_\# du = \frac{1}{Z} e^{-H(u)} du = Gb
\end{equation*}
which means that the Gibbs measure is invariant under the flow, or equivalently that the law of the random solutions (generated by initial data distributed according to the Gibbs measure) is $Gb$ itself, at any time.
This result was rigorously proven by Bourgain in \cite{bourgain1994periodic} for~\eqref{NLS}. Since then, the invariance of Gibbs measures has been studied for many models. These measures appears naturally as density measures with respect to suitable Gaussian measures. In \cite{tzvetkov2015quasiinvariant}, Tzvetkov initiated a program on the study of Hamiltonian PDEs with initial data distributed according to Gaussian measures. Contrary to the Gibbs measure, we do not expect these measures to be invariant along the flow. Indeed, this is what the following formal computations suggest:  formally, we have $\mu_s = \frac{1}{Z_s}e^{-\frac{1}{2}\| u \|_{H^s}^2} du$, so that:
\begin{equation*}
	\Phi(t)_\# \mu_s = \frac{1}{Z_s}e^{-\frac{1}{2}\| \Phi(-t) u \|_{H^s}^2} \Phi(t)_\#du = e^{-\frac{1}{2}(\| \Phi(-t) u \|_{H^s}^2-\| u \|_{H^s}^2)} \frac{1}{Z_s}e^{-\frac{1}{2}\| u \|_{H^s}^2}du 
\end{equation*}
that is,
\begin{equation}\label{formal transport of mus}
	\Phi(t)_\# \mu_s = e^{-\frac{1}{2}(\| \Phi(-t) u \|_{H^s}^2-\| u \|_{H^s}^2)} d\mu_s
\end{equation}
Since in general the flow $\Phi(t)$ does not preserve the $H^s$-norm, we expect $\Phi(t)_\# \mu_s$ to be not equal but absolutely continuous with respect to $\mu_s$. If that is indeed the case for every $t \in \R$, we say that $\mu_s$ is \textit{quasi-invariant} along the flow. From the Radon-Nikodym theorem, $\mu_s$ is quasi-invariant along the flow, if and only if, for every $t \in \R$, and every $A \in \cB(\hsigt)$:
\begin{equation*}
	\mu_s(A) = 0 \implies \mu_s(\Phi(t)^{-1}(A)) = 0
\end{equation*}
Formula~\eqref{formal transport of mus}, despite being formal, is of interest. Practically, one can test this formula with other Hamiltonian PDEs. For example, considering the propagator $e^{it\p_x^2}$ of linear Schrödinger equation: 
\begin{equation}\label{LS}
	\begin{cases}
		i\p_t u + \p_x^2 u = 0, \hspace{.5cm} (t,x) \in \R \times \T \\
		u|_{t=0} = u_0 
	\end{cases}
	\tag{LS}
\end{equation}
we recover the fact that $(e^{it\p_x^2})_\# \mu_s = \mu_s$, since $e^{it\p_x^2}$ preserves the $H^s$-norm. Moreover, for Hamiltonian PDEs that preserves the $L^2$-norm, we can expect the \textit{white noise}, that is the measure $\mu_0$, to be invariant along the flow. Technically, this formula will be useful because it will hold true when replacing~\eqref{NLS} with an approximated truncated equation. \\

Quasi-invariant results have been proven recently for many models, see \cite{burq2024almost,coe2024sharp,debussche2021quasi,forlano2025improvedquasiinvarianceresultperiodic,forlano_and_soeng2022transport,forlano2022quasi,forlano_tolomeo_invariance_stochastic_wave,forlano_and_trenberth2019transport,genovese2022quasi,genovese2023quasi,genovese2023transport,gunaratnam2022quasi,oh_soeng2021quasi,oh2018optimal,oh2019quasi,oh2017quasi,oh2020quasi,planchon2020transport,planchon2022modified,sosoe2020quasi,tzvetkov2015quasiinvariant}. In the case of~\eqref{NLS}, the first result on quasi-invariance emerged as a consequence of the invariance of the Gibbs measures (see again \cite{bourgain1994periodic}). This measure being rigorously defined as:
\begin{equation*}
	Gb = e^{-\frac{1}{6} \| u\|^6_{L^6(\T)}} d\mu_1,
\end{equation*}
we obtain that 
\begin{equation*}
	\Phi(t)_\# Gb = Gb \implies \Phi(t)_\# \mu_1 = e^{-\frac{1}{6}(\| \Phi(-t)u\|^6_{L^6(\T)} -  \| u\|^6_{L^6(\T)})} d\mu_1
\end{equation*}
which implies the quasi-invariance for $\mu_1$. Later on, Planchon-Tzvetkov-Visciglia proved in \cite{planchon2020transport} the quasi-invariance of $\mu_s$ for all $s=2k$, with $k \in \N \setminus \{0\}$. This result was then extended to all $s>\frac{3}{2}$ in \cite{knez24transportlowregularitygaussian}. As mention above, knowing that there is a global flow to~\eqref{NLS} in $\hsigt$, for all $\sigma > \frac{2}{5}$, one can define the transported measure $\Phi(t)_\# \mu_s$ for all $s>\frac{9}{10}$, and wonder if quasi-invariance holds for all $s>\frac{9}{10} $. In this paper, we provide a positive answer to this question:
\begin{thm}\label{thm quasi-inv}
	For every $s > \frac{9}{10}$, the Gaussian measure $\mu_s$ is quasi-invariant along the flow of~\eqref{NLS}. In other words, the push-forward measure $\Phi(t)_\# \mu_s$ is absolutely continuous with respect to $\mu_s$ for every time $t\in \R$.
\end{thm} 

\noindent \textbf{Overview of the proof and comments :} Let us first remark that, writing as it is, the formal density in~\eqref{formal transport of mus} is ill-defined on the support of $\mu_s$. Indeed, from~\eqref{phi in h - but not in h}, we have that $\mu_s(H^{s}(\T))=0$, so both $\| u\|_{H^s}^2$ and $\| \Phi(-t) u\|_{H^s}^2$ are infinite on the support of $\mu_s$. However, it is the difference between these two infinite terms that appears in~\eqref{formal transport of mus}, so we can still hope for some cancellations that might give meaning to this difference. To apply this approach, we first work with a truncated approximated version of~\eqref{NLS}:
\begin{equation*}
	\begin{cases}
		i\p_t u + \p_x^2 u = \pi_N \left(|\pi_Nu|^4\pi_Nu \right), \hspace{0.2cm} (t,x) \in \R \times \T \\
		u|_{t=0} = u_0 
	\end{cases}
\end{equation*}
with its flow denoted $\Phi^N(t)$, for which we rigorously have:
\begin{align}\label{intro trsprt mus by PhiN}
	\Phi^N(t)_\# \mu_s &= G_{s,N,t} \mu_s & &\textnormal{with} & G_{s,N,t}(u) := e^{-\frac{1}{2}(\| \pi_N \Phi^N(-t)u\|_{H^s}^2-\| \pi_N u\|_{H^s}^2)}
\end{align}
The challenge then is to "take the limit $N \ra \infty$" into this equality, that is, deduce from this the quasi-invariance of $\mu_s$ along the original flow $\Phi(t)$. First and foremost, we need to prove that $\Phi^N(t)$ is indeed an approximation of $\Phi(t)$ (in a suitable way). Thereafter, we can employ different approach which can lead to the quasi-invariance of $\mu_s$ under $\Phi(t)$ that we mention (informally): \\
-- Find directly a limit to the \textit{truncated densities} $G_{s,N,t}$ when $N \ra \infty$. To do this, one can try to prove uniform (in $N \in \N$) $L^p(d\mu_s)$ bound for the $G_{s,N,t}$ (with eventually an additional cut-off) which allow to extract from the $G_{s,N,t}$ a sequence that weakly-converges to a function $G_{s,t}$ (for such an example we refer to \cite{debussche2021quasi}). Other types of convergence can lead to the quasi-invariance; for example, in \cite{knez24transportlowregularitygaussian}, uniform convergence on compact sets (for continuous functions) was used. \\
-- Prove a quantitative inequality of the type (for some $\alpha \in (0,1)$):
\begin{equation}\label{intro quantitative ineq}
	\mu_s(\Phi^N(t)A) \leq C_{s,t} \mu_s(A)^{1-\alpha} \hsp \hsp \textnormal{which transfers to} \hsp \hsp \mu_s(\Phi(t)A) \leq C_{s,t} \mu_s(A)^{1-\alpha}
\end{equation}
with $A$ a Borel set of $\hsigt$, which is possibly contained in a ball  $B^{\hsig}_R$ of $\hsigt$ (or even compact) in which case the constant $C_{s,t}$ depends also on $R>0$ (for an example of the application of this method we refer to \cite{tzvetkov2015quasiinvariant}). \\
We also refer to the introduction of \cite{forlano_and_trenberth2019transport} where different methods to prove quasi-invariance are reviewed. \\

In this paper, we follow the second approach. In any case, for both method, we need to analyze the difference $\| \pi_N \Phi^N(-t)u\|_{H^s}^2-\| \pi_N u\|_{H^s}^2$, which writes thanks to the fundamental theorem of calculus and the additivity of the flow:
\begin{equation}\label{diffce Hs norm fund thm of calculus}
	\frac{1}{2}\| \pi_N \Phi^N(-t)u\|_{H^s}^2-\frac{1}{2}\| \pi_N u\|_{H^s}^2 = \int_0^{-t} D_{s,N}(\Phi^N(\tau)) d\tau
 \end{equation}
with $D_{s,N}(v) :=  \frac{d}{dt}\big( \frac{1}{2}\| \pi_N \Phi^N(t)v\|_{H^s}^2\big)|_{t=0}$. In fact, we will work with a rewriting of this equality; the application of the Poincaré-Dulac normal reduction from \cite{sun2023quasi} (which essentially relies on an integration by part in the formula above) leads to consider intermediate objects, a \textit{modify energy} and a \textit{weighted Gaussian measure}: 
\begin{align*}
	E_{s,N}(u) &= \frac{1}{2}\| \pi_N u \|_{H^s}^2  +R_{s,N}(u) & &\textnormal{and} & \rho_{s,N} = e^{-R_{s,N}(u)} d\mu_s
\end{align*}
with $R_{s,N}$ called the \textit{energy correction}, which transform~\eqref{diffce Hs norm fund thm of calculus} and \eqref{intro trsprt mus by PhiN} into:
\begin{align*}
	E_{s,N}(\pi_N \Phi^N(-t)u) - E_{s,N}(\pi_N u) &=  \int_0^{-t} Q_{s,N}(\Phi^N(\tau)) d\tau & &\textnormal{and} & \Phi^N(t)_\# \rho_{s,N} = F_{s,N,t} d\rho_{s,N}
\end{align*} 
with $F_{s,N,t}(u):=e^{-(E_{s,N}(\pi_N \Phi^N(-t)u) - E_{s,N}(\pi_N u))}$, and $Q_{s,N}(u) :=  \frac{d}{dt}E_{s,N}(\Phi^N(t)u)|_{t=0}$ called \textit{the modified energy derivative}. One may see the modified energies $E_{s,N}$ and the weighted Gaussian measures $\rho_{s,N}$ as analogs for every $s$ of the standard energy and the Gibbs measure:
\begin{align*}
	E(u) &:= \frac{1}{2} \|u \|^2_{L^2} + H(u) =\frac{1}{2} \|u \|^2_{H^1} + \frac{1}{6} \|u \|_{L^6}^6, &  Gb&= e^{-\frac{1}{6} \|u \|_{L^6}^6} d\mu_1
\end{align*} 
where $\frac{1}{6} \|u \|_{L^6}^6$ is seen as a correction term. The significant difference is that $E_{s,N}$ is not conserved by the flow. Consequently, we do not hope for $\rho_{s,N}$ to be invariant under the flow. However, the hope is that $E_{s,N}$ is "closer" to be conserved than the original energy $\frac{1}{2} \|u \|^2_{H^s}$, and that we can recover the quasi-invariance of $\mu_s$ from the quasi-invariance of the $\rho_{s,N}$ (for every $N \in \N$), as we recover the quasi-invariance of $\mu_1$ from the invariance of $Gb$. To transfer the quasi-invariance from the $\rho_{s,N}$ to $\mu_s$, we will establish quantitative inequalities on the sets $\rho_{s,N}(\Phi^N(t)(A))$ (see ~\eqref{quantitative ineq quasi-inv WGM}) uniformly in $N$, following the analysis initiated in \cite{tzvetkov2015quasiinvariant}. These inequalities will be obtained through $L^p$ estimates on both $R_{s,N}$ and $Q_{s,N}$. \\
We emphasize that normal form reductions has already been used in the context of quasi-invariance, see for example \cite{oh2017quasi,oh2018optimal,oh_soeng2021quasi,forlano_and_trenberth2019transport}. \\

\noindent \textbf{Further remarks :} -- In this paper, the main tool to establish the quantitative $L^p$ estimates (for $R_{s,N}$ and $Q_{s,N}$) is the Boué-Dupuis formula.  One can try to obtain these through Wiener Chaos instead, as it was for example performed in \cite{sun2023quasi}; however, it seems to the author that this method does not work (or is more complicated) in order to extend the quasi-invariance of $\mu_s$ from $s>\frac{3}{2}$ (see \cite{knez24transportlowregularitygaussian}) down to $s>\frac{9}{10}$, even though it could extend the quasi-invariance down to $s>s_*$ for some $s_* \in (\frac{9}{10},\frac{3}{2})$. \\
-- It would be interesting to find a sharp threshold for $s$ above which we have quasi-invariance of $\mu_s$ and under which singular results happen, typically that the transported measure $\Phi(t)_\# \mu_s$ and $\mu_s$ are mutually singular. Progress have been made in this direction in \cite{coe2024sharp} in the case of the cubic Szegö equation. However, for \eqref{NLS}, surpassing the threshold $s>\frac{9}{10}$ requires the construction of a global flow in $\hsigt$ for $\sigma\leq \frac{2}{5}$, either deterministically (for every initial data) or probabilistically (for $\mu_s$-almost every initial data). In the latter case, connections between quasi-invariance and well-posedness have been explored in \cite{forlano2022quasi}. \\
-- In this paper, we obtain only the quasi-invariance of $\mu_s$ for all $s>\frac{9}{10}$, without providing any additional quantitative results on the generated Radon-Nikodym derivatives, which may prove useful for other applications. Hence, this paper extends partially the result in \cite{knez24transportlowregularitygaussian}, where $L^p(d\mu_s)$ bounds for the Radon-Nikodym derivatives are obtained (for all $s>\frac{3}{2}$), relying on the conservation of the mass and the Hamiltonian. Note also that in the case $s\leq \frac{3}{2}$, we are not in a position to use the conservation of the Hamiltonian anymore.
For other quantitative results of this type, we refer to \cite{debussche2021quasi,forlano_and_soeng2022transport,forlano_and_trenberth2019transport,planchon2022modified,genovese2023transport}. \\

\noindent \textbf{Organization of the paper :} This paper is organized as follows: \\
First, in Section~\ref{section Preliminaries}, we present some notations and tools that will be used throughout the paper; in particular, we state the Boué-Dupuis fromula which will be crucial for the energy estimates. In Section~\ref{section truncated system Poinacaré-Dulac and energy estimates}, we gather the central properties for the proof of quasi-invariance: we review the Poincaré-Dulac normal form reduction (performed for the truncated flow) in~\cite{knez24transportlowregularitygaussian}, which is an adaptation of the one in~\cite{sun2023quasi}, and then, we state quantitative $L^p$ estimates on the energy functionals produced by this normal form reduction. These estimates will be obtained here under the assumption of exponential integrability properties that are proved in the dedicated Section~\ref{section Exponential integrability} using the Boué-Dupuis formula. In section~\ref{section Transport of GM by the truncated flow}, we study the transport of $\mu_s$ under the truncated flow, notably by establishing~\eqref{intro trsprt mus by PhiN}. Combining the results of Sections~\ref{section truncated system Poinacaré-Dulac and energy estimates} and~\ref{section Transport of GM by the truncated flow} will allow us to prove, in Section~\ref{section proof of the quasi-inv}, the quasi-invariance of $\mu_s$. In preparation for the key Section~\ref{section Exponential integrability}, we develop our deterministic toolbox in Section~\ref{section lower bound omega and counting bounds} and~\ref{section Deterministic estimates}. More precisely, in Section~\ref{section lower bound omega and counting bounds} we prove a lower bound on the \textit{resonant function} and some counting bounds; and, in Section~\ref{section Deterministic estimates} we prove deterministic estimate, notably relying on Strichartz estimates. Finally, in Appendix~\ref{appendix construction of a global flow and approximation}, we review the construction of the flow in $H^{\frac{2}{5}+}(\T)$ (using the Bourgain spaces) and we prove approximation properties of the flow by the truncated flow. \\

\noindent \textbf{Acknowledgment :} The author would like to thank Chenmin Sun and Nikolay Tzvetkov for suggesting this problem, for their support, and for insightful discussions. This work is partially supported by the ANR project Smooth ANR-22-CE40-0017.

%% file: Preliminaries.tex
In this section, we group some notations and tools that we will use throughout the paper.

\subsection{Notations}
To refer to the non-linearities in~\eqref{NLS} and~\eqref{truncated equation}, we consider:
\begin{notn}\label{notation non-linearities} For $u : \T \ra \C$, we denote:
	\begin{align*}
	F(u) & := |u|^4u, &	F_N(u) &:= \pi_N(|\pi_Nu|^4 \pi_Nu)
	\end{align*}
	for $N \in \N$, and $\pi_N$ the Dirichlet projector given by:
	\begin{equation*}
		\pi_N(\sum_{k \in \Z} \widehat{u}(k) e^{ikx}) = \sum_{|k| \leq N} \widehat{u}(k) e^{ikx}
	\end{equation*} 
\end{notn}

\begin{notns} We will intensively use the following notations:\\
-- Given a set of frequencies $k_1,...,k_m \in \Z$, we denote by $k_{(1)},...,k_{(m)}$ a rearrangement of the $k_j's$ such that 
		\begin{equation*}
			|k_{(1)}| \geq |k_{(2)}| \geq ... \geq |k_{(m)}| 
		\end{equation*}
-- Similarly, given a set of dyadic integers $N_1,...,N_m \in 2^{\N}$, we denote by $N_{(1)},...,N_{(m)}$ a non-increasing rearrangement of the $N_j$. \\
-- Moreover, for $\alpha \in \R$, we write concisely $M^{\alpha +}$ (where $M \in \N$) to refer to a function $C(M)$ which satisfies: for every $\eps>0$, there exists a constant $C_\eps >0$ such that for every $M \in \N$, $C(M) \leq C_\eps M^{\alpha +\eps}$. With this notation, note that we have:
\begin{equation*}
	N_1^{\alpha_1 + }N_2^{\alpha_2 + }...N_m^{\alpha_m + } = N_1^{\alpha_1 + }N_2^{\alpha_2  }...N_m^{\alpha_m  }
\end{equation*}
for $N_1,...,N_m$ (dyadic) integers ordered such that $N_1\geq N_2 \geq \dots \geq N_m$. \\
-- Finally, for $\alpha \in \R$, we write concisely $M^{\alpha-}$ to refer to a function $C(M)$ which satisfies: there exists $\eps>0$ and a constant $C_0>0$ such that for every $M \in \N$, $C(M) \leq C_0  M^{\alpha - \eps}$.
\end{notns}

\subsection{The resonant function and the symmetrized derivatives}

Here, we introduce quantities that emerge in the Poincaré-Dulac normal reduction of Section~\ref{section truncated system Poinacaré-Dulac and energy estimates}.
\begin{defn}\label{def psi and omega}
		-- For $\vec{k}=(k_1,...,k_6)\in \Z^6$, we define the quantities:
		\begin{align*}
			\psi_{2s}(\vec{k}) &:= \sum_{j=1}^{6} (-1)^{j-1}|k_j|^{2s}, & \Omega(\vec{k}) &= \sum_{j=1}^6 (-1)^{j-1} k_j^2
		\end{align*}
		called respectively the \textit{symmetrized derivative} (of order $2s$) and the \textit{resonant function}. \\
		 -- Similarly, for $\vec{k}=(k_1,...,k_4)\in \Z^4$, we define the degenerated quantities:
		\begin{align*}
			\psi^{\downarrow}_{2s}(\vec{k}) &:= \sum_{j=1}^{4} (-1)^{j-1}|k_j|^{2s}, & \Omega^{\downarrow}(\vec{k}) &= \sum_{j=1}^4 (-1)^{j-1} k_j^2
		\end{align*}
		called respectively the \textit{degenerated symmetrized derivative} (of order $2s$) and the \textit{degenerated resonant function}.
\end{defn}
\begin{rem}
	For $(k_1,...,k_6) \in \Z^6$, $\psi^{\downarrow}_{2s}$ and $\Omega^{\downarrow}$ are respectively obtained from $\psi_{2s}$ and $\Omega$ when the frequencies $k_5$ and $k_6$ are equal, that is:
	\begin{align*}
		\psi^{\downarrow}_{2s}(k_1,...,k_4) &= \psi_{2s}(k_1,...,k_4,k_5,k_5) & \Omega^{\downarrow}(k_1,...,k_4)& = \Omega(k_1,...,k_4,k_5,k_5)
	\end{align*}
\end{rem}
Classically, we have the factorization:
\begin{lem}[Factorization of the degenerated resonant function]\label{lem factorization dgn Omega}
	Let $(k_1,k_2,k_3,k_4) \in \Z^4$ such that $k_1-k_2+k_3-k_4=0$; then:
	\begin{equation*}
		\Omega^\downarrow(\vec{k}) = 2(k_2-k_1)(k_3-k_2)
	\end{equation*}
\end{lem}
For conciseness, it will be convenient later to consider additionally the following quantities:
\begin{defn}\label{def Big Psi}
	For $\vec{k}=(k_1,...,k_6) \in \Z^6$, we define:
	\begin{align*}
		\Psi_{2s}^{(0)}(\vec{k}) &:= \1_{\Omgz} \psi_{2s}(\vec{k}), & \Psi_{2s}^{(1)}(\vec{k}) &:= \1_{\Omgnz} \frac{\psi_{2s}(\vec{k})}{\Omgk}
	\end{align*}
	Similarly, we also consider the degenerated version of this quantities:
	\begin{align*}
		\Psi_{2s}^{(0),\downarrow}(\vec{k}) &:= \1_{\Omega^\downarrow(\vec{k}) = 0} \psi^\downarrow_{2s}(\vec{k}), & \Psi_{2s}^{(1),\downarrow}(\vec{k}) &:= \1_{\Omega^\downarrow(\vec{k}) \neq 0} \frac{\psi^\downarrow_{2s}(\vec{k})}{ \Omega^\downarrow(\vec{k})}
	\end{align*}
\end{defn}

\begin{lem}\label{lemma psi estimate}
	Let $s>0$. For every $k_1-k_2+k_3-k_4+k_5-k_6 = 0$,
	\begin{equation}\label{psi estimate}
		| \psi_{2s}(\vec{k}) | \leq C_s  \big(k_{(1)} ^{2(s-1)} ( | \Omega(\vec{k}) |+ k_{(3)} ^2) + k_{(3)}^{2s}\big)
	\end{equation}
	where the constant $C_s>0$ only depends on $s$. As a consequence, we have~: 
	\begin{equation}\label{psi estimate when Omg=0}
		| \Psi^{(0)}_{2s}(\vec{k}) | \leq C_s ( k_{(1)} ^{2(s-1)} \vee k_{(3)}^{2(s-1)}) k_{(3)}^2
	\end{equation}
	and,
	\begin{equation}\label{psi/Omg estimate}
		| \Psi^{(1)}_{2s}(\vec{k}) | \leq C_s \big( k_{(1)} ^{2(s-1)} +  (k_{(1)} ^{2(s-1)} \vee k_{(3)}^{2(s-1)}) \frac{k_{(3)}^2}{|\Omgk|} \big)
	\end{equation}
	where the symbol $\vee$ denotes the max. Moreover, the analogous statements hold replacing $\psi_{2s}$, $\Psi^{(0)}_{2s}$, $\Psi^{(1)}_{2s}$, and $\Omega$ by their degenerated version $\psi^\downarrow_{2s}$, $\Psi^{(0),\downarrow}_{2s}$, $\Psi^{(1),\downarrow}_{2s}$, and $\Omega^\downarrow$.
\end{lem}

\begin{proof} We prove Lemma~\ref{lemma psi estimate} only for the original functions $\psi_{2s}$, $\Psi^{(0)}_{2s}$, $\Psi^{(1)}_{2s}$, and $\Omega$, as the proof for the associated degenerated functions is identical. Moreover, we only establish~\eqref{psi estimate}, as~\eqref{psi estimate when Omg=0} and~\eqref{psi/Omg estimate} follow from it.\\
	Thanks to the symmetries, we only need to consider the two following cases : when $k_{(1)}=k_1$, $k_{(2)}=k_2$ and when $k_{(1)}=k_1$, $k_{(2)}=k_3$. \\
	
	\underline{Case 1 :} When $k_{(1)}=k_1$ and $k_{(2)}=k_2$. In this case we use the mean value theorem as follows:
	\begin{multline*}
		|\psi_{2s}(\vec{k})|  \leq k_1^{2s} - k_2^{2s} + |k_3^{2s}-k_4^{2s}+k_5^{2s}-k_6^{2s}|
		\leq \sup_{t \in [k_2^2,k_1^2]} \frac{d}{dt}( t^s) (k_1^2 - k_2^2) + 4 k_{(3)}^{2s} \\
		\leq s (k_{(1)}^{2(s-1)}\vee k_{(2)}^{2(s-1)})(\Omega(\vec{k}) - k_3^2+k_4^2-k_5^2+k_6^2)  + 4 k_{(3)}^{2s}
		\leq  C_s  \big(k_{(1)} ^{2(s-1)} ( | \Omega(\vec{k}) |+ k_{(3)} ^2) + k_{(3)}^{2s}\big)
	\end{multline*}
	where we have used the fact that $|k_{(2)}| \sim |k_{(1)}|$ from the constraint $\linc$. \\
	
	\underline{Case 2} : When $k_{(1)}=k_1$ and $k_{(2)}=k_3$. Let us consider two subcases. Firstly, if $|k_{(3)}| \sim |k_{(1)}|$, then we can simply write:
	\begin{equation*}
		|\psi_{2s}(\vec{k})| \leq 6 k_{(1)}^{2s} \leq C_s k_{(3)}^{2s} \leq C_s  \big(k_{(1)} ^{2(s-1)} ( | \Omega(\vec{k}) |+ k_{(3)} ^2) + k_{(3)}^{2s}\big)
	\end{equation*}
	Secondly, if $|k_{(3)}| \ll |k_{(1)}|$, then, since in addition $k_{(1)}=k_1$ and $k_{(2)}=k_3$, we have:
	\begin{equation*}
		\Omgk = k_{(1)}^2 + k_{(2)}^2 + o(k_{(1)}^2) \geq \frac{1}{2} k_{(1)}^2
	\end{equation*} 
	and it implies that
	\begin{equation*}
		|\psi_{2s}(\vec{k})| \leq 6 k_{(1)}^{2s} \leq 12 k_{(1)}^{2(s-1)}|\Omgk| \leq C_s  \big(k_{(1)} ^{2(s-1)} ( | \Omega(\vec{k}) |+ k_{(3)} ^2) + k_{(3)}^{2s}\big)
	\end{equation*}  
	This completes the proof of Lemma~\ref{lemma psi estimate}.
\end{proof}

\subsection{Probabilistic toolbox} In this paragraph, we present the two main probabilistic tools that we will use for the energy estimates in Section~\ref{section Exponential integrability}. The first one is often referred to as the \textit{square-root cancellation}. The second one is the Boué-Dupuis variational formula.

\begin{lem}\label{lem L2 estimate with orthogonality} 
	Let $(\Omega,\cF,\P)$ be a probability space. Let $m\in \N$, $\iota_1,...,\iota_m \in \{-,+ \}$, and $g_{k_1},...,g_{k_m}$ be complex standard i.i.d Gaussians. We consider the following multi-linear expression of Gaussians:
	\begin{equation}\label{mult lin gauss}
		G(\omega) := \sum_{\substack{k_1,...,k_m \\ \forall \iota_i \neq \iota_j, \hsp k_i \neq k_j}} c_{k_1,...,k_m} \cdot \prod_{j=1}^m g_{k_j}^{\iota_j}(\omega), \hspace{1cm} \omega \in \Omega 
	\end{equation}
	where $c_{k_1,...,k_m}$ is a sequence of $l^2(\Z^m;\C)$. Then, there exists $C>0$ such that:
	\begin{equation}\label{ineq L2 estimate orhtogonality}
		\norm{G}_{L^2(\Omega)} \leq C \biggl( \sum_{\substack{k_1,...,k_m \\ \forall \iota_i \neq \iota_j, \hsp k_i \neq k_j}} | c_{k_1,...,k_m} |^2 \biggr)^{\frac{1}{2}}
	\end{equation}
\end{lem}

\begin{rem}
	The proof of this lemma relies on the independence and the fact that if $g$ is a complex standard Gaussian random variable, then for every $k,l \in \N$, $\E[g^k \cjg{g}^l] = 0$ if $k \neq l$. For a detailed proof (in the case $m=3$) we refer to \cite{knez24transportlowregularitygaussian}.
\end{rem}

\begin{rem}
	When $\iota_i \neq \iota_j$  \textbf{and} $\hsp k_i = k_j$, we say that $k_i$ and $k_j$ are paired. If one allows such parings in the sum in \eqref{mult lin gauss}, then the inequality \eqref{ineq L2 estimate orhtogonality} does not hold anymore in general. This will force us later to treat separately such pairing contributions.
\end{rem}

In this paper, the key element in order to obtain exponential integrability for the energy functionals (defined in the forthcoming Section~\ref{section truncated system Poinacaré-Dulac and energy estimates}) is the Boué-Dupuis variational formula introduced in \cite{boue_dupuis98}. See also the related paper by Üstünel \cite{ustunel14}. This approach was adopted by Barashkov and Gubinelli in \cite{barashkov_gubinelli20}, and it has since been applied in the context of quasi-invariance in \cite{forlano_and_soeng2022transport}, \cite{forlano2022quasi},\cite{coe2024sharp}. More precisely, we will use the following simplified version of the Boué-Dupuis formula, which has been stated and proven in Lemma 2.6 of \cite{forlano2022quasi}. 

\begin{lem}[see Lemma 2.6 of \cite{forlano2022quasi}]\label{lem boue-dupuis}
	Let $s\in \R$ and $N \in \N$. Let $\cF:C^\infty(\T)\ra\R$ be a measurable function such that $\E[|\cF_-(\pi_N \phi^\omega)|^p]~< \infty$, for some $p>1$, where $\cF_-$ denotes the negative part of $\cF$. Then,
	\begin{equation}\label{ineq boue dupuis}
		\log \int e^{\cF(\pi_N u)} d\mu_s	\leq \E \big[\sup_{V \in H^s} \{ \cF(\pi_N \phi^\omega + \pi_N V) - \frac{1}{2}\norm{V}^2_{H^s} \} \big]
	\end{equation}
	where we recall that $\phi^\omega$ is the random variable in~\eqref{phi omega}, whose law is $\mu_s$.
\end{lem}

%% file: truncated_system_normal_form_energy_estimates.tex
\subsection{Truncated system}For $N\in \N$, we work with the following equation :
\begin{equation}\label{truncated equation}
	\begin{cases}
		i\p_t u + \p_x^2 u = \pi_N \left(|\pi_Nu|^4\pi_Nu \right), \hspace{0.2cm} (t,x) \in \R \times \T \\
		u|_{t=0} = u_0 
	\end{cases}
\end{equation}
that we call the \textit{truncated equation}, where we recall that $\pi_N$ is the projector on frequencies $\leq N$. We also define $\pi_N^\perp:=\text{Id}-\pi_N$, the projector on frequencies $> N$. More explicitly,
\begin{align*}
	\pi_N \big( \sum_{k \in \Z} u_k e^{ikx} \big) &= \sum_{|k| \leq N} u_k e^{ikx} & &\textnormal{and} & 	\pi_N^{\perp} \big( \sum_{k \in \Z} u_k e^{ikx} \big) &= \sum_{|k| > N} u_k e^{ikx}
\end{align*} 
We denote by $\Phi^N$ the flow of \eqref{truncated equation}, that we call the \textit{truncated flow}. We also recall that we denote by $\Phi$ the flow of~\eqref{NLS}. The truncated flow $\Phi^N$ can be factorized: 
\begin{prop}[Factorization of the truncated flow, see Appendix of \cite{knez24transportlowregularitygaussian}]\label{prop factorization of the truncated flow}
	Let $N \in \N$, $s \in \R$, and $\sigma < s-\frac{1}{2}$.  Set,
	\begin{align*}
		E_N &:= \pi_N \hsigt, & E_N^{\perp} := \pi_N^{\perp}\hsigt = (\textnormal{Id} -\pi_N) \hsigt
	\end{align*}
	We identify $\hsigt$ with the product space $E_N \times E_N^\perp$ through the isomorphism:
	\begin{equation*}
		\begin{split}
			E_N &\times E_N^\perp \lra \hsigt \\
			 (w&,v) \longmapsto w+v
		\end{split}
	\end{equation*}
	 Then, for every $t \in \R$, the truncated flow $\Phi^N(t)$ can be factorized as $(\tld{\Phi}^N(t),e^{it\p_x^2})$ on $E_N \times E_N^{\perp}$ in the sense that:
	\begin{equation*}\label{factorization of the truncated flow}
		\begin{split}
			\Phi^N(t) : \hsp&E_N \times E_N^{\perp} \lra  E_N \times E_N^{\perp}  \\
			& (w,v) \longmapsto (\tld{\Phi}^N(t)(w),e^{it\p_x^2}v)
		\end{split}
	\end{equation*}
	where $e^{it\p_x^2}$ is the propagator of the linear Schrödinger equation~\eqref{LS}, and where $\tld{\Phi}^N$ is the flow on $E_N$ of the following finite dimensional ordinary differential equation:
	\begin{equation}\label{FNLS}
		\begin{cases}
			i\p_t w + \p_x^2 w = \pi_N \left(|w|^4w \right) \\
			u|_{t=0} = w_0 \in E_N
		\end{cases}
		\tag{FNLS}
	\end{equation}
\end{prop}

 Moreover, the truncated flow $\Phi^N$ is an approximation of the original $\Phi$ in the following sense:

\begin{prop}[Approximation of the flow by the truncated flow]\label{prop approx Phi by PhiN set version} Let $\sigma>\frac{2}{5}$ and $K$ a compact subset of $\hsigt$. Then, \\
-- If $T>0$, we have that for any $\eps > 0$, there exists $N_{\eps,K,T} \in \N$ such that for all $N\geq N_{\eps,T,K}$, and every $|t|\leq T$:
\begin{equation}\label{PhiN K included PhiK + B}
	\Phi^N(t)(K) \subset \Phi(t)(K) + \hsigball{\eps}
\end{equation}
-- If $t \in \R$, we have that for any $\eps > 0$, there exists $N_{\eps,K,t} \in \N$ such that for all $N\geq N_{\eps,K,t}$:
\begin{equation}\label{PhiK included PhiN(K+B)}
	\Phi(t)(K) \subset \Phi^N(t)(K + \hsigball{\eps})
\end{equation}
\end{prop}

 We provide a proof of this proposition in Appendix~\ref{appendix construction of a global flow and approximation}. \\
 
Later, we will see that (for all $s\in\R$) $\mu_s$ is quasi-invariant under $\Phi^N(t)$. More precisely, we will be able to give an explicit formula for the Radon-Nikodym derivative $\frac{d(\Phi^N(t)_\# \mu_s)}{d\mu_s}$. The key ingredients for transferring the quasi-invariance from $\Phi^N(t)$ to the original flow $\Phi(t)$ is the approximation property above combined with uniform in $N\in \N$ estimates on \textit{energy functionals} which are produced by the following normal form reduction.  

\subsection{Poincaré-Dulac normal form reduction} In this paragraph, we state the result obtained through the normal form reduction performed in~\cite{knez24transportlowregularitygaussian}, which is a slight adaptation of that in~\cite{sun2023quasi}. To that end, we first define the following \textit{energy functionals}:

\begin{defn}\label{def energy functionals}
	Firstly, we define $\cM_s : C^\infty(\T)^6 \ra \R$ and $\cT_s : C^\infty(\T)^6 \ra \R$ the two following multi-linear maps:
	\begin{equation*}
		\cM_s(u_1,...,u_6):= \sum_{\linc} \Psi^{(0)}_{2s}(\vec{k})\widehat{u}_1(k_1)\cjg{\widehat{u}_2}(k_2)...\cjg{\widehat{u}_6}(k_6)
	\end{equation*}
	and,
	\begin{equation*}
		\cT_s(u_1,...,u_6):= \sum_{\linc} \Psi^{(1)}_{2s}(\vec{k}) \widehat{u}_1(k_1)\cjg{\widehat{u}_2}(k_2)...\cjg{\widehat{u}_6}(k_6)
	\end{equation*}
	where we recall that $\Psi^{(0)}_{2s}$ and $\Psi^{(1)}_{2s}$ are defined in Definition~\ref{def Big Psi}.
	Secondly, we consider:
	\begin{equation*}
		\cN_s(u_1,...,u_6) := \cT_s(F(u_1),u_2,...,u_6)
	\end{equation*}
	Finally, for $N\in \N$, we define the truncated version of these maps as~:
	\begin{align*}
		\cT_{s,N}(u_1,...,u_6) &:= \cT_s(\pi_N u_1,...,\pi_N u_6), & \cM_{s,N}(u_1,...,u_6) &:= \cM_s(\pi_N u_1,...,\pi_N u_6)
	\end{align*}
	and,
	\begin{equation*}
		\cN_{s,N}(u_1,...,u_6) := \cT_s(F_N u_1,\pi_Nu_2,...,\pi_Nu_6) 
	\end{equation*}
	where we recall that $F$ and $F_N$ are the non-linearities given in Notation~\ref{notation non-linearities}.
\end{defn}

\begin{rem}
	 Let $N\in \N$ and $\cF_{s,N} \in \{ \cM_{s,N}, \cT_{s,N}, \cN_{s,N} \}$. Then, there exists a unique
	 \begin{equation*}
		(\Psi_{2s},G) \in \{ (\Psi^{(0)}_{2s},id) \}  \cup \{\Psi^{(1)}_{2s} \} \times \{ id, F_N\}
	\end{equation*}	such that:
	\begin{equation*}
		\cF_{s,N}(u_1,...,u_6)= \sum_{\linc} \Psi_{2s}(\vec{k}) \widehat{Gu}_1(k_1)\cjg{\widehat{u}_2}(k_2)...\cjg{\widehat{u}_6}(k_6)
	\end{equation*}
	with the $u_j \in C^\infty(\T)$ satisfying $supp(\widehat{u_j})\subset [-N,N]$. This concise formulation will enable us to prove energy estimates for the three functions $\cM_{s,N}, \cT_{s,N}$, and $\cN_{s,N}$ in a single step.
\end{rem}

\begin{defn}[Energy correction and modified energy]\label{def modified energy and energy correction}
	For $N \in \N$, we define the \textit{energy correction} $R_{s,N}$ as the function :
	\begin{equation}\label{RsN}
		R_{s,N}(u):= \frac{1}{6} \Re \hsp \cT_{s,N}(u)
	\end{equation}
	Moreover, we introduce the \textit{modified energy} $E_{s,N}$ as the function:
	\begin{equation}\label{EsN}
		E_{s,N}(u) := \frac{1}{2} \norm{\pi_N u}_{H^s(\T)}^2 + R_{s,N}(u),
	\end{equation}
\end{defn}

\begin{defn}[Modified energy derivative]\label{def modified energy derivative}
	For $N \in \N$, we define the \textit{modified energy derivative} as the function~:
	\begin{equation*}
		Q_{s,N}(u):= -\frac{1}{6} \Im \hsp \cM_{s,N}(u) + \Im \hsp \cN_{s,N}(u)
	\end{equation*}
\end{defn}

\begin{prop}[Poincaré-Dulac normal form reduction]\label{prop Poincaré Dulac} For all $N\in \N$, and every $u :\T \ra \C$, we have:
	\begin{equation*}
		\frac{d}{dt}E_{s,N}(\Phi^N(t)u)|_{t=0}  = Q_{s,N}(u)
	\end{equation*}
\end{prop}

For a proof of this proposition, we refer to the computations performed in~\cite{sun2023quasi}, Section~2.

\subsection{Exponential integrability and quantitative \texorpdfstring{$L^p$}{TEXT} estimates} Here, and throughout this paper, we fix $\sigma < s-\frac{1}{2}$ close enough to $s-\frac{1}{2}$. To recall this, we might sometimes use the notation $\sigma = s-\frac{1}{2}-$. Moreover, for $R>0$, we denote $B^{\hsig}_R$ the closed centered ball of radius  $R>0$ in $\hsigt$.

\begin{prop}\label{prop exp integrability Fs,N}
	Let $s>\frac{9}{10}$. Then, for $m>2$ large enough, we have that for every $R>0$, there exists a constant $C_{s,R}>0$, such that for every $\alpha \geq 1$, $N\in\N$, and $\cF_{s,N} \in \{ \cM_{s,N}, \cT_{s,N}, \cN_{s,N}\}$,
	\begin{equation*}
		\log \int \1_{B^{\hsig}_R}(u) e^{\alpha | \cF_{s,N}(u)|}d\mu_s   \leq C_{s,R} \hsp \alpha^m 
	\end{equation*}
	As a consequence, we have:
	\begin{equation}\label{exp integrability RsN and QsN}
		\log \int \1_{B^{\hsig}_R}(u) e^{\alpha | R_{s,N}(u)|}d\mu_s  +  \log \int \1_{B^{\hsig}_R}(u) e^{\alpha | Q_{s,N}(u)|}d\mu_s   \leq C_{s,R} \hsp \alpha^m 
	\end{equation}
\end{prop}

\begin{prop}\label{prop quantitative Lp estimate}
	Let $s>\frac{9}{10}$. For $\beta \in (0,1)$ close enough to 1, we have that for every $R>0$, there exists a constant $C_{s,R}>0$ such that:
	\begin{equation}\label{quantitative Lp estimate}
		\| \1_{B^{\hsig}_R}R_{s,N} \|_{L^p(d\mu_s)} + \| \1_{B^{\hsig}_R}Q_{s,N} \|_{L^p(d\mu_s)} \leq C_{s,R} p^\beta
	\end{equation}
	for every $N \in \N$ and $p \in [1,+\infty)$.
\end{prop}

\begin{rem}
	Here, we obtain~\eqref{quantitative Lp estimate} as a consequence of~\eqref{exp integrability RsN and QsN} (see below), and that last inequality will be established through the Boué-Dupuis formula. Analogous (to \eqref{quantitative Lp estimate}) quantitative estimates were obtained through Wiener Chaos in~\cite{sun2023quasi}.
\end{rem}

\begin{proof}[Proof of Proposition \ref{prop quantitative Lp estimate} assuming Proposition~\ref{prop exp integrability Fs,N}] Let $J_{s,N} \in \{R_{s,N}, Q_{s,N}\}$ and let us prove the estimate~\eqref{quantitative Lp estimate} for $J_{s,N}$. We denote $B_R:= B^{\hsig}_R$. By Cavalieri's principle,
	\begin{equation*}
		\norm{\1_{B_R} J_{s,N}}_{L^p(d\mu_s)}^p = p \int_0^\infty \lambda^{p-1} \mu_s(\1_{B_R} |J_{s,N}| > \lambda) d\lambda = p \int_0^\infty \lambda^{p-1} \mu_s(\1_{B_R} e^{\alpha |J_{s,N}|} > e^{\alpha \lambda}) d\lambda
	\end{equation*}	
	On the other hand, by the Markov inequality and Proposition~\ref{prop exp integrability Fs,N}, we have for all $\alpha \geq 1$,
	\begin{equation*}
		\mu_s(\1_{B_R} e^{\alpha |J_{s,N}|} > e^{\alpha \lambda}) \leq e^{- \alpha \lambda} \int \1_{B_R} e^{\alpha |J_{s,N}|} d\mu_s \leq e^{C \alpha^m}e^{- \alpha \lambda}
	\end{equation*}
	with $C = C_{s,R} > 0$ a constant depending on $s$ and $R$, and $m>2$ large enough. Plugging this into the equality above yields:
	\begin{equation*}
		\norm{\1_{B_R} J_{s,N}}_{L^p(d\mu_s)}^p \leq p e^{C \alpha^m} \int_0^\infty \lambda^{p-1} e^{-\alpha \lambda} d\lambda
	\end{equation*}
	and by performing $\lfloor p \rfloor -1$ integration by parts,
	\begin{equation*}
		\begin{split}
			\int_0^\infty \lambda^{p-1} e^{-\alpha \lambda} d\lambda &= \alpha^{1-\lfloor p \rfloor} (p-1)(p-2)...(p-\lfloor p \rfloor+1) \int_0^\infty \lambda^{p-\lfloor p \rfloor}e^{-\alpha \lambda} d\lambda \\
			& \leq \alpha^{1-\lfloor p \rfloor} p^{\lfloor p \rfloor -1} \int (1+\lambda) e^{- \lambda} d\lambda 
		\end{split}
	\end{equation*}
	so, setting $C_0 := \int ( 1+ \lambda) e^{- \lambda} d\lambda$, we deduce that:
	\begin{equation*}
		\norm{\1_{B_R} J_{s,N}}_{L^p(d\mu_s)}^p \leq C_0 e^{C \alpha^m} \alpha^{1-\lfloor p \rfloor} p^{\lfloor p \rfloor} \leq C_0 p^p e^{C \alpha^m} \alpha^{2-p}
	\end{equation*}
	This inequality is true for all $\alpha \geq 1$. Then, optimizing in $\alpha$ the (RHS), we find that for $\alpha = p^{\frac{1}{m}}$ (which is indeed greater than $1$),
	\begin{equation*}
		\norm{\1_{B_R} J_{s,N}}_{L^p(d\mu_s)}^p \leq C_0 p^p e^{Cp} p^{\frac{2}{m}-\frac{p}{m}} 
	\end{equation*}
	which implies 
	\begin{equation*}
		\norm{\1_{B_R} J_{s,N}}_{L^p(d\mu_s)} \leq C_0^{\frac{1}{p}}e^C p^{1-\frac{1}{m} }p^{\frac{2}{pm}} \leq C  p^{1-\frac{1}{m}}
	\end{equation*}
	where $C>0$ has changed but still depends only on $s$ and $R$. This completes the proof.
\end{proof}

%% file: truncated_transport.tex
In this section, we prove that $\mu_s$ is quasi-invariant along the truncated flow, providing an explicit formula for the Radon-Nikodym derivative. Essentially, the finite-dimensional nature of \eqref{truncated equation} will allow us to make rigorous the formal computations that lead to~\eqref{formal transport of mus} in the introduction. Even though the results of this section are known, see for example the seminal paper~\cite{tzvetkov2015quasiinvariant}, we provide the proofs that lead to~\eqref{truncated trsprt of mus} and~\eqref{truncated trsprt of rhosN} for completeness.  Here, we fix $s \in \R$ and $\sigma < s-\frac{1}{2}$. For any topological space $X$, we denote $\cB(X)$ the associated Borel $\sigma$-algebra. For any $N \in \N$, recall that we denote:
	\begin{align*}
	E_N &:= \pi_N \hsigt, & E_N^{\perp} := \pi_N^{\perp}\hsigt = (\textnormal{Id} -\pi_N) \hsigt
\end{align*}
Moreover, through the homeomorphism:
\begin{equation}\label{identification hsig with E_N x E_N perp}
	\begin{split}
		E_N &\times E_N^\perp \lra \hsigt \\
		 (w&,v) \longmapsto w+v
	\end{split}
\end{equation}
we identify $\hsigt$ with the product space $E_N \times E_N^\perp$. Therefore, we can write:
\begin{equation*}
	\cB(\hsigt) = \cB(E_N \times E_N^\perp) = \cB(E_N) \otimes \cB(E_N^\perp)
\end{equation*}
where the last equality follows from the fact that both $E_N$ and $E_N^\perp$ are separable topological spaces. \\

Furthermore, we consider on $E_N$ the Lebesgue measure, denoted $\cL^{2N+1}$, associated to the orthonormal basis $(e^{inx})_{|n| \leq N}$. More precisely, we define $\cL^{2N+1}$ as the pushforward measure of the Lebesgue measure on $\C^{2N+1}$ through the map:
\begin{equation}\label{identification C^2N+1 with E_N}
	\begin{split}
		&\C^{2N+1} \lra E_N \\
		(u_{-N},&...,u_N) \longmapsto \sum_{|n|\leq N} u_n e^{inx} 
	\end{split}
\end{equation} 
Finally, we recall that the Gaussian measure $\mu_s$ is obtained as the law of the random variable:
\begin{equation*}
	\begin{split}
	\phi : (\Omega,&\cA,\P) \lra (\hsig , \cB(\hsig)) \\
	       & \omega \longmapsto \sum_{n \in \Z} \frac{g_n(\omega)}{\langle n\rangle^s} e^{inx}
	 \end{split}
\end{equation*}
where $(g_n)_{n \in \Z}$ is a sequence of independent standard complex Gaussian random variables on the probability space $(\Omega,\cF,\P)$, and where the sum converges in $L^2(\Omega, \hsig)$.
\begin{lem}\label{lem mus = musN x musN perp}
	Let $N \in \N$. We can decompose the Gaussian measure $\mu_s$ as:
	\begin{equation}\label{mus = musN x musN perp}
		\mu_s = \mu_{s,N} \otimes \mu_{s,N}^\perp
	\end{equation}
	where
	\begin{align*}
		\mu_{s,N} &= (\pi_N)_\# \mu_s, & & \textnormal{and}&  \mu_{s,N}^\perp &= (\pi_N^\perp)_\# \mu_s 
	\end{align*}
	Moreover, we have:
	\begin{equation}\label{formula for musN}
		\mu_{s,N} = \frac{1}{Z_N} e^{-\frac{1}{2}\norm{\pi_N u}_{H^s}^2} d\cL^{2N+1}
	\end{equation}
	where $Z_N$ is the normalizing constant given by $Z_N = \prod_{|n|\leq N} \frac{2\pi}{\langle n \rangle^{2s}}$.
\end{lem}

\begin{proof}
	First, we prove~\eqref{mus = musN x musN perp}. From the uniqueness of the product measure, it is sufficient to prove that for every $A \times B \in E_N \times E_N^\perp$, we have:
	\begin{equation*}
		\mu_s(A \times B) = \mu_{s,N}(A) \mu_{s,N}^\perp(B)
	\end{equation*} 
	On the other hand, from the independence of the $\sigma$-algebra generated by the Gaussian $(g_n)_{|n|\leq N}$ with the $\sigma$-algebra generated by the Gaussian $(g_n)_{|n| > N}$, we have
	\begin{equation*}
		\mu_s(A \times B) = \P( \phi \in A \times B) = \P( \{\pi_N \phi \in A \} \cap \{\pi_N^\perp \phi \in B \} ) = \P(\pi_N\phi \in A) \P(\pi_N^\perp \phi \in B)
	\end{equation*}
	Then, from the definition of a pushforward measure:
	\begin{equation*}
		\P(\pi_N \phi \in A) = \P(\phi \in \pi_N^{-1}(A)) = \mu_s(\pi_N^{-1}(A)) = (\pi_N)_\# \mu_s(A) = \mu_{s,N}(A)
	\end{equation*}  
	and similarly, $\P(\pi_N^\perp \phi \in B) = (\pi_N^\perp)_\# \mu_s(B) = \mu_{s,N}^\perp(B)$. This concludes the proof of~\eqref{mus = musN x musN perp}. \\
	
	We now proceed to the proof of~\eqref{formula for musN}. We denote $dz$ and $d\vec{z}$ the Lebesgue measure on $\C$ and $\C^{2N+1}$ respectively. We recall that the law of the $g_n$ is $\frac{1}{2\pi} e^{-\frac{1}{2}|z|^2}dz$. Let $f$ be a non-negative measurable function on $E_N$. Then, from the independence of the $g_n$, and through the identification~\eqref{identification C^2N+1 with E_N}, we have:
	\begin{multline*}
		\int_{E_N} f(u) d\mu_{s,N} = \int_\Omega f(\sum_{|n|\leq N} \frac{g_n}{\langle n\rangle^s} e^{inx}) d\P(\omega) = \int_{\C^{2N+1}} f(\frac{z_{-N}}{\langle -N \rangle^{s}},...,\frac{z_N}{\langle N \rangle^{s}}) d\P_{(g_{-N},...,g_N)} \\
		=\int_{\C^{2N+1}} f(\frac{z_{-N}}{\langle -N \rangle^{s}},...,\frac{z_N}{\langle N \rangle^{s}}) \prod_{|n| \leq N} \frac{e^{-\frac{1}{2}|z_n|^2}}{2\pi} dz_{-N} ... dz_N \\ = \frac{1}{Z_N}\int_{\C^{2N+1}} f(\vec{z}) e^{-\frac{1}{2} \sum_{|n| \leq N} \langle n \rangle^{2s} |z_n|^2}  d\vec{z} 
		= \frac{1}{Z_N}\int_{E_N} f(u) e^{-\frac{1}{2} \norm{\pi_N u}_{H^s}^2} d\cL^{2N+1}
	\end{multline*}
	This equality, satisfied by all non-negative measurable function $f$ on $E_N$, implies the desired formula~\eqref{formula for musN}.
\end{proof}

\begin{lem}\label{lem pushforward of a density measure} Let $(X,\cA)$ be a measurable space and $f: X \ra [0,+\infty]$ a non-negative measurable function. Let $Y$ a set and $\chi : X \ra Y$ a bijection. Then,
	\begin{equation*}
		\chi_\# (f d\mu) = (f \circ \chi^{-1}) \chi_\# d\mu
	\end{equation*}
	where $Y$ is equipped with the pushforward $\sigma$-algebra.
\end{lem}

\begin{proof}
	Let $\psi$ a non-negative measurable function on $Y$. Then,
	\begin{multline*}
		\int_Y \psi(v) [(f \circ \chi^{-1}) \chi_\# d\mu](v) = \int_Y \psi(v) f(\chi^{-1}(v)) [\chi_\#d\mu](v) = \int_X \psi(\chi(u)) f(u) d\mu(u) \\ = \int_Y \psi(v)[\chi_\#(fd\mu)](v)
	\end{multline*}
	This indeed means that $\chi_\# (f d\mu) = (f \circ \chi^{-1}) \chi_\# d\mu$.
\end{proof}

\begin{prop}[Transport of Gaussian measures by the truncated flow]\label{prop trsp truncated measures}
	Let $s>0$, $N \in \N$, and $t\in \R$. Then,
	\begin{equation}\label{truncated trsprt of mus}
		\Phi^N(t)_\# \mu_s = \exp(-(\frac{1}{2}\norm{\pi_N \Phi^N(-t)u}_{H^s}^2-\frac{1}{2}\norm{\pi_N u}_{H^s}^2)) d\mu_s
	\end{equation}
	Moreover, if we set
	\begin{equation}\label{weighted gausssian measures}
		\rho_{s,N} := e^{-R_{s,N}(u)} d\mu_s,
	\end{equation}
	called the weighted Gaussian measure, we have:
	\begin{equation}\label{truncated trsprt of rhosN}
		\begin{split}
			\Phi^N(t)_\# \rho_{s,N} &= \exp(-(E_{s,N}(\Phi^N(-t)u) - E_{s,N}(u)))d\rho_{s,N} \\
			& = \exp\big( - \int_0^{-t} Q_{s,N}(\Phi^N(\tau)u) d\tau \big)d\rho_{s,N}
		\end{split}
	\end{equation}
\end{prop}

\begin{proof}
	From the identification~\eqref{identification hsig with E_N x E_N perp}, the factorization of the truncated flow in Proposition~\ref{prop factorization of the truncated flow}, and Lemma~\ref{lem mus = musN x musN perp}, we can write:
	\begin{equation}\label{PhiN push on EN x EN perp}
		\Phi^N(t)_\# \mu_s = (\tld{\Phi}^N(t),e^{it\p_x^2})_\# (\mu_{s,N} \otimes \mu_{s,N}^\perp) = \tld{\Phi}^N(t)_\# \mu_{s,N} \otimes (e^{it\p_x^2})_\# \mu_{s,N}^\perp
	\end{equation}
	On the one hand, by Lemma~\ref{lem pushforward of a density measure}, 
	\begin{equation*}
		 \tld{\Phi}^N(t)_\# \mu_{s,N} = \tld{\Phi}^N(t)_\#(\frac{1}{Z_N}e^{-\frac{1}{2}\norm{\pi_N u}_{H^s}^2}d\cL^{2N+1}) =  \frac{1}{Z_N}e^{-\frac{1}{2}\norm{\pi_N \tld{\Phi}^N(-t)u}_{H^s}^2}d(\tld{\Phi}^N(t)_\#\cL^{2N+1})
	\end{equation*}
	Since the equation~\eqref{truncated equation} satisfied by $\tld{\Phi}^N(t)$ is Hamiltonian (see for example the appendix of~\cite{knez24transportlowregularitygaussian}), we have, by Liouville's theorem, $\tld{\Phi}^N(t)_\#\cL^{2N+1} = \cL^{2N+1}$. Moreover, we have
		$\tld{\Phi}^N(-t)=\pi_N\tld{\Phi}^N(-t)=\pi_N \Phi^N(-t)$.
	Thus, continuing the equality above, we obtain:
	\begin{equation}\label{tld Phi pushforward musN}
		\begin{split}
			\tld{\Phi}^N(t)_\# \mu_{s,N} &= e^{-(\frac{1}{2}\norm{\pi_N \Phi^N(-t)u}_{H^s}^2-\frac{1}{2}\norm{\pi_N u}_{H^s}^2)} \frac{e^{-\frac{1}{2}\norm{\pi_N u}_{H^s}^2}}{Z_N}d\cL^{2N+1} \\ 
			&= e^{-(\frac{1}{2}\norm{\pi_N \Phi^N(-t)u}_{H^s}^2-\frac{1}{2}\norm{\pi_N u}_{H^s}^2)} d\mu_{s,N}
		\end{split}
	\end{equation}
	On the other hand, since $e^{it\p_x^2}$ acts continuously on $L^2(\Omega,\hsig)$, we have:
	\begin{equation*}
		e^{it\p_x^2}\big(\sum_{|n|>N}\frac{g_n}{\langle n \rangle^s}e^{inx}\big) = e^{it\p_x^2}\big(\lim_{M \ra \infty}\sum_{M\geq |n|>N}\frac{g_n}{\langle n \rangle^s}e^{inx}\big) = \sum_{|n|> N} \frac{e^{-itn^2}g_n}{\langle n \rangle^s}e^{inx}
	\end{equation*}
	where the limit takes place in the space $L^2(\Omega,\hsig)$. Consequently, since $(e^{-itn^2}g_n)$ is still a sequence of independent standard complex Gaussian, we have:
	\begin{equation}\label{S(t) pushforward musN perp}
		(e^{it\p_x^2})_\# \mu_{s,N}^\perp = (e^{it\p_x^2})_\# \big(\sum_{|n|>N}\frac{g_n}{\langle n \rangle^s}e^{inx}\big)_\# \P = \big(\sum_{|n|>N}\frac{e^{-itn^2}g_n}{\langle n \rangle^s}e^{inx}\big)_\# \P = \mu_{s,N}^\perp
	\end{equation}
	Now, combining~\eqref{PhiN push on EN x EN perp},~\eqref{tld Phi pushforward musN}, and~\eqref{S(t) pushforward musN perp}, leads to:
	\begin{equation*}
		\Phi^N(t)_\# \mu_s = [e^{-(\frac{1}{2}\norm{\pi_N \Phi^N(-t)u}_{H^s}^2-\frac{1}{2}\norm{\pi_N u}_{H^s}^2)} d\mu_{s,N}] \otimes d\mu_{s,N}^\perp
	\end{equation*}
	and by Fubini-Tonelli, this writes as
	\begin{equation*}
		\Phi^N(t)_\# \mu_s = e^{-(\frac{1}{2}\norm{\pi_N \Phi^N(-t)u}_{H^s}^2-\frac{1}{2}\norm{\pi_N u}_{H^s}^2)} [d\mu_{s,N} \otimes d\mu_{s,N}^\perp] = e^{-(\frac{1}{2}\norm{\pi_N \Phi^N(-t)u}_{H^s}^2-\frac{1}{2}\norm{\pi_N u}_{H^s}^2)} d\mu_s
	\end{equation*}
	Hence, we have established~\eqref{truncated trsprt of mus}. As a consequence, let us now proceed to prove~\eqref{truncated trsprt of rhosN}. Using Lemma~\ref{lem pushforward of a density measure} again, then~\eqref{truncated trsprt of mus}, and then the definition of the modified energy~\eqref{EsN}, yields:
	\begin{equation*}
		\begin{split}
			\Phi^N(t)_\# \rho_{s,N} = e^{-R_{s,N}(\Phi^N(-t)u)}\Phi^N(t)_\# \mu_s & = e^{-(E_{s,N}(\Phi^N(-t)u) - E_{s,N}(u))}e^{-R_{s,N}(u)}d\mu_s \\
			& = e^{-(E_{s,N}(\Phi^N(-t)u) - E_{s,N}(u))} d\rho_{s,N}
		\end{split}
	\end{equation*}
	Finally, using the fundamental theorem of calculus and the identity from Proposition~\ref{prop Poincaré Dulac}, we have the rewriting:
	\begin{equation*}
		E_{s,N}(\Phi^N(-t)u) - E_{s,N}(u) = \int_0^{-t} \frac{d}{d\tau} E_{s,N}(\Phi^N(\tau)u) d\tau = \int_0^{-t} Q_{s,N}(\Phi^N(\tau)u) d\tau
	\end{equation*}
	where we have used the additivity of $\Phi^N(t)$ in order to write:
	\begin{center}
		$\frac{d}{d\tau} E_{s,N}(\Phi^N(\tau)u) = \frac{d}{dt}E_{s,N}(\Phi^N(t+\tau)u)|_{t=0} = \frac{d}{dt}E_{s,N}(\Phi^N(t)(\Phi^N(\tau)u))|_{t=0} = Q_{s,N}(\Phi^N(\tau)u)$
	\end{center}
	This completes the proof of Proposition~\ref{prop trsp truncated measures}.
\end{proof}

%% file: proof_quasi_inv.tex
This section is dedicated to the proof of Theorem~\ref{thm quasi-inv}. We recall that we fix the parameter $\sigma<s-\frac{1}{2}$ close to $s-\frac{1}{2}$, and see the Gaussian measures $\mu_s$ as probability measures on $(\hsigt,\cB(\hsigt))$. For transferring the quasi-invariance for the truncated flow to the original flow, we combine the approximation properties of $\Phi(t)$ by $\Phi^N(t)$ with the uniform (in $N \in \N$) energy estimates from Propositions~\ref{prop exp integrability Fs,N} and~\ref{prop quantitative Lp estimate}. Furthermore, we work intermediately with the weighted Gaussian measures defined in~\eqref{weighted gausssian measures}.

\begin{lem}\label{lem quantitative ineq quasi-inv WGM}
	Let $s>\frac{9}{10}$, $R>0$, and $T>0$. There exists a constant $C_{s,R,T}>0$, such that for every compact set $K \subset B^{\hsig}_R$, there exists an integer $N_{K,T}$ such that:
	\begin{equation}\label{quantitative ineq quasi-inv WGM}
		\rho_{s,N}(\Phi^N(t)K) \leq C_{s,R,T} \hsp \rho_{s,N}(K)^{\frac{1}{2}}
	\end{equation} 
	for every $N \geq N_{K,T}$, and every $|t| \leq T$.
\end{lem}

\begin{proof} First of all, thanks to~\eqref{PhiN K included PhiK + B} and~\eqref{Phi B included B}, there exists $\Lambda=\Lambda_{R,T}>0$ and $N_{K,T} \in \N$ such that for every $N \geq N_{K,T}$, and every $|t| \leq T$ :
	\begin{equation*}
		\Phi_N(t)(K) \subset \Phi(t)(K) + B^{\hsig}_1 \subset B^{\hsig}_{\Lambda}
	\end{equation*}
	Secondly, from~\eqref{truncated trsprt of rhosN}, we have:
	\begin{equation*}
		\rho_{s,N}(\Phi^N(t)K) = \Phi^N(-t)_\# \rho_{s,N}(K) = \int_K \exp\big( - \int_0^{t} Q_{s,N}(\Phi^N(\tau)u) d\tau \big)d\rho_{s,N}
	\end{equation*}
	Consequently,
	\begin{equation*}
		\frac{d}{dt}\rho_{s,N}(\Phi^N(t)K) = - \int_K Q_{s,N}(\Phi^N(t)u) d(\Phi^N(-t)_\# \rho_{s,N}) = -\int_{\Phi^N(t)K}Q_{s,N}(u) d\rho_{s,N}
	\end{equation*}
	and it follows that for every $p \in [1,\infty)$:
	\begin{equation*}
		|\frac{d}{dt}\rho_{s,N}(\Phi^N(t)K)| \leq \| \1_{B^{\hsig}_{\Lambda}} Q_{s,N}\|_{L^p(d\rho_{s,N})} \rho_{s,N}(\Phi^N(t)K)^{1 - \frac{1}{p}}
	\end{equation*}
	Furthermore, the energy estimates \eqref{exp integrability RsN and QsN} and~\eqref{quantitative Lp estimate} imply by the Cauchy-Schwarz inequality that:
	\begin{equation*}
		\| \1_{B^{\hsig}_{\Lambda}} Q_{s,N} \|_{L^p(d\rho_{s,N})} \leq C_{s,R,T} \| \1_{B^{\hsig}_{\Lambda}} Q_{s,N} \|_{L^{2p}(d\mu_s)}  \leq C_{s,R,T} p^\beta
	\end{equation*} 
	(where the constant $C_{s,R,T}$ changes from an inequality to another); and plugging this in the inequality above yields:
	\begin{equation*}
		|\frac{d}{dt}\rho_{s,N}(\Phi^N(t)K)| \leq C_{s,R,T}  p^\beta \rho_{s,N}(\Phi^N(t)K)^{1 - \frac{1}{p}}
	\end{equation*}
	In other words, the function $F(t) := \rho_{s,N}(\Phi^N(t)K)$ satisfies the differential inequality:
	\begin{equation}\label{differential ineq}
		|F'(t)| \leq C_{s,R,T}  p^\beta F(t)^{1-\frac{1}{p}}
	\end{equation}
	and integrating this leads to:
	\begin{equation}\label{integration of the differential ineq}
		F(t) \leq (F(0)^{\frac{1}{p}} + C_{s,R,T}p^{\beta-1})^p
	\end{equation} 
	Our goal is to prove that $F(t)\leq C_{s,R,T} F(0)^{\frac{1}{2}}$; it is true when $F(0)=0$ thanks to~\eqref{truncated trsprt of rhosN}, so we assume that $F(0) > 0$. Then, we can rewrite~\eqref{integration of the differential ineq} as:
	\begin{equation*}
		F(t) \leq F(0)(1+F(0)^{-\frac{1}{p}} C_{s,R,T} p^{\beta-1})^p = F(0)\exp(p \log(1+F(0)^{-\frac{1}{p}} C_{s,R,T} p^{\beta-1}))
	\end{equation*}
	and using the inequalities:
	\begin{align*}
		\log(1+x)& \leq x & &\textnormal{and} & C_{s,R,T} p^\beta \leq \alpha p + C_{s,R,T,\alpha} \hspace{0.3cm} \textnormal{(for any fixed $\alpha>0$, and every $p\geq 1$)}
	\end{align*} we obtain: 
	\begin{equation}\label{ineq on F(t)}
		F(t) \leq F(0)\exp(C_{s,R,T}p^\beta F(0)^{-\frac{1}{p}}) \leq F(0)\exp(\alpha p F(0)^{-\frac{1}{p}} + C_{s,R,T,\alpha} F(0)^{-\frac{1}{p}}) 
	\end{equation}
	Now, if $\log(\frac{1}{F(0)})<1$, that is $F(0)>e^{-1}$, we simply take $p=1$ and $\alpha=1$, so that:
	\begin{equation*}
		F(t) \leq C_{s,R,T}F(0)
	\end{equation*}
	and since from~\eqref{exp integrability RsN and QsN}, we have $F(0) \leq C_{s,R}$, we can write (with again $C_{s,R,T}$ changing):
	\begin{equation*}
		F(t) \leq C_{s,R,T}F(0)^{\frac{1}{2}}F(0)^{\frac{1}{2}} \leq C_{s,R,T}F(0)^{\frac{1}{2}}
	\end{equation*}
	which is the desired inequality. Next, if $\log(\frac{1}{F(0)}) \geq 1$, then we choose $p = \log(\frac{1}{F(0)})$ and $\alpha = e^{-1}/2$, so that $\alpha p F(0)^{-\frac{1}{p}} = -\frac{1}{2} \log F(0)$; and thus we get from~\eqref{ineq on F(t)} that:
	\begin{equation*}
		F(t) \leq F(0) \exp(-\frac{1}{2}\log F(0)) \exp(C_{s,R,T}) \leq C_{s,R,T} F(0)^{\frac{1}{2}}
	\end{equation*}
 	which is the desired inequality. The proof is now complete.
\end{proof}

\begin{proof}[Proof of Theorem \ref{thm quasi-inv}] Let $T>0$ and $|t|\leq T$. We prove that $\Phi_\#(t)\mu_s$ is absolutely continuous with respect to $\mu_s$. Since $T>0$ is arbitrary, it will indeed imply Theorem~\ref{thm quasi-inv}.  \\
	
	We introduce $\sigma' \in (\sigma, s-\frac{1}{2})$. Since $H^{\sigma'}(\T)$ is of $\mu_s$-full measure, it suffices to prove that for any $A \in \cB(H^{\sigma'}(\T))$, we have:
	\begin{equation*}
		\mu_s(A) = 0 \implies \mu_s(\Phi(t) A) =0
	\end{equation*}
	Let us then fix $A \in \cB(H^{\sigma'}(\T))$ such that $\mu_s(A)=0$. First, we write:
	\begin{equation*}
		\mu_s(\Phi(t)A) = \mu_s(\bigcup_{R>0} \Phi(t)(A \cap B^{H^{\sigma'}}_R))= \lim_{R \ra \infty} \mu_s(\Phi(t)(A \cap B^{H^{\sigma'}}_R))
	\end{equation*}
	so that, denoting $A_R := A \cap B^{H^{\sigma'}}_R \subset B^{H^{\sigma'}}_R$, it is enough to prove that $\mu_s(\Phi(t)A_R)=0$. Moreover, by inner regularity of $\mu_s$, we have:
	\begin{equation*}
		\mu_s(\Phi(t)A_R) := \sup \{ \mu_s(K) \hsp : \hsp K \subset \Phi(t)A_R, \hsp K \hsp \textnormal{compact set in $H^{\sigma'}$}\}
	\end{equation*}
	Let us then fix $K \subset \Phi(t)A_R$ a compact set in $H^{\sigma'}$ and prove that $\mu_s(K)=0$; it will imply that $\mu_s(\Phi(t) A_R)=0$, thereby completing the proof. We consider the set:
	\begin{equation*}
		\tld{K} := \Phi(-t) K = \Phi(t)^{-1}(K) 
	\end{equation*}
	which is a compact set in $H^{\sigma'}$ (by continuity of the flow), and satisfies $\tld{K} \subset A_R \subset B^{H^{\sigma'}}_R$. Thus, we can apply~\eqref{PhiK included PhiN(K+B)}; for $\eps>0$, we consider $N_{\eps}$ such that for every $N \geq N_{\eps}$:
	\begin{equation*}
		\Phi(t) \tld{K} \subset \Phi^N(t)(\tld{K} + B^{H^{\sigma'}}_\eps)
	\end{equation*}
	Since $\sigma'>\sigma$, the ball $B^{H^{\sigma'}}_\eps$ is compact in $\hsig$; therefore, $\tld{K}+ B^{H^{\sigma'}}_\eps$ is compact in $\hsig$, and included (for $\eps$ small enough) in the ball $B_{2R}^{H^{\sigma'}} \subset B_{2R}^{H^{\sigma}}$. Thanks to~\eqref{PhiN K included PhiK + B} and~\eqref{Phi B included B}, we deduce that  there exists a constant $\Lambda=\Lambda_{R,T}>0$ such that for $N \geq N_{\eps}$ (with possibly $N_\eps$ even larger):
	\begin{equation*}
			\Phi(t) \tld{K} \subset \Phi^N(t)(\tld{K} + B^{H^{\sigma'}}_\eps) \subset B^{H^{\sigma}}_{\Lambda}
	\end{equation*}
	Hence, from this,~\eqref{exp integrability RsN and QsN},~\eqref{quantitative ineq quasi-inv WGM}, and the Cauchy-Schwarz inequality, we are now able to write, with $N$ large enough:
	\begin{multline*}
		\mu_s(K) = \mu_s(\Phi(t)\tld{K}) \leq \mu_s(\Phi^N(t)(\tld{K} + B^{H^{\sigma'}}_\eps)) \leq C_{s,R,T} \rho_{s,N}(\Phi^N(t)(\tld{K} + B^{H^{\sigma'}}_\eps))^{\frac{1}{2}} \\
		\leq C_{s,R,T} \rho_{s,N}(\tld{K} + B^{H^{\sigma'}}_\eps)^{\frac{1}{4}} \leq  C_{s,R,T} \mu_s(\tld{K} + B^{H^{\sigma'}}_\eps)^{\frac{1}{8}}
	\end{multline*}	
	with $C_{s,R,T}$ varying from one inequality to another. Passing to the limit $\eps \ra 0$ leads to (since $\tld{K}$ is closed in $H^{\sigma'}$):
	\begin{equation*}
		\mu_s(K) \leq C_{s,R,T} \mu_s\big(\bigcap_{\eps>0}\tld{K} + B^{H^{\sigma'}}_\eps\big)^{\frac{1}{8}} = C_{s,R,T} \mu_s(\tld{K})^{\frac{1}{8}}
	\end{equation*}
	This concludes the proof because $\mu_s(\tld{K}) \leq \mu_s(A) = 0$.
\end{proof}

The remaining part of the paper is dedicated to the proof of the Exponential integrability properties in Proposition~\ref{prop exp integrability Fs,N}; and also, in Appendix~\ref{appendix construction of a global flow and approximation}, to the construction of a global flow and the approximation properties of $\Phi(t)$ by $\Phi^N(t)$.

%% file: lower_bound_res_func_and_counting_bounds.tex
This section and the next one are intended to prepare the analysis that we will perform to prove Proposition~\ref{prop exp integrability Fs,N} in the central Section~\ref{section Exponential integrability}. In the present section, we establish a lower bound on the resonant function $\Omega$ (see Definition~\ref{def psi and omega}) and counting bounds. 

\subsubsection{A lower bound on the resonant function}

\begin{lem}\label{lem lower bound Omega when three high and three low freq}
	There exists a constant $c_0>0$ such that for every $(k_1,...,k_6) \in \Z^6 \setminus \{0\}$ such that:
	\begin{align}\label{conditions linc and k4 ll k3}
		k_1 -k_2 +k_3 -k_4 +k_5 -k_6 & = 0 & \textnormal{and}& & |k_{(4)}| & \leq \frac{1}{10} |k_{(3)}| 
	\end{align} 
	we have: 
	\begin{equation}\label{lower bound Omg three high three low}
		|\Omgk| \geq c_0 |k_{(1)}||k_{(3)}|
	\end{equation}
	Moreover, the analogous statement holds replacing $\Omega$ by its degenerated version $\Omega^\downarrow$.
\end{lem}

\begin{rem}\label{rem Omgk = 0 implies N4 sim N3} 
	As a consequence, note that if $\Omgk = 0$ and $\linc$ and $k_{(3)} \neq 0$, then we have $|k_{(3)}| < 10 |k_{(4)}|$.
\end{rem}
\begin{proof}
	Once again, we only do the proof of Lemma~\ref{lem lower bound Omega when three high and three low freq} for $\Omega$, as the proof for $\Omega^\downarrow$ is identical. First, we assume that $k_{(1)}$ and $k_{(2)}$ have different signatures. Indeed, if we assume otherwise, then the proof of~\eqref{lower bound Omg three high three low} is easier, writing:
	\begin{equation*}
		|\Omgk| \geq k_{(1)}^2 + k_{(2)}^2 - k_{(3)}^2 - 3 k_{(4)}^2 \geq k_{(1)}^2 + k_{(2)}^2 - (1 + \frac{3}{100})k_{(3)}^2 \geq \frac{1}{2}k_{(1)}^2
	\end{equation*}
	Secondly, we observe that it is necessary that $k_{(1)}$ and $k_{(2)}$ have the same sign. If by contradiction we assume that $k_{(1)}$ and $k_{(2)}$ have opposite signs, then $|k_{(1)}- k_{(2)}| = |k_{(1)}|~+~| k_{(2)}| $, and we would deduce that 
	\begin{equation*}
		\linc \implies 0 \geq |k_{(1)}- k_{(2)}| - |k_{(3)}| - 3 |k_{(4)}| \geq |k_{(1)}| + |k_{(2)}| - (1+\frac{3}{10}) |k_{(3)}|
	\end{equation*} 
	and therefore we would obtain $0 \geq \frac{1}{2} |k_{(1)}|$, which is absurd beacause $|k_{(1)}| > 0$ since by assumption $\vec{k}\neq0$. To sum up, we are left with the case where $k_{(1)}$ and $k_{(2)}$ have different signatures and same sign. In this situation, we have:
	\begin{equation*}
		\begin{split}
			|\Omgk| \geq |k_{(1)}^2-k_{(2)}^2| -k_{(3)}^2 -3k_{(4)}^2 &\geq  |k_{(1)}-k_{(2)}||k_{(1)}+k_{(2)}| -(1 + \frac{3}{100})k_{(3)}^2 \\
			& = |k_{(1)}-k_{(2)}|(|k_{(1)}|+|k_{(2)}|) -(1 + \frac{3}{100})k_{(3)}^2
		\end{split}
	\end{equation*}
	Moreover, from the constraint $\linc$, we have $ |k_{(1)}-k_{(2)}| \geq |k_{(3)}|-3|k_{(4)}|$, which implies by assumption that $|k_{(1)}-k_{(2)}| \geq \frac{7}{10}|k_{(3)}|$. Plugging this in the inequaltiy above leads to
	\begin{equation*}
		\begin{split}
			|\Omgk| \geq \frac{7}{10}|k_{(3)}| (|k_{(1)}|+|k_{(2)}|) -(1 + \frac{3}{100})k_{(3)}^2 &\geq \frac{1}{10}|k_{(3)}| |k_{(1)}| + \frac{13}{10}k_{(3)}^2 -(1 + \frac{3}{100})k_{(3)}^2 \\
			&\geq \frac{1}{10}|k_{(3)}| |k_{(1)}|
		\end{split}
	\end{equation*}
	This completes the proof of Lemma~\ref{lem lower bound Omega when three high and three low freq}.
\end{proof}

\subsubsection{Counting bounds}  

\begin{lem}\label{lem counting bound}
There exists a constant $C>0$ such that for every dyadic integers $N_1,...,N_m$, every $\eps_1,...,\eps_m \in \{-1,+1\}$, and every $l \in \Z$,
	\begin{equation}\label{elementary counting bound}
		\sum_{k_1,...,k_m \in \Z} \1_{\eps_1 k_1 +\eps_2 k_2 +...+\eps_m k_m = l} \cdot \left( \prod_{j=1}^{m} \1_{|k_j| \sim N_j} \right) \leq C N_{(2)}N_{(3)}...N_{(m)}
	\end{equation}
\end{lem}

\begin{lem}\label{lem counting bound 3 integers}
	For any $\eps >0$ there exists $C>0$ such that for any $N_1,N_2,N_3$ positive integers and $l,q \in \Z$,
	\begin{equation}\label{counting boung 3 integers with different signatures}
		\begin{split}
			\# \{ (k_1,k_2,k_3)  \in \Z^3 : \hspace{0.1cm}  k_1-&k_2+k_3=l,\hspace{0.1cm} k_1^2-k_2^2+k_3^2=q, \\
			& k_1,k_3 \neq k_2, \hspace{0.1cm} |k_1|\leq N_1, |k_2|\leq N_2, \hspace{0.1cm} |k_3|\leq N_3\} \leq C N_{(2)}^{\eps}
		\end{split}
	\end{equation}
	and,
		\begin{equation}\label{counting bound 3 integers with same signature}
		\begin{split}
			\# \{ (k_1,k_2,k_3)  \in \Z^3 : \hspace{0.1cm}  k_1+&k_2+k_3=l,\hspace{0.1cm} k_1^2+k_2^2+k_3^2=q, \\
			& \hspace{0.1cm} |k_1|\leq N_1, |k_2|\leq N_2, \hspace{0.1cm} |k_3|\leq N_3\} \leq C N_{(2)}^{\eps}
		\end{split}
	\end{equation}
	where $N_{(1)}\geq N_{(2)} \geq N_{(3)}$ is a non-increasing rearrangement of $N_1,N_2,N_3$.
\end{lem}

\begin{rem}
--We indicate that similar counting bounds can be found in \cite{deng2021optimal}, Lemma 3.4. \\ 
-- In this paper, the weaker bound $C N_{(1)}^\eps$ (instead of $C N_{(2)}^\eps$) will suffice. We stated Lemma~\ref{lem counting bound 3 integers} with the bound $C N_{(2)}^\eps$ as it may prove useful in other contexts. \\
-- The constraint $k_1,k_3 \neq k_2$ in \eqref{counting boung 3 integers with different signatures} is crucial. Indeed, without this constraint, the bound in \eqref{counting boung 3 integers with different signatures} does not hold anymore. This is due to the fact that if we remove this constraint, then for any $N_1,N_2,N_3$ such that $N_1=N_2$, we can take $l=N_3$ and $q=N_3^2$ so that the inclusion
		\begin{equation*}
			\begin{split}
				\{  (k_1,k_1, N_3) : \hsp k_1 \in \Z, \hsp |k_1|\leq N_1 \} & \subset \{ (k_1,k_2,k_3) \in \Z^3 : \hspace{0.1cm}  k_1-k_2+k_3=N_3, \\
				& \hspace{0.1cm} k_1^2-k_2^2+k_3^2=N_3^2,\hsp |k_1|\leq N_1, |k_2|\leq N_2, \hspace{0.1cm} |k_3|\leq N_3\}
			\end{split}
		\end{equation*}
		would imply that 
		\begin{equation*}
			\begin{split}
				\# \{ (k_1,k_2,k_3)  \in \Z^3 : \hspace{0.1cm}  k_1-k_2+k_3&=N_3,\hspace{0.1cm}  k_1^2-k_2^2+k_3^2=N_3^2,\\
				& |k_1|\leq N_1, |k_2|\leq N_2, \hspace{0.1cm} |k_3|\leq N_3\} \geq 2 N_1,
			\end{split}
		\end{equation*}
		and this lower bound is incompatible with the upper bound in~\eqref{counting boung 3 integers with different signatures}.
\end{rem}

\begin{proof}[Proof of Lemma \ref{lem counting bound 3 integers}]
	Fix $\eps > 0$. \\
	\textbullet Firstly, we establish the bound \eqref{counting boung 3 integers with different signatures}. Set 
	\begin{equation*}
		\begin{split}
			\cS := \{ (k_1,k_2,k_3) & \in \Z^3 : \hspace{0.1cm}  k_1-k_2+k_3=l,\hspace{0.1cm} k_1^2-k_2^2+k_3^2=q, \\
			& k_1,k_3 \neq k_2, \hspace{0.1cm} |k_1|\leq N_1, |k_2|\leq N_2, \hspace{0.1cm} |k_3|\leq N_3\}
		\end{split}
	\end{equation*}
	
	Fix $(k_1,k_2,k_3) \in \cS$. Replacing $k_2$ by $k_1+k_3-l$ in the equation $k_1^2-k_2^2+k_3^2=q$ we get that 
	\begin{equation*}
		k_1^2 - (k_1+k_3-l)^2 + k_3^2 = q
	\end{equation*}
	On the other hand,
	\begin{equation*}
		\begin{split}
			k_1^2 - (k_1+k_3-l)^2 + k_3^2 &= -2(k_1k_3-k_1l-k_3l) - l^2 \\
			& = -2(k_1-l)(k_3-l) + l^2
		\end{split}
	\end{equation*}
	Putting together these two equations we obtain 
	\begin{equation}\label{counting bound, substituting k2 in the equation}
		2(k_1-l)(k_3-l) = l^2 - q
	\end{equation}
	We deduce that $2(k_1-l)$ and $k_3-l$ are two divisors of $l^2-q$. Denote $\cD(l^2-q)$ the set of divisors of $l^2-q$. We just have shown that the map :
	\begin{equation*}
		\begin{split}
			\varphi : & \hspace{0.6cm}\cS \longrightarrow \cD(l^2-q) \times \cD(l^2-q) \\
			& (k_1,k_2,k_3) \longmapsto \left(2(k_1 -l), k_3-l \right)
		\end{split}
	\end{equation*}
	is well defined. It is also injective since if $(k_1,k_2,k_3) \in \cS$ and $(n_1,n_2,n_3)\in \cS$ are such that $2(k_1-l)=2(n_1-l)$ and $k_3-l=n_3-l$ then $k_1=n_1$ and $k_3=n_3$, and using now the linear constraint $k_1-k_2+k_3=l=n_1-n_2+n_3$ we also have $k_2=n_2$. Now the injectivity of $\varphi$ implies that 
	\begin{equation*}
		\# \cS \leq \left( \#\cD(l^2-q)\right)^2
	\end{equation*}
	
	By equation \eqref{counting bound, substituting k2 in the equation} we see that we cannot have $l^2-q=0$. Indeed, if that happens, then we would have $k_1 = l$ or $k_3=l$, but then the linear constraint $k_1-k_2+k_3 = l$ would imply $k_3=k_2$ or $k_1=k_2$, which is excluded in the definition of $\cS$. Thus, since $l^2-q \neq 0$, the standard \textit{divisor bound} says there exists $C>0$ only depending on $\eps$ such that
	\begin{equation*}
		\# \cD(l^2-q) \leq C|l^2-q|^{\frac{\eps}{20}}
	\end{equation*}
	
	Moreover we have $|l^2-q| \leq l^2+|q| \leq 9N_{(1)}^2$, where we recall that $N_{(1)}\geq N_{(2)} \geq N_{(3)}$ is a non-increasing rearrangement of $N_1,N_2,N_3$. Thus we have 
	\begin{equation}\label{estimate number of divisors of q-l2}
		\# \cS \leq  \# \left(\cD(l^2-q)\right)^2 \leq C N_{(1)}^{\frac{\eps}{10}}
	\end{equation}
	where the constant $C$ changes but still depends only on $\eps$. \\
	We now consider two cases :\\
	
	-- If $N_{(1)} \leq 2 N_{(2)}^{10}$, then from \eqref{estimate number of divisors of q-l2} we get that
	\begin{equation*}
		\# \cS \leq C N_{(2)}^{\eps} 
	\end{equation*}
	which is the desired bound. \\
	
	-- If $N_{(1)} > 2N_{(2)}^{10}$, we adapt a trick that has already been used in \cite{burq2005bilinear} (Lemma 3.2). Without loss of generality, we assume that $N_1 = N_{(1)}$ (since the analysis of cases $N_2=N_{(1)}$ and $N_3=N_{(1)}$ are similar). We split the solution set $\cS$ as 
	\begin{equation*}
		\cS = \left( \cS \cap \{ |k_1| \leq 2N_{(2)}^{10}\} \right) \sqcup \left( \cS \cap \{ |k_1| > 2N_{(2)}^{10}\} \right) 
	\end{equation*}
	For the set $\cS \cap \{ |k_1| \leq 2N_{(2)}^{10}\}$, we apply the bound previously obtained; for the other set, $ \cS \cap \{ |k_1| > 2N_{(2)}^{10}\} \leq C N_{(2)}^{\eps}$, that we denote $\widetilde{\cS}$, we prove in what follows that $\# \widetilde{\cS} \leq 2$. \\
	Fix $(k_1,k_2,k_3) \in \widetilde{\cS}$. The equation 
	\begin{equation*}
		k_1^2 = q + k_2^2 - k_3^2
	\end{equation*}
	implies that 
	\begin{equation*}
		|k_1| \in [\sqrt{q-N_3^2},\sqrt{q+N_2^2}]
	\end{equation*}
	We observe that the length of this interval is less than 1. Indeed, by mean value inequality,
	\begin{equation*}
		\begin{split}
			\sqrt{q + N_2^2} - \sqrt{q-N_3^2} & \leq \frac{1}{2\sqrt{q-N_3^2}}(N_2^2+N_3^2) \\
			& \leq \frac{N_{(2)}^2}{\sqrt{q-N_3^2}}
		\end{split}
	\end{equation*}
	Since $q = k_1^2-k_2^2+k_3^2$, we have $q \geq 4N_{(2)}^{20}-N_{(2)}^2$ so that $q-N_3^2 \geq 4N_{(2)}^{20}-2N_{(2)}^2 \geq 2 N_{(2)}^{20}$. Thus we obtain 
	\begin{equation*}
		\sqrt{q + N_2^2} - \sqrt{q-N_3^2} \leq \frac{N_{(2)}^2}{\sqrt{q-N_3^2}} \leq \frac{N_{(2)}^2}{\sqrt{2}N_{(2)}^{10}} \leq \frac{1}{\sqrt{2}} < 1
	\end{equation*}
	As there is at most one integer in the interval $[\sqrt{q-N_3^2},\sqrt{q+N_2^2}]$, we have at most two choices for $k_1$. \\
	We finally remark that when $k_1$ is fixed then there is at most one choice for $k_2$ and $k_3$. In fact, if we fix $k_1$ and take any $k_2$ and $k_3$ such that $(k_1,k_2,k_3) \in \widetilde{\cS}$, then by using for example equation \eqref{counting bound, substituting k2 in the equation} we get $k_3 = l + \frac{l^2-q}{2(k_1-l)}$ and $k_2 = k_1+k_3-l = k_1 + \frac{l^2-q}{2(k_1-l)}$. \\
	
	In conclusion, we have shown that 
	\begin{equation*}
		\# \widetilde{\cS} \leq 2
	\end{equation*}
	so \eqref{counting boung 3 integers with different signatures} is proven. \\
	\textbullet The proof of \eqref{counting bound 3 integers with same signature} follows from the same strategy, considering again the two cases $N_{(1)} \lesssim N_{(2)}^{10}$ and $N_{(1)} \gtrsim N_{(2)}^{10}$. We point out that the slight difference is that we are lead to use the divisor bound in Euclidean ring $\Z[e^{2i\pi/3}]$. We also refer to \cite{book_tzirakis},  Lemma~2.11. 
\end{proof}

As a Corollary of Lemma~\ref{lem counting bound 3 integers} and Lemma~\ref{lemma psi estimate}, we have:

\begin{cor}\label{cor psi estimate and counting bound} Let $s > \frac{3}{4}$ and $\Psi_{2s} \in \{\Psi^{(0)}_{2s},\Psi^{(1)}_{2s} \}$. Let $\eps >0$ and $N_1,...,N_6$ dyadic integers. For $\vec{k}=(k_1,...k_6) \in \Z^6$, we denote:
	\begin{align*}
		C_0(\vec{k}) &:= \1_{\linc} \hsp \1_{k_3 \neq k_2,k_4} \hsp \hsp \big( \prod_{j=1}^6 \1_{|k_j| \sim N_j} \big), & 	C_1(\vec{k}) := \1_{\linc} \hsp \big( \prod_{j=1}^6 \1_{|k_j| \sim N_j} \big)
	\end{align*}
Then,
	\begin{equation}\label{psi estimate + counting bound different signatures}
		\sum_{k_1,...,k_6} C_0(\vec{k}) \Psi_{2s}(\vec{k})^2 \hsp   \leq C_{s,\eps}N_{(1)}^\epsilon [N_{(1)}^{2s-\frac{3}{2}}N_{(3)}^{\frac{3}{2}}(N_1 N_5 N_6)^\frac{1}{2}]^2
	\end{equation}
	and,
	\begin{equation}\label{psi estimate + counting bound same signature}
		\sum_{k_1,...,k_6} C_1(\vec{k}) \Psi_{2s}(\vec{k})^2 \hsp   \leq C_{s,\eps}N_{(1)}^\epsilon [N_{(1)}^{2s-\frac{3}{2}}N_{(3)}^{\frac{3}{2}}(N_{2} N_{4} N_{6} \wedge N_1 N_3 N_5)^\frac{1}{2}]^2
	\end{equation}
	where $C_\eps>0$ is a constant only depending in $\eps$, and where the symbol $\wedge$ denotes the $\min$.
\end{cor}

\begin{rem}
	Note that the contsraint term $C_1(\vec{k})$ does not contain any condition of non-pairing. The bound in~\eqref{psi estimate + counting bound same signature} relies on the application of~\eqref{counting bound 3 integers with same signature}, which explains why it involves integers with the same signature.
\end{rem}

\begin{proof} The proof of~\eqref{psi estimate + counting bound different signatures} and~\eqref{psi estimate + counting bound same signature} is very similar. The difference is that we use either the counting bound~\eqref{counting boung 3 integers with different signatures} or~\eqref{counting bound 3 integers with same signature}. Thus, we only provide the proof of~\eqref{psi estimate + counting bound different signatures}. Now, we need to consider two cases.\\
	
-- \underline{Case 1:} When $\Psi_{2s} = \Psi_{2s}^{(0)}$. In this case, we start by applying~\eqref{psi estimate when Omg=0}:
\begin{equation*}
		\sum_{k_1,...,k_6 } C_0(\vec{k}) \Psi_{2s}^{(0)}(\vec{k})^2 \lesssim_s [(N_{(1)}^{2(s-1)} \vee N_{(3)}^{2(s-1)}) N_{(3)}^2]^2 \sum_{\substack{k_1,...,k_6 \\ \Omgz}} C_0(\vec{k})
\end{equation*}
Then, we use the counting bound~\eqref{counting boung 3 integers with different signatures} as follows:
\begin{multline}\label{sum C(k) Omgz}
	\sum_{\substack{k_1,...,k_6 \\ \Omgz}} C_0(\vec{k}) =  \sum_{\substack{k_1,k_5,k_6 \\ \forall i=1,5,6:\hsp |k_i|\sim N_i}}  \sum_{\substack{k_2,k_3,k_4 \\ \forall j=2,3,4:\hsp |k_j|\sim N_j}}  \1_{k_2-k_3+k_4=k_1+k_5-k_6} \hsp \1_{k_3 \neq k_2,k_4} \hsp \1_{k_2^2 - k_3^2 +k_4^2= k_1^2 + k_5^2 - k_6^2} \\
	\lesssim_\eps N_{(1)}^\eps \sum_{\substack{k_1,k_5,k_6 \\ \forall i=4,5,6:\hsp |k_i|\sim N_i}} 1 \hsp \lesssim_\eps N_{(1)}^\eps N_1 N_5 N_6 
\end{multline}
Moreover, since $s>\frac{3}{4}$, we have\footnote{This inequality is a bit crude; however, it will be acceptable later for the energy estimates. The reason we use it is that it simplifies the bound of Corollary~\ref{cor psi estimate and counting bound}. }:
\begin{equation}\label{crude estimate in N_1 N_3}
	(N_{(1)}^{2(s-1)} \vee N_{(3)}^{2(s-1)}) N_{(3)}^2 = (N_{(1)}^{2(s-1)} N_{(3)}^{\frac{1}{2}} \vee N_{(3)}^{2s - \frac{3}{2}}) N_{(3)}^{\frac{3}{2}} \leq (N_{(1)}^{2s - \frac{3}{2}} \vee N_{(3)}^{2s - \frac{3}{2}}) N_{(3)}^{\frac{3}{2}} = N_{(1)}^{2s - \frac{3}{2}}N_{(3)}^{\frac{3}{2}}
\end{equation}
	Therefore, combining the three inequalities above, we obtain:
	\begin{equation*}
			\sum_{k_1,...,k_6 } C_0(\vec{k}) \Psi_{2s}^{(0)}(\vec{k})^2 \lesssim_{s,\eps} N_{(1)}^\eps [N_{(1)}^{2s - \frac{3}{2}}N_{(3)}^{\frac{3}{2}} (N_1N_5N_6)^{\frac{1}{2}} ]^2
	\end{equation*}
--	\underline{Case 2:} When $\Psi_{2s} = \Psi_{2s}^{(1)}$.
In this case, we start by applying~\eqref{psi/Omg estimate}, and then we use~\eqref{crude estimate in N_1 N_3}:
 \begin{multline*}
 		\sum_{k_1,...,k_6} C_0(\vec{k}) \Psi_{2s}^{(1)}(\vec{k})^2  \lesssim_s [N_{(1)}^{2(s-1)}]^2 \sum_{k_1,...,k_6} C_0(\vec{k}) + [(N_{(1)}^{2(s-1)} \vee N_{(3)}^{2(s-1)}) N_{(3)}^2]^2 \sum_{\substack{k_1,...,k_6 \\ \Omgnz}} \frac{C_0(\vec{k})}{\Omgk^2} \\
 		\lesssim_s [N_{(1)}^{2(s-1)}]^2 \sum_{k_1,...,k_6} C_0(\vec{k}) + [N_{(1)}^{2s - \frac{3}{2}}N_{(3)}^{\frac{3}{2}}]^2 \sum_{\substack{k_1,...,k_6 \\ \Omgnz}} \frac{C_0(\vec{k})}{\Omgk^2}
 \end{multline*}
 On the one hand, by Lemma~\ref{lem counting bound},
 \begin{equation*}
 	\sum_{k_1,...,k_6} C_0(\vec{k}) \lesssim N_{(1)} N_{(3)} N_1 N_5 N_6
 \end{equation*}
 On the other hand, we write:
 \begin{equation*}
 	 \sum_{\substack{k_1,...,k_6 \\ \Omgnz}} \frac{C(\vec{k})}{\Omgk^2} =  \sum_{\kappa \in \Z \setminus \{ 0\}} \frac{1}{\kappa^2} \sum_{\substack{k_1,...,k_6 \\ \Omgk = \kappa }} C_0(\vec{k}) 
 \end{equation*}
 Thus, similarly to~\eqref{sum C(k) Omgz}, we deduce that:
 \begin{equation*}
 	 \sum_{\substack{k_1,...,k_6 \\ \Omgnz}} \frac{C(\vec{k})}{\Omgk^2} \lesssim_\eps \sum_{\kappa \in \Z \setminus \{ 0\}} \frac{1}{\kappa^2} N_{(1)}^\eps N_1 N_5 N_6  \lesssim_\eps N_{(1)}^\eps N_1 N_5 N_6 
 \end{equation*}
 Combining the inequalities above leads to:
 \begin{multline*}
 \sum_{k_1,...,k_6}C_0(\vec{k}) \Psi_{2s}^{(1)}(\vec{k})^2  \lesssim_{s,\eps} N_{(1)}^\eps [N_{(1)}^{2s - \frac{3}{2}}N_{(3)}^{\frac{1}{2}} (N_1 N_5 N_6)^{\frac{1}{2}}]^2 +  N_{(1)}^\eps [N_{(1)}^{2s-\frac{3}{2}} N_{(3)}^{\frac{3}{2}}(N_1 N_5 N_6)^{\frac{1}{2}}]^2 \\
 	\lesssim_{s,\eps}  N_{(1)}^\eps [N_{(1)}^{2s-\frac{3}{2}} N_{(3)}^{\frac{3}{2}}(N_1 N_5 N_6)^{\frac{1}{2}}]^2
 \end{multline*}
 Hence, the proof of Corollary~\ref{cor psi estimate and counting bound} is complete.
\end{proof}

%% file: Deterministic_estimates.tex
Let us recall that we need to prove the exponential integrability properties from Proposition~\ref{prop exp integrability Fs,N} for the proof of quasi-invariance (Theorem~\ref{thm quasi-inv}) to be complete. The proof of this proposition will combine deterministic and probabilistic approaches. In this section, we build our deterministic toolbox for the proof of Proposition~\ref{prop exp integrability Fs,N}, the probabilistic one having already been elaborated in Section~\ref{section Preliminaries}. We start by an $L^2(\T)$-estimate for the non-linearity $F(U) = |U|^4U$ localized in frequency. The restriction $s>\frac{9}{10}$ for the quasi-invariance in Theorem~\ref{thm quasi-inv} comes partly from this estimate.

\begin{lem}\label{lem localized L2 estimate on |u|4u}
	Let $s>\frac{9}{10}$, $\sigma=s-\frac{1}{2}-$, and $R>0$. Let also:
	\begin{center}
		$G \in \{id, F\}$, where $ F : U \mapsto |U|^4U$
	\end{center}
	Then, for every $\eps>0$ small enough, there exists a constant $C_{\eps,R}>0$, such that for every dyadic integer $N_1$, and every $U \in C^\infty(\T)$, such that $\hsignorm{U} \leq R$, we have:
	\begin{equation}\label{localized L2 estimate on |u|4u}
		\norm{P_{N_1}G(U)}_{L^2(\T)} \leq C_{\eps,R} N_1^{n_s +\eps} 
	\end{equation} 
	where $n_s$ is the strictly negative number given by:
	\begin{equation}\label{ns}
		n_s =	\begin{cases}
			\frac{1}{2}-s \hspace{0.3cm} \textnormal{for $s>1$} \\
			\frac{9}{2}-5s \hspace{0.3cm} \textnormal{for $s\leq1$}
		\end{cases}
	\end{equation}
\end{lem}

\begin{rem}
	Note that we always have $\norm{P_{N_1}U}_{L^2} \lesssim N_1^{\frac{1}{2} -s +} \hsignorm{U}$. Besides, we have $n_s \geq \frac{1}{2}-s$. Then, \eqref{localized L2 estimate on |u|4u} is interesting only when $G(U) = |U|^4U$. The reason we stated Lemma~\ref{lem localized L2 estimate on |u|4u} for $G\in\{id,F\}$ is solely for consistency with the proof of Lemma~\ref{lem exp integrability Rs}.
\end{rem}

\begin{proof} We only need to consider the case $G(U)=|U|^4U$.\\
	By Parseval's theorem, and since the Fourier transform converts the product into convolution, we have:
	\begin{equation*}
		\norm{P_{N_1}(|U|^4U)}_{L^2(\T)}^2 = \sum_{|k_1|\sim N_1} \big| \sum_{\lincba} U_{p_1}\cjg{U_{p_2}}...U_{p_5}\big|^2
	\end{equation*}
	Then, we decompose dyadically the sum in the $p_j$'s, and deduce that:
	\begin{equation*}
		\norm{P_{N_1}(|U|^4U)}_{L^2(\T)}^2 = \sum_{|k_1|\sim N_1} \big| \sum_{P_1,...,P_5}\sum_{\lincba} U^{P_1}_{p_1}\cjg{U_{p_2}^{P_2}}...U^{P_5}_{p_5}\big|^2
	\end{equation*}
	where we adopt the notation $U_{p_j}^{P_j} = \1_{|p_j| \sim P_j} U_{p_j}$. Next, expanding the square leads to:
	\begin{multline}\label{2nd gen expanding the square}
		\norm{P_{N_1}(|U|^4U)}_{L^2(\T)}^2 \\ \leq   \sum_{\substack{P_1,...,P_5 \\ Q_1,...Q_5}} \sum_{|k_1|\sim N_1} \big( \sum_{\lincba} |U^{P_1}_{p_1}U_{p_2}^{P_2}...U^{P_5}_{p_5}|\big)\big( \sum_{q_1-q_2+...+q_5=k_1} |U^{Q_1}_{q_1}U_{q_2}^{Q_2}...U^{Q_5}_{q_5}|\big) 
	\end{multline}
	Let us now invoke $P_{(1)} \geq ... \geq P_{(5)}$ a non-increasing rearrangement of $P_1,...,P_5$; and similarly we invoke $Q_{(1)} \geq ... \geq Q_{(5)}$ a non-increasing rearrangement of $Q_1,...,Q_5$. Note that the constraints in the sums imply that the non-zero contributions occur only when $P_{(1)} \gtrsim N_1$ and $Q_{(1)} \gtrsim N_1$. In what follows, we prove that for $\eps>0$ small enough we have:
	\begin{equation}\label{CS 2nd gen}
		\big[ \sum_{|k_1|\sim N_1} \big( \sum_{\lincba} |U^{P_1}_{p_1}U_{p_2}^{P_2}...U^{P_5}_{p_5}|\big)^2 \big] ^{\frac{1}{2}}\lesssim N_{1}^{n_s+\eps} P_{(1)}^{-\frac{\eps}{2}}\hsignorm{U}^5
	\end{equation}
	Note that the same statement will also hold for the $Q_j$'s. To prove~\eqref{CS 2nd gen}, we use the Cauchy-Schwarz inequality, and then Lemma~\ref{lem counting bound}, as follows:
	\begin{multline*}
		\sum_{|k_1|\sim N_1} \big( \sum_{\lincba} |U^{P_1}_{p_1}U_{p_2}^{P_2}...U^{P_5}_{p_5}|\big)^2 \\
		\leq \sum_{|k_1|\sim N_1} \big( \sum_{\lincba} \prod_{j=1}^5 \1_{|p_j|\sim P_j}\big) \big( \sum_{\lincba} |U^{P_1}_{p_1}U_{p_2}^{P_2}...U^{P_5}_{p_5}|^2 \big) \\
		\lesssim P_{(2)}...P_{(5)}  \sum_{|k_1|\sim N_1}  \sum_{\lincba} |U^{P_1}_{p_1}U_{p_2}^{P_2}...U^{P_5}_{p_5}|^2 \\ \lesssim  P_{(2)}...P_{(5)} \prod_{j=1}^5 \norm{U^{P_j}}_{L^2(\T)}^2 \lesssim [P_{(1)}^{\frac{1}{2}-s+\delta}(P_{(2)}...P_{(5)})^{1-s}\hsignorm{U}^5]^2  \lesssim [P_{(1)}^{n_s + \delta}\hsignorm{U}^5]^2 
	\end{multline*}
	where $\delta>0$ is a parameter arbitrary small, which goes to 0 as $\sigma$ goes to $s-\frac{1}{2}$; and where we recall that $n_s$ is the strictly negative number given in~\eqref{ns}. Now, let $\eps>0$ small enough such that $n_s + \eps <0$. Then, let us choose $\delta<\frac{\eps}{2}$. Since $P_{(1)} \gtrsim N_1$, we have: 
	\begin{equation*}
		P_{(1)}^{n_s+\delta} \lesssim N_1^{n_s+ \frac{\eps}{2}+\delta} P_{(1)}^{-\frac{\eps}{2}} \lesssim N_1^{n_s + \eps}P_{(1)}^{-\frac{\eps}{2}}
	\end{equation*} 
	Using this in the inequality above leads to~\eqref{CS 2nd gen}. Let us now continue the estimate~\eqref{2nd gen expanding the square}. Using the Cauchy-Schwarz inequality in the $k_1$ summation, and then~\eqref{CS 2nd gen}, yields:
	\begin{equation*}
		\norm{P_{N_1}(|U|^4U)}_{L^2(\T)}^2 \lesssim [N_1^{n_s + \eps} \hsignorm{U}^5 ]^2 \sum_{\substack{P_1,...,P_5 \\ Q_1,...Q_5}} P_{(1)}^{-\frac{\eps}{2}} Q_{(1)}^{-\frac{\eps}{2}} \lesssim_\eps [N_1^{n_s + \eps} \hsignorm{U}^5 ]^2
	\end{equation*}
	which completes the proof of Lemma~\ref{lem localized L2 estimate on |u|4u} 
\end{proof}

The deterministic estimates that will follow are based on the following Strichartz estimates (for the propagator of linear Schrödinger equation), due to Bourgain in~\cite{Bourgain1993}.

\begin{thm}[Strichartz estimates]\label{Strichartz estimate} Let $\eps>0$. Then, for any function $g : \T \ra \C$, we have:
	\begin{equation*}
		\norm{e^{it\p_x^2}g}_{L^6(\T_t \times \T_x)} \leq C_{\eps} \norm{g}_{H^{\eps}(\T)}
	\end{equation*}
	and,
	\begin{equation*}
		\norm{e^{it\p_x^2}g}_{L^4(\T_t \times \T_x)} \leq C \norm{g}_{L^2(\T)}
	\end{equation*}
where $C>0$ is a universal constant, and $C_\eps >0$ is a constant only depending on $\eps$.
\end{thm}
From these, we can prove: 

\begin{lem}\label{Strichartz's trick}  Let $( f_{k_1}^{(1)} )_{k_1 \in \Z}$,...,$( f_{k_6}^{(6)} )_{k_6 \in \Z}$ any sequences  of complex numbers that satisfy $f^{(j)}_{k_j}=\1_{|k_j| \sim N_j} f^{(j)}_{k_j} $, where $N_1,...,N_6$ are dyadic integers. Then, \\
	
-- for every $\eps>0$, there exists a constant $C_{\eps} > 0$, only depending on $\eps$, such that for any $\kappa \in \Z$ :
		\begin{equation}\label{L6 Strichartz estimate}
			\sum_{\substack{\linc \\ \Omega(\vec{k})= \kappa}} \prod_{j=1}^6 |f_{k_j}^{(j)}| \leq C_{\eps} N_{(1)}^{\eps} \prod_{j=1}^6 \norm{f^{(j)}}_{l^2}
		\end{equation}
-- there exists a universal constant $C>0$, such that for any $\kappa \in \Z$ :
		\begin{equation}\label{L4 Strichartz estimate}
			\sum_{\substack{k_1-k_2+k_3-k_4 = 0 \\ \Omega^{\downarrow}(\vec{k})= \kappa}} \prod_{j=1}^4 |f_{k_j}^{(j)}| \leq C \prod_{j=1}^4 \norm{f^{(j)}}_{l^2}
		\end{equation}
\end{lem}
We refer to \cite{knez24transportlowregularitygaussian} for a proof of Theorem~\ref{Strichartz estimate} $\implies$ Lemma~\ref{Strichartz's trick}. The next two lemmas will serve as reference estimates for handling deterministic terms in the Proof of Proposition~\ref{prop exp integrability Fs,N}. They are meant to be combined with Lemma~\ref{lem localized L2 estimate on |u|4u}.
\begin{lem}[Deterministic dyadic estimates]\label{lem deterministic estimate via Strichartz's trick} Let $s >0$. Let $\Psi_{2s} \in \{ \Psi_{2s}^{(0)}, \Psi_{2s}^{(1)}\}$ and $\Psi^\downarrow_s \in \{ \Psi_{2s}^{(0),\downarrow}, \Psi_{2s}^{(1),\downarrow}\}$.
 Let $( f_{k_1}^{(1)} )_{k_1 \in \Z}$,...,$( f_{k_6}^{(6)} )_{k_6 \in \Z}$ be sequences  of complex numbers that satisfy $f^{(j)}_{k_j}~=~\1_{|k_j| \sim N_j} f^{(j)}_{k_j} $, where $N_1,...,N_6$ are dyadic integers. Then,
	 for every $\eps > 0$, there exists a constant $C_{\eps,s} > 0$, only depending on $\eps$ and $s$, such that
		\begin{equation}\label{L6 deterministic estimate}
			\sum_{\linc } | \Psi_{2s}(\vec{k}) | \prod_{j=1}^6 | f_{k_j}^{(j)} | \leq C_{\eps,s}N_{(1)}^\eps (N_{(1)}^{2(s-1)} \vee N_{(3)}^{2(s-1)}) N_{(3)}^{2} \prod_{j=1}^6 \norm{f^{(j)}}_{l^2}
		\end{equation}
		and similarly for $N_1,..,N_4$ dyadic integers,
		\begin{equation}\label{L4 deterministic estimate}
				\sum_{k_1-k_2+k_3-k_4=0} |\Psi_{2s}^\downarrow(\vec{k})| \prod_{j=1}^4 | f_{k_j}^{(j)} | \leq  C_{s} (N_{(1)}^{2(s-1)} \vee N_{(3)}^{2(s-1)}) N_{(3)}^{2} \prod_{j=1}^4\norm{f^{(j)}}_{l^2}
		\end{equation}
	where $C_s>0$ is a constant only depending on $s$.
\end{lem}

In the situation where $N_{(4)} \ll N_{(3)}$, we have the following refined statement:
\begin{lem}[Refined deterministic dyadic estimates]\label{lem refined deterministic estimate via Strichartz's trick} 	Let $s >0$. Let $\Psi_{2s} \in \{\Psi_{2s}^{(0)},\Psi_{2s}^{(1)} \}$ and $\Psi^\downarrow_s \in \{ \Psi_{2s}^{(0),\downarrow}, \Psi_{2s}^{(1),\downarrow}\}$. Let $N_1,...,N_6$ be dyadic integers such that  $N_{(4)} \ll N_{(3)}$. Let $( f_{k_1}^{(1)} )_{k_1 \in \Z}$,...,$( f_{k_6}^{(6)} )_{k_6 \in \Z}$ sequences such that $f^{(j)}$is localized at the frequency $N_j$. There exists a constant $C_s >0$, only depending on $s$, such that	\begin{equation}\label{L6 refined deterministic estimate}
		\sum_{\linc} |\Psi_{2s}(\vec{k})| \prod_{j=1}^6 | f_{k_j}^{(j)} | \leq C_{s} (N_{(1)}^{2(s-1)} \vee N_{(3)}^{2(s-1)}) (N_{(3)}N_{(4)}N_{(5)}N_{(6)})^{\frac{1}{2}} \prod_{j=1}^6 \norm{f^{(j)}}_{l^2}
\end{equation} 
	Similarly, if  $N_1,...,N_4$ are dyadic integers such that  $N_{(4)} \ll N_{(3)}$, then
		\begin{equation}\label{L4 refined deterministic estimate}
		\sum_{k_1-k_2+k_3-k_4=0} |\Psi_{2s}^{\downarrow}(\vec{k})| \prod_{j=1}^4 | f_{k_j}^{(j)} | \leq  C_{s} (N_{(1)}^{2(s-1)} \vee N_{(3)}^{2(s-1)})(N_{(3)} N_{(4)})^{\frac{1}{2}} \prod_{j=1}^4\norm{f^{(j)}}_{l^2}
	\end{equation}
\end{lem}

\begin{rem}
	The estimates \eqref{L6 refined deterministic estimate} and \eqref{L4 refined deterministic estimate} are relevant only when $\Psi_{2s} = \Psi_{2s}^{(1)}$ and $\Psi^\downarrow_{2s} = \Psi_{2s}^{(1),\downarrow}$, because when $\Psi_{2s} = \Psi_{2s}^{(0)}$ or $\Psi^\downarrow_{2s} = \Psi_{2s}^{(0),\downarrow}$, then the contributions in the sums in~\eqref{L6 refined deterministic estimate} and~\eqref{L4 refined deterministic estimate} are zero, since we cannot have $|k_{(4)}| \ll |k_{(3)}|$ thanks to Remark~\ref{rem Omgk = 0 implies N4 sim N3}. We stated these estimates for  $\Psi_{2s} \in \{\Psi_{2s}^{(0)},\Psi_{2s}^{(1)} \}$ and $\Psi^\downarrow_s \in \{ \Psi_{2s}^{(0),\downarrow}, \Psi_{2s}^{(1),\downarrow}\}$ only for consistency with the proof of Lemma~\ref{lem exp integrability Rs}.
\end{rem}

\begin{proof}[Proof of Lemma \ref{lem deterministic estimate via Strichartz's trick} and Lemma \ref{lem refined deterministic estimate via Strichartz's trick}] Here, we only provide the proof of~\eqref{L6 deterministic estimate} and~\eqref{L6 refined deterministic estimate}; the proof of~\eqref{L4 deterministic estimate} and~\eqref{L4 refined deterministic estimate} is identical when using the lemmas~\ref{lemma psi estimate},~\ref{lem lower bound Omega when three high and three low freq} and~\ref{Strichartz's trick} in their degenerated version. \\
	
The arrangement of the frequencies does not play a role here, so we assume for more readability that $N_j=N_{(j)}$, for $j=1,...,6$.\\

\textbullet First we prove~\eqref{L6 deterministic estimate}. We need to consider two cases: \\
-- \textit{Case 1:} When $\Psi_{2s}(\vec{k}) = \Psi^{(0)}_{2s}(\vec{k})$. This case is the easiest; indeed, we obtain~\eqref{L6 deterministic estimate} simply by applying the estimate on $\Psi^{(0)}_{2s}$ in~\eqref{psi estimate when Omg=0}, and then the $L^6$-Strichartz's estimate~\eqref{L6 Strichartz estimate}. \\
 
\noindent-- \textit{Case 2:} When $\Psi_{2s}(\vec{k}) = \Psi^{(1)}_{2s}(\vec{k})$. We start by applying Lemma~\ref{lemma psi estimate} :
	\begin{equation}\label{application psi/Omg estimate}
			\sum_{\linc} |\Psi_{2s}^{(1)}(\vec{k})| \prod_{j=1}^6 | f_{k_j}^{(j)} |  \leq C_s \big(  N_{1}^{2(s-1)} \hsp \textbf{I} +  (N_{1}^{2(s-1)} \vee N_{3}^{2(s-1)}) N_{3}^2 \hsp \textbf{II} \big)
	\end{equation}
	where we denote 
	\begin{align*}
		\textbf{I} &:=\sum_{\substack{\linc \\ \Omgnz}} \prod_{j=1}^6 | f_{k_j}^{(j)} | & &\textnormal{and} &	\textbf{II} &:= \sum_{\substack{\linc \\ \Omgnz}} \frac{1}{|\Omega(\vec{k})|}\prod_{j=1}^6 | f_{k_j}^{(j)} |
	\end{align*}
	Now, we estimate separately \textbf{I} and \textbf{II}. \\
	
	--\underline{Estimate for \textbf{I} :} To treat this term, we remove the constraint $\Omgnz$ and we use the Cauchy-Schwarz inequality in the $k_{1}$ summation as follows:
	\begin{multline}\label{determinstic estimate without Omg}
		\textbf{I} \leq \sum_{k_3,k_4,k_5,k_6} |f^{(3)}_{k_3}f^{(4)}_{k_4}f^{(5)}_{k_5}f^{(6)}_{k_6}| \sum_{k_1} |f^{(1)}_{k_1}||f^{(2)}_{k_1+k_3...-k_6}| \leq \prod_{j=3}^6 \norm{f^{(j)}}_{l^1} \norm{f^{(1)}}_{l^2} \norm{f^{(2)}}_{l^2} \\ \lesssim  \left( N_{3} N_{4}N_{5} N_{6} \right)^{\frac{1}{2}} \prod_{j=1}^6 \norm{f^{(j)}}_{l^2}
	\end{multline}
	and $N_1^{2(s-1)} \left( N_{3} N_{4}N_{5} N_{6} \right)^{\frac{1}{2}} \leq N_1^{2(s-1)} N_3^2$, so $N_1^{2(s-1)}\textbf{I}$ is indeed bounded by the (RHS) of~\eqref{L6 deterministic estimate}.\\
		--\underline{Estimate for \textbf{II} :} Firstly, we observe that  $|\Omega(\vec{k})|\leq 6 N_{1}^2$, and then we apply Lemma~\ref{Strichartz's trick}:
		\begin{multline*}
			\textbf{II} \leq \sum_{ 6N_{1}^2 \geq |\kappa| \geq 1} \frac{1}{|\kappa|} \sum_{\substack{\linc \\ \Omgeqkp}} \prod_{j=1}^6 |f_{k_j}^{(j)}| \leq C_\eps N_1^{\frac{\eps}{2}} \prod_{j=1}^6 \norm{f^{(j)}}_{l^2} \sum_{ 6N_{1}^2 \geq |\kappa| \geq 1} \frac{1}{|\kappa|} \\
			\leq C_\eps N_1^{\frac{\eps}{2}} \log(N_1) \prod_{j=1}^6 \norm{f^{(j)}}_{l^2} \leq 	C_\eps N_1^{\eps}  \prod_{j=1}^6 \norm{f^{(j)}}_{l^2}
		\end{multline*}
		where the constant $C_\eps$ changes but still depends only on $\eps$. \\ 
		Finally, coming back to~\eqref{application psi/Omg estimate}, and with the estimate of \textbf{I} above, the estimate~\eqref{L6 deterministic estimate} is proven. \\
		
		\textbullet Now, we assume $N_{(4)} \ll N_{(3)}$, that is, here, $N_4 \ll N_3$, and we proceed to the proof of~\eqref{L6 refined deterministic estimate}. Once again we start from the inequality~\eqref{application psi/Omg estimate}. We still estimate \textbf{I} using~\eqref{determinstic estimate without Omg}. However, to estimate \textbf{II}, we use the lower bound on $\Omgk$ from Lemma~\ref{lem lower bound Omega when three high and three low freq} (note that we are indeed in a position to apply this lemma) to obtain: 
		\begin{multline*}
			\textbf{II} \leq C (N_1 N_3)^{-1} \textbf{I} \\ \leq C (N_1 N_3)^{-1} \left( N_{3} N_{4}N_{5} N_{6} \right)^{\frac{1}{2}} \prod_{j=1}^6 \norm{f^{(j)}}_{l^2} \leq C N_3^{-2} \left( N_{3} N_{4}N_{5} N_{6} \right)^{\frac{1}{2}} \prod_{j=1}^6 \norm{f^{(j)}}_{l^2}
		\end{multline*}
		where in the second inequality we have used \eqref{determinstic estimate without Omg} again. Hence, combining~\eqref{application psi/Omg estimate} with the inequality above and~\eqref{determinstic estimate without Omg}, we deduce that:
		\begin{multline*}
			\sum_{\linc}|\Psi^{(1)}_{2s}(\vec{k})| \prod_{j=1}^6 | f_{k_j}^{(j)} |  \\
			\leq C_s \big(  N_{1}^{2(s-1)}\left( N_{3} N_{4}N_{5} N_{6} \right)^{\frac{1}{2}} +  (N_{1}^{2(s-1)} \vee N_{3}^{2(s-1)})  \left( N_{3} N_{4}N_{5} N_{6} \right)^{\frac{1}{2}}\big) \prod_{j=1}^6 \norm{f^{(j)}}_{l^2} \\
			\leq C_s (N_{1}^{2(s-1)} \vee N_{3}^{2(s-1)})  \left( N_{3} N_{4}N_{5} N_{6} \right)^{\frac{1}{2}} \prod_{j=1}^6\norm{f^{(j)}}_{l^2}
		\end{multline*}
		This completes the proof of Lemma~\ref{lem deterministic estimate via Strichartz's trick}.
\end{proof}

%% file: Exponential_integrability.tex
This section is central. It is dedicated to the proof of  Proposition~\ref{prop exp integrability Fs,N}, using the properties prepared in Sections~\ref{section lower bound omega and counting bounds} and~\ref{section Deterministic estimates}. We organize this section in three paragraphs. In the first one, we expand the energy functionals into dyadic blocks. We then concentrate on these blocks, providing a key estimate that we will prove implies the desired proposition. The key estimate is proven in the third paragraph; and in preparation for its proof, we review a strategy in the second paragraph.
 
\subsection{Dyadic estimate and proof of the exponential integrability} Recall that $\cM_{s,N},\cT_{s,N}$, and $\cN_{s,N}$ have been defined in Definition~\ref{def energy functionals}.

\begin{defn}[Dyadic block]\label{def dyadic block} Let $N \in \N$ and $\cF_{s,N} \in \{ \cM_{s,N},\cT_{s,N},\cN_{s,N}\}$, and consider the unique $(\Psi_{2s},G) \in \{\Psi^{(0)}_{2s},\Psi^{(1)}_{2s}\} \times \{id , F_N \}$ such that:
		\begin{equation*}
		\cF_{s,N}(u_1,...,u_6)= \sum_{\linc} \Psi_{2s}(\vec{k}) \widehat{Gu}_1(k_1)\cjg{\widehat{u}_2}(k_2)...\cjg{\widehat{u}_6}(k_6)
	\end{equation*}
	for all $u_j \in C^\infty(\T)$ satisfying $supp(\widehat{u_j})\subset [-N,N]$.
	Then, for every dyadic vectors $\vec{N}=(N_1,...,N_6)\in (2^\N)^6$, we define the dyadic block for $\cF_{s,N}$ as:
	\begin{equation}\label{Fs dyadic block}
	\cF_{s,\vec{N}}(u_1,...,u_6) := \sum_{\substack{\linc \\ \forall j, \hsp |k_j|\sim N_j}} \Psi_{2s}(\vec{k}) \widehat{Gu}_1(k_1)\cjg{\widehat{u}_2}(k_2)...\cjg{\widehat{u}_6}(k_6)
\end{equation}
Consequently, we have the dyadic decomposition:
	\begin{equation}\label{dyadic decomposition in the rem}
 	\cF_{s,N}= \sum_{\vec{N}} 	\cF_{s,\vec{N}}
 \end{equation}
 where the summation is performed over the dyadic vectors $\vec{N}=(N_1,...,N_6) \in (2^\N)^6$.
\end{defn}

\begin{notn}\label{notn fromula dyadic block}Let $N \in \N$ and $\cF_{s,N} \in \{\cM_{s,N},\cT_{s,N}, \cN_{s,N} \}$. For any function $W : \T \ra \C$, we use the notations:
	\begin{align*}
		\forall k \in \Z, \hsp W_k &:= \widehat{W}(k), & W^L &:= P_L \circ \pi_N (W)
	\end{align*}
   where $L$ is dyadic and $P_L$ is the corresponding dyadic frequency projector. Then, for any function $U:\T \ra \C$ with $supp(\widehat{U}) \subset [-N,N]$, we have:
	\begin{equation*}
		\cF_{s,\vec{N}}(U) = \sum_{\linc} \Psi_{2s}(\vec{k}) G^{N_1}_{k_1} \cjg{U^{N_2}_{k_2}}U^{N_3}_{k_3}\cjg{U^{N_4}_{k_4}}U^{N_5}_{k_5}\cjg{U^{N_6}_{k_6}}
	\end{equation*}
	with the abuse of notation $G = G(U)$, that we use for commodity.
\end{notn}

\begin{lem}\label{lem exp integrability Rs}
	Let $s>\frac{9}{10}$, $\sigma= s - \frac{1}{2}-$, $N \in \N$ and $\vec{N}=(N_1,...,N_6) \in (2^\N)^6$. Let $\cF_{s,N} \in \{\cM_{s,N},\cT_{s,N}, \cN_{s,N} \}$, and invoke $\cF_{s,\vec{N}}$ the corresponding dyadic block. Then,  for $\theta \in (0,1)$ small enough, we have that for every $R>0$, there exists $C_{s,R}>0$, only depending on $s$ and $R$, such that for every $V \in H^s(\T)$~:
	\begin{multline}\label{estimate Fs dyadic block}
		\big|\cF_{s,\vec{N}}(\phi +V) \big| \1_{B_R}(\pi_N \phi + \pi_N V) \\ \leq C_{s,R} \big( N_{(1)}^{0-} + N_{(1)}^{0-}\norm{V}_{H^s}^{2-\theta}(1+\hsignorm{\phi})^{1+\theta} + X_{\vec{N}} \norm{V}_{H^s} + Y_{\vec{N}} \big)
	\end{multline} 
	where $X_{\vec{N}}$ and $Y_{\vec{N}}$ are two non-negative random variables that satisfy~:
    \begin{align}\label{rv dyadic estimate Rs}
        \E[X_{\vec{N}}^2] & \lesssim_s N_{(1)}^{0-}, & \E[Y_{\vec{N}}] & \lesssim_s N_{(1)}^{0-}
    \end{align}
\end{lem}
For more clarity, we postpone the proof of Lemma~\ref{lem exp integrability Rs} for the end of this section. Now, we show how this lemma, combined with the simplified Boué-Dupuis formula of Lemma~\ref{lem boue-dupuis}, imply Proposition~\ref{prop exp integrability Fs,N}. 
\begin{proof}[Proof of Proposition~\ref{prop exp integrability Fs,N} assuming Lemma~\ref{lem exp integrability Rs}] Let $\cF_{s,N} \in \{ \cM_{s,N}, \cT_{s,N}, \cN_{s,N}\}$.
	Our starting point is the simplified Boué-Dupuis formula from Lemma~\ref{lem boue-dupuis}, applied to the function $\cF = \alpha \cF_{s,N} \1_{B_R}$~:
	\begin{multline}\label{boue dupuis Rs}
		\log \int \1_{B_R}(u) e^{\alpha | \cF_{s,N}( u)|}d\mu_s \leq \log \int  e^{\alpha | \cF_{s,N}(\pi_N u)|\1_{B_R}(\pi_N u)}d\mu_s \\
		\leq \E \big[\sup_{V \in H^s} \{ \alpha | \cF_{s,N}( \phi^\omega +  V)|\1_{B_R}(\pi_N \phi^\omega + \pi_N V) - \frac{1}{2}\norm{V}^2_{H^s} \} \big]
	\end{multline}
	As in Definition~\ref{def dyadic block}, we decompose $\cF_{s,N}$ dyadically.
	Thus, from the above inequality we obtain:
	\begin{multline*}
		\log \int \1_{B_R}(u) e^{\alpha | \cF_{s,N}( u)|}d\mu_s
		\\ \leq \E \big[\sup_{V \in H^s} \{ \sum_{\vec{N}} \alpha | \cF_{s,\vec{N}}( \phi^\omega + V)|\1_{B_R}(\pi_N \phi^\omega + \pi_N V) - \frac{1}{2}\norm{V}^2_{H^s} \} \big]
	\end{multline*}
Now, applying Lemma~\ref{lem exp integrability Rs}, introducing the parameter $\theta \in (0,1)$ close 0 and the random variables $X_{\vec{N}}$ and $Y_{\vec{N}}$ satisfying~\eqref{rv dyadic estimate Rs}, yields:
\begin{multline*}
     \log \int \1_{B_R}(u) e^{\alpha | \cF_{s,N}( u)|}d\mu_s
		 \\
    \lesssim_{s,R} \E \big[\sup_{V \in H^s} \{ \alpha +\alpha \|V\|_{H^s}^{2-\theta}(1+\hsignorm{\phi})^{1+\theta} + \alpha \sum_{\vec{N}}X_{\vec{N}} \| V\|_{H^s}  + \alpha \sum_{\vec{N}}Y_{\vec{N}}  - \frac{1}{2}\norm{V}^2_{H^s} \} \big] \\
    \lesssim \alpha +\E \big[\sup_{x >0} \alpha x^{2-\theta}(1+\hsignorm{\phi})^{1+\theta} - \frac{1}{4}x^2\big] + \E \big[\sup_{x >0} \big(\alpha \sum_{\vec{N}}X_{\vec{N}}\big) x - \frac{1}{4}x^2\big]  + \alpha \E \big[ \sum_{\vec{N}}Y_{\vec{N}}  \big]
\end{multline*}
Let us now recall the following elementary facts:
\begin{align*}
	\sup_{x>0} \beta x^{2-\theta} - \frac{1}{4}x^2 & \lesssim_\theta \beta^{\frac{2}{\theta}} & &\textnormal{and} & \sup_{x>0} \beta x - \frac{1}{4}x^2 \lesssim \beta^2, \hspace{1cm} \textnormal{for any $\beta \geq 0$.}
\end{align*}
Then, continuing the estimates above:
\begin{equation*}
	 \log \int \1_{B_R}(u) e^{\alpha | \cF_{s,N}( u)|}d\mu_s \lesssim_{s,R} \alpha+ \alpha^{\frac{2}{\theta}}\E\big[(1+\hsignorm{\phi})^{\frac{2(1+\theta)}{\theta}} \big] + \alpha^2\E\big[\big( \sum_{\vec{N}} X_{\vec{N}}\big)^2 \big] + \alpha \sum_{\vec{N}} \E[Y_{\vec{N}}]
\end{equation*}
To conclude, we firstly write that $\E[(1+\hsignorm{\phi})^{m}] \leq C_m <+\infty$ for every $m>0$, from Fernique's integrability theorem (see for example \cite{bogachev1998gaussian} or \cite{kuo2006gaussian}); and secondly from~\eqref{rv dyadic estimate Rs}:
\begin{align*}
	\E\big[\big( \sum_{\vec{N}} X_{\vec{N}}\big)^2 \big] &= \sum_{\vec{N},\vec{M}} \E[X_{\vec{N}}X_{\vec{M}}] \leq \sum_{\vec{N},\vec{M}} \E[X_{\vec{N}}^2]^{\frac{1}{2}}\E[X_{\vec{M}}^2]^{\frac{1}{2}} <+\infty, & \sum_{\vec{N}} \E[Y_{\vec{N}}] & < +\infty
\end{align*}
Therefore, we indeed have in the end:
\begin{equation*}
	\log \int \1_{B_R}(u) e^{\alpha | \cF_{s,N}( u)|}d\mu_s \lesssim_{s,R} \alpha^{\frac{2}{\theta}}
\end{equation*}
for all $\alpha\geq 1$.
\end{proof}

\subsection{Strategy and remarks for the proof of the dyadic estimate}\label{par strategy and remarks}
For commodity we use the notation:
\begin{equation*}
	U := \pi_N \phi + \pi_N V
\end{equation*}
Let $\cF_{s,N} \in \{\cM_{s,N},\cT_{s,N},\cN_{s,N}\}$ and consider the associated dyadic block $\cF_{s,\vec{N}}$. As in Notation~\ref{notn fromula dyadic block}, we consider $(\Psi_{2s},G) \in \{\Psi^{(0)}_{2s} ,\Psi^{(1)}_{2s}\} \times \{id, F_N\}$ such that: 
\begin{equation}\label{dyadic formula for FsN}
		\cF_{s,\vec{N}}(U)
		=  \sum_{\linc} \Psi_{2s}(\vec{k}) G^{N_1}_{k_1} \cjg{U^{N_2}_{k_2}}U^{N_3}_{k_3}\cjg{U^{N_4}_{k_4}}U^{N_5}_{k_5}\cjg{U^{N_6}_{k_6}} 
\end{equation}
still with the abuse of notation $G = G(U)$. The strategy to estimate $\cF_{s,\vec{N}}(U)$ is to expand the formula above in $U=\pi_N \phi + \pi_N V$. When we replace $U \in \hsigt$ by $V \in H^s(\T)$, we gain $\frac{1}{2}$-derivative; when we replace $U$ by $\phi$, we gain randomness. However, we are limited to expanding the formula~\eqref{dyadic formula for FsN} until we have at most two $V$ terms, in order to ultimately remain controlled by the counter-part-term $-\frac{1}{2} \norm{V}_{H^s}^2$ in the Boué-Dupuis formula. More concretely, this limitation allow us to obtain a power less than 2 on $\norm{V}_{H^s}$ in~\eqref{estimate Fs dyadic block}; and this point was indeed essential in the proof of Proposition~\ref{prop exp integrability Fs,N} assuming Lemma~\ref{lem exp integrability Rs}. \\

The gains obtained expanding~\eqref{dyadic formula for FsN} are measured in terms of negative powers of the $N_j$. When we replace $U^{N_j}$ by $V^{N_j}$, we easily see the benefit from:
\begin{align*}
	\norm{U^{N_j}}_{L^2} &\lesssim N_j^{\frac{1}{2}-s+} \hsignorm{U}, & \norm{V^{N_j}}& \lesssim N_j^{-s} \norm{V}_{H^s}
\end{align*}
When we deal with the $\phi^{N_j}$, we benefit from the independence of the Gaussian random variables $(g_n)_{n \in \Z}$ through Lemma~\ref{lem L2 estimate with orthogonality}, and from the counting bounds from Lemma~\ref{lem counting bound 3 integers}, relying on the presence of constraints on the resonant function $\Omega$ (see the proof of Corollary~\ref{cor psi estimate and counting bound}). Note also that the higher the $N_j$ are, the better our estimates will be.\\

In our analysis, the counting bounds in Lemma~\ref{lem counting bound 3 integers} play a key role, and this will force us to not expand in $U$ inside the term $G^{N_1}$, where $G(U)=\pi_N (|U|^4U)$ (which corresponds to the most difficult case). More precisely, since the Fourier transform converts the product into convolution, we have:
\begin{equation*}
	(|U|^4U)_{k_1} =  \sum_{\lincba} U_{p_1} \cjg{U_{p_2}}...U_{p_5}
\end{equation*}
and wee see that this sum does not contain any constraint on the resonant function $\Omega$, which prevents us from applying the counting bounds from Lemma~\ref{lem counting bound 3 integers}. For that reason, we will not expand in $U$ inside $G^{N_1}$ and we will treat this term  purely deterministically, using Lemma~\ref{lem localized L2 estimate on |u|4u}. \\
We emphasize that similar strategies have been performed recently in~\cite{coe2024sharp} and~\cite{forlano2025improvedquasiinvarianceresultperiodic}.
\subsection{Proof of the dyadic estimate}

\begin{proof}[Proof of Lemma~\ref{lem exp integrability Rs}]   Let $s>\frac{9}{10}$, $N \in \N$, and $\vec{N}=(N_1,...,N_6) \in (2^\N)^6$. Recall that we use the notation:
	\begin{equation*}
		U := \pi_N \phi + \pi_N V
	\end{equation*}
	 Again, as in Definition~\ref{def dyadic block} and Notation~\ref{notn fromula dyadic block}, we consider $(\Psi_{2s},G) \in \{\Psi^{(0)}_{2s} ,\Psi^{(1)}_{2s}\} \times \{id, F_N\}$ such that: 
	\begin{equation}\label{FsN fromula in the proof}
			\cF_{s,\vec{N}}(U) = \sum_{\linc} \Psi_{2s}(\vec{k}) G^{N_1}_{k_1} \cjg{U^{N_2}_{k_2}}U^{N_3}_{k_3}\cjg{U^{N_4}_{k_4}}U^{N_5}_{k_5}\cjg{U^{N_6}_{k_6}} 
	\end{equation}
	still with the abuse of notation $G = G(U)$.\\
	
	We consider the worst case when $N_1 = N_{(1)}$. Then, without loss of generality, we assume that:
	\begin{align*}
		N_{(2)} &= N_2 & &\textnormal{and} & N_{(3)} &= N_3 & &\textnormal{and} & 	N_{(4)} &= N_4 
	\end{align*}
	The other cases are similar (when the $N_{(2)},N_{(3)},N_{(4)}$ does not have the same signature) or slightly simpler (when the $N_{(2)},N_{(3)},N_{(4)}$ have the same signature). Moreover, we assume $N_1 \sim N_2$. Indeed, the other case, $N_1 \ll N_2$, corresponds to a zero contribution in~\eqref{FsN fromula in the proof} due to the constraint $\linc$. Thus, we will consistently replace $N_2$ with $N_1$ in our estimates.\\
	
	We distinguish between two frequency regimes: when $N_{4} \sim N_{3}$ and when $N_{4} \ll N_{3}$. Hence, we write:
	\begin{equation*}
		\cF_{s,\vec{N}}(U) = \1_{N_4 \ll N_3} \cF_{s,\vec{N}}(U) + \1_{N_4 \sim N_3}\cF_{s,\vec{N}}(U)
	\end{equation*}
	and we treat separately both terms. \\
	
	\textbullet \textbf{The first regime:} When $N_4 \ll N_3$. In this case, we estimate~\eqref{FsN fromula in the proof}
	purely deterministically. Here, we are in position to apply the refined deterministic estimate~\eqref{L6 refined deterministic estimate} from Lemma~\ref{lem refined deterministic estimate via Strichartz's trick}. Doing so leads to:
	\begin{equation*}
		|\cF_{s,\vec{N}}(U)| \lesssim_s (N_{1}^{2(s-1)} \vee N_{3}^{2(s-1)}) (N_{3}N_{4}N_{5}N_{6})^{\frac{1}{2}} \norm{G^{N_1}}_{L^2(\T)}\prod_{j=2}^6\norm{U^{N_j}}_{L^2(\T)}
	\end{equation*}
	Combining this with~\eqref{localized L2 estimate on |u|4u} yields:
	\begin{equation*}
		\begin{split}
			|\cF_{s,\vec{N}}(U)| \1_{B_R}(U) &\lesssim_{s,R} (N_{1}^{2(s-1)} \vee N_{3}^{2(s-1)}) N_1^{n_s+\frac{1}{2}-s +} (N_{3}N_{4}N_{5}N_{6})^{1-s} \\ &\lesssim
			\begin{cases}
				N_1^{2(s-1)+2(\frac{1}{2}-s)+} = N_1^{-1+} \hspace{0.3cm} \textnormal{for $s>1$} \\
				N_1^{5-6s}N_3^{2(s-1) + 4(1-s)} \hspace{0.3cm} \textnormal{for $s\leq 1$}
			\end{cases}
		\end{split}
	\end{equation*}
	We deduce that:
	\begin{equation*}
		|\cF_{s,\vec{N}}(U)| \1_{B_R}(U) \lesssim_{s,R} N_1^{0-}
	\end{equation*}
	Indeed, this is true for $s>1$;  when $s \leq 1$, we have $5-6s>0$ (since $s>\frac{9}{10}$), so that we can write:
	\begin{center}
		$N_1^{5-6s}N_3^{2(s-1) + 4(1-s)} \leq N_1^{0-} \iff 5 -6s + 2(s-1) + 4(1-s)<0 \iff s>\frac{7}{8}$
	\end{center}
	and we have $\frac{7}{8}<\frac{9}{10}$. In particular, we have obtained, when $N_4 \ll N_3$, that $\cF_{s,\vec{N}}(U)$ satisfies the bound in~\eqref{estimate Fs dyadic block}.\\
	
	\textbullet \textbf{The second regime:} When $N_4 \sim N_3$. In this situation, treating $\cF_{s,\vec{N}}$ purely deterministically is not enough. Following the strategy of Paragraph~\ref{par strategy and remarks}, we expand in $U$ the formula~\eqref{FsN fromula in the proof}. But before doing so, we decompose $\cF_{s,\vec{N}}(U)$ into a non-pairing contribution and a pairing contribution as follows:
	\begin{equation}\label{Fs = non pairing + pairing}
		\begin{split}
			\cF_{s,\vec{N}}(U) &= \sum_{\substack{\linc \\ k_3 \neq k_2,k_4}} \Psi_{2s}(\vec{k}) G^{N_1}_{k_1}\cjg{U^{N_2}_{k_2}}...\cjg{U^{N_6}_{k_6}} + \sum_{\substack{\linc \\ k_3 \in \{k_2,k_4 \} }} \Psi_{2s}(\vec{k}) G^{N_1}_{k_1}\cjg{U^{N_2}_{k_2}}...\cjg{U^{N_6}_{k_6}}
			\\& =: \frkF_{s,\vec{N}}(G,U,...,U) + \cF_{s,\vec{N}}^{\downarrow}(G,U,...,U)
		\end{split}
	\end{equation}
	where $\frkF_{s,\vec{N}}$ and $\cF^\downarrow_{s,\vec{N}}$ are the multi-linear maps induced, respectively, by the left sum and the right sum in~\eqref{Fs = non pairing + pairing}. For the term $\frkF_{s,\vec{N}}$, we will exploit the independence of the Gaussian $g_{k_3}$ from $g_{k_2}$ and $g_{k_4}$; and the term $\cF^\downarrow_{s,\vec{N}}$ will simplify as a degenerate version of $\frkF_{s,\vec{N}}$.\\
	
	\textbf{The non-pairing contribution :} This contribution corresponds to the term $\frkF_{s,\vec{N}}$ in~\eqref{Fs = non pairing + pairing}. Here, we establish the bound in~\eqref{estimate Fs dyadic block} for $\frkF_{s,\vec{N}}(G,U,...,U)$. To do so, we start by decomposing $\frkF_{s,\vec{N}}$ thanks to the multi-linearity. More precisely, we expand $\frkF_{s,\vec{N}}$ at the level of the frequencies $N_2,N_3$ and $N_4$ as follows:
	\begin{equation}\label{expanding 3 hf FsN}
		\begin{split}
			\frkF_{s,\vec{N}}&(G,U,U,U,U,U)\\  &=\frkF_{s,\vec{N}}(G,V,V,U,U,U) + \frkF_{s,\vec{N}}(G,V,\phi,V,U,U) + \frkF_{s,\vec{N}}(G,\phi,V,V,U,U) \\
			&+ \frkF_{s,\vec{N}}(G,V,\phi,\phi,U,U) + \frkF_{s,\vec{N}}(G,\phi,V,\phi,U,U) + \frkF_{s,\vec{N}}(G,\phi,\phi,V,U,U) \\ 
			& + \frkF_{s,\vec{N}}(G,\phi,\phi,\phi,U,U) \\ 
			&=: \frkF_{s,\vec{N}}^{(D2)} + \frkF_{s,\vec{N}}^{(D1)} +\frkF_{s,\vec{N}}^\omega 
		\end{split}
	\end{equation}
	where $\frkF_{s,\vec{N}}^{(D2)}$,  $\frkF_{s,\vec{N}}^{(D1)}$ and $\frkF_{s,\vec{N}}^\omega$ correspond respectively to the first, second and third lines of the right term of the equality in~\eqref{expanding 3 hf FsN}. Note that in $\frkF_{s,\vec{N}}^{(D2)}$ we have grouped the terms with two $V$; in $\frkF_{s,\vec{N}}^{(D1)}$ we have grouped the terms with exactly one $V$; in $\frkF_{s,\vec{N}}^\omega$ we have isolated the term without any $V$. In what follows, we study separately these three terms.\\
	
	We recall that we now have $N_3 \sim N_4$, and still $N_1 \sim N_2$. Hence, in the following, we will consistently replace in $N_2$ by $N_1$  and $N_4$ by $N_3$ in our estimates.\\
	
	\textbf{Study of $\frkF_{s,\vec{N}}^{(D2)}$:} Here, we benefit from the presence of two $V \in H^s(\T)$. In $\frkF_{s,\vec{N}}^{(D2)}$ (see~\eqref{expanding 3 hf FsN}), the least favorable term is:
	\begin{equation*}
		\frkF_{s,\vec{N}}(G,\phi,V,V,U,U) =  \sum_{\substack{\linc \\ k_3 \neq k_2,k_4}}\Psi_{2s}(\vec{k}) G^{N_1}_{k_1}\cjg{\phi^{N_2}_{k_2}}V^{N_3}_{k_3}\cjg{V^{N_4}_{k_4}}U^{N_5}_{k_5}\cjg{U^{N_6}_{k_6}}
	\end{equation*}  because the two $V$'s are located at the frequencies $N_3$ and $N_4$ (and not at the higher frequency $N_2\sim N_1$). Let us then provide a detailed analysis for this term only; the analysis for the other two is similar.\\
	Applying~\eqref{L6 deterministic estimate} from Lemma~\ref{lem deterministic estimate via Strichartz's trick} leads to:
	\begin{multline*}
		|\frkF_{s,\vec{N}}(\phi,V,V,U,U,U)| \\ \lesssim  N_{1}^{0+} (N_{1}^{2(s-1)} \vee N_{3}^{2(s-1)}) N_{3}^{2} \norm{G^{N_1}}_{L^2}\norm{\phi^{N_2}}_{L^2} \norm{V^{N_3}}_{L^2} \norm{V^{N_4}}_{L^2} \norm{U^{N_5}}_{L^2} \norm{U^{N_6}}_{L^2} 
	\end{multline*}
	We interpolate $V^{N_4}$ between $\hsig$ and $H^s$ as: 
	\begin{center}
		$\norm{V^{N_4}}_{H^{\theta \sigma +(1-\theta)s}} \leq \hsignorm{V^{N_4}}^\theta \norm{V^{N_4}}_{H^s}^{1-\theta}$ for $\theta \in (0,1)$, 
	\end{center}
	with $\theta \in (0,1)$ close to 0 so that we may write $\theta \sigma +(1-\theta)s = s-$. By using additionally~\eqref{localized L2 estimate on |u|4u} (recalling that $n_s$ is defined in~\eqref{ns}), we deduce that:
	\begin{multline*}
		|\frkF_{s,\vec{N}}(\phi,V,V,U,U,U)|\1_{B_R}(U) \\
		\lesssim_{s,R}  N_{1}^{n_s+ \frac{1}{2}-s+} (N_{1}^{2(s-1)} \vee N_{3}^{2(s-1)}) N_{3}^{2-2s}(N_5N_6)^{\frac{1}{2}-s} \norm{\phi}_{\hsig}  \norm{\pi_N V}_{\hsig}^{\theta}\norm{ V}_{H^s}^{2-\theta}   
	\end{multline*}
	Moreover, since $s>\frac{9}{10}>\frac{1}{2}$, we can crudely use the bound $(N_5N_6)^{\frac{1}{2}-s}\leq 1$. Using the triangular inequality, we also bound $\norm{\pi_N V}_{\hsig}$ by $\hsignorm{U}+ \hsignorm{\phi}$ (recalling that $U=\pi_N \phi + \pi_N V$). Then, from the above inequality we obtain:
	\begin{equation*}
		|\frkF_{s,\vec{N}}(\phi,V,V,U,U,U)|\1_{B_R}(U) 
		\lesssim_{s,R}  N_{1}^{n_s+ \frac{1}{2}-s+} (N_{1}^{2(s-1)} \vee N_{3}^{2(s-1)}) N_{3}^{2-2s} (1+\hsignorm{\phi})^{1+\theta}\norm{V}_{H^s}^{2-\theta}
	\end{equation*}
	Finally, to obtain the bound in~\eqref{estimate Fs dyadic block}, we observe that:
	\begin{equation*}
		 N_{1}^{n_s+ \frac{1}{2}-s+} (N_{1}^{2(s-1)} \vee N_{3}^{2(s-1)}) N_{3}^{2-2s} \leq \begin{cases}
			N_1^{-1+} N_3^{2-2s} \hspace{0.2cm} \textnormal{for $s>1$} \\
			N_1^{5-6s+} \hspace{0.2cm} \textnormal{for $s\leq 1$}
		\end{cases}
		\leq \begin{cases}
			N_1^{-1+} \hspace{0.2cm} \textnormal{for $s>1$} \\
			N_1^{0-} \hspace{0.2cm} \textnormal{for $\frac{5}{6}<s\leq 1$}
		\end{cases}
	\end{equation*}
	The analysis of $\frkF_{s,\vec{N}}(G,V,V,U,U,U)$ and $\frkF_{s,\vec{N}}(G,V,\phi,V,U,U)$ is identical, and in the end we have:
	\begin{equation*}
		|\frkF_{s,\vec{N}}^{(D2)}|\1_{B_R}(U) \lesssim_{s,R} N_{1}^{0-} (1+\hsignorm{\phi})^{1+\theta}\norm{V}_{H^s}^{2-\theta}
	\end{equation*}
	In particular, we have indeed obtained the bound~\eqref{estimate Fs dyadic block} for $\frkF_{s,\vec{N}}^{(D2)}$. \\ 
	
	\textbf{Study of $\frkF_{s,\vec{N}}^{(D1)}$:} Recall that $\frkF_{s,\vec{N}}^{(D1)}$ is defined as the second line of the right term in~\eqref{expanding 3 hf FsN}. Here, we benefit from the presence of one $V$, and from the randomness through the square-root cancellation (see Lemma~\ref{lem L2 estimate with orthogonality}). Our analysis of the three terms composing $\frkF_{s,\vec{N}}^{(D1)}$ is identical, so we only focus, for example, on:
	\begin{equation*}
		\frkF_{s,\vec{N}}(G,\phi,\phi,V,U,U) =  \sum_{\substack{\linc \\ k_3 \neq k_2,k_4}} \Psi_{2s}(\vec{k}) G^{N_1}_{k_1}\cjg{\phi^{N_2}_{k_2}}\phi^{N_3}_{k_3}\cjg{V^{N_4}_{k_4}}U^{N_5}_{k_5}\cjg{U^{N_6}_{k_6}}
	\end{equation*} 
	We rewrite this as:
	\begin{equation*}
		\frkF_{s,\vec{N}}(\phi,\phi,V,U,U,U) =  \sum_{k_1,...,k_6 } C_0(\vec{k})\Psi_{2s}(\vec{k}) G^{N_1}_{k_1}\cjg{\phi^{N_2}_{k_2}}\phi^{N_3}_{k_3}\cjg{V^{N_4}_{k_4}}U^{N_5}_{k_5}\cjg{U^{N_6}_{k_6}}
	\end{equation*} 
	gathering all the constraints in:
	\begin{equation}\label{C(k) constraints for Fs}
		C_0(\vec{k}) := \1_{\linc} \cdot \1_{k_3 \neq k_2,k_4} \cdot \big( \prod_{j=1}^6 \1_{ |k_j| \sim N_j} 
		\1_{ |k_j| \leq N}	\big)
	\end{equation}
	Next, by the Cauchy-Schwarz inequality in the $k_1,k_4,k_5,k_6$ summation, we have:
	\begin{multline}\label{CS for RD1}
		|\frkF_{s,\vec{N}}(G,\phi,\phi,V,U,U)| \\ \leq \Big( \sum_{k_1,k_4,k_5,k_6} |G^{N_1}_{k_1}V^{N_4}_{k_4}U^{N_5}_{k_5}U^{N_6}_{k_6}|^2 \Big)^{\frac{1}{2}} \Big( \sum_{k_1,k_4,k_5,k_6} \big| \sum_{k_2,k_3} C_0(\vec{k}) \Psi_{2s}(\vec{k}) \phi^{N_2}_{k_2}\cjg{\phi^{N_3}_{k_3}} \big|^2 \Big)^{\frac{1}{2}} 
	\end{multline}
	Furthermore (recalling that $N_4 \sim N_3$), 
	\begin{multline*}
		\Big( \sum_{k_1,k_4,k_5,k_6} |G^{N_1}_{k_1}V^{N_4}_{k_4}U^{N_5}_{k_5}U^{N_6}_{k_6}|^2 \Big)^{\frac{1}{2}} \1_{B_R}(U) \leq \norm{G^{N_1}}_{L^2} \norm{V^{N_4}}_{L^2}  \norm{U^{N_5}}_{L^2} \norm{U^{N_6}}_{L^2} \1_{B_R}(U) \\ \lesssim_{s,R} N_1^{n_s+} N_3^{-s} (N_{5}N_{6})^{\frac{1}{2}-s}   \norm{V}_{H^s}
	\end{multline*}
	So, from~\eqref{CS for RD1}, we obtain 
	\begin{equation*}
		|\frkF_{s,\vec{N}}(G,\phi,\phi,V,U,U)|\1_{B_R}(U) \lesssim_{s,R}  X_{\vec{N}} \norm{V}_{H^s}
	\end{equation*}
	with $X_{\vec{N}}$ the non-negative random variable defined as:
	\begin{equation*}
		X_{\vec{N}} :=  N_1^{n_s+} N_3^{-s} (N_{5}N_{6})^{\frac{1}{2}-s} \Big( \sum_{k_1,k_4,k_5,k_6} \big| \sum_{k_2,k_3} C_0(\vec{k})\Psi_{2s}(\vec{k}) \phi^{N_2}_{k_2}\cjg{\phi^{N_3}_{k_3}} \big|^2 \Big)^{\frac{1}{2}}
	\end{equation*}
	Our goal is to prove that $\frkF_{s,\vec{N}}(G,\phi,\phi,V,U,U)$ satisfies the bound in~\eqref{estimate Fs dyadic block}; then, in what follows, we prove that $X_{\vec{N}}$ satisfies the first bound in~\eqref{rv dyadic estimate Rs}. Thanks to the condition $k_2 \neq k_3$, we are in a position to use Lemma~\ref{lem L2 estimate with orthogonality}; therefore we can write:
	\begin{multline*}
		\E[X_{\vec{N}}^2] = [ N_1^{n_s+} N_3^{-s} (N_{5}N_{6})^{\frac{1}{2}-s}]^2 \sum_{k_1,k_4,k_5,k_6}  \E \big[ \big|\sum_{k_2,k_3} C_0(\vec{k})\Psi_{2s}(\vec{k}) \frac{g_{k_2}\cjg{g_{k_3}}}{(\langle k_2 \rangle \langle k_3 \rangle)^s} \big|^2 \big] \\ 
		\lesssim	[ N_1^{n_s+} N_3^{-s} (N_{5}N_{6})^{\frac{1}{2}-s}]^2 \sum_{k_1,k_4,k_5,k_6}   \sum_{k_2,k_3} C_0(\vec{k})\Psi_{2s}(\vec{k})^2 (N_2N_3)^{-2s} 
	\end{multline*}
	Now, using the fact that $N_{1} \sim N_{2}$, and the bound~\eqref{psi estimate + counting bound different signatures} from Corollary~\ref{cor psi estimate and counting bound}, we deduce that:
	\begin{equation}\label{E[X_N^2] to estimate}
		\E[X_{\vec{N}}^2] \lesssim [N_{1}^{n_s+s-1+}N_3^{\frac{3}{2}-2s} (N_{5}N_{6})^{1-s} ]^2
	\end{equation}
	therefore,
	\begin{equation*}
		\E[X_{\vec{N}}^2] \lesssim \begin{cases}
			[N_{1}^{-\frac{1}{2}+}N_3^{\frac{3}{2}-2s}]^2 \hspace{0.3cm} \textnormal{for $s>1$} \\
			[N_{1}^{\frac{7}{2}-4s+}N_3^{\frac{3}{2}-2s + 2(1-s)} ]^2 \hspace{0.3cm} \textnormal{for $s\leq1$}
		\end{cases}
	\end{equation*}
	Let us verify that this indeed implies $\E[X_{\vec{N}}^2]\lesssim N_1^{0-}$ when $s>\frac{9}{10}$. We easily see that this is the case when $s>1$. When $s \leq 1$, we write:
	\begin{equation*}
		N_{1}^{\frac{7}{2}-4s+}N_3^{\frac{3}{2}-2s + 2(1-s)} \leq N_1^{0-} \iff \begin{cases}
			\frac{7}{2}-4s<0 \\
			\frac{7}{2}-4s + \frac{3}{2}-2s + 2(1-s) <0
			\end{cases}
			\iff s>\frac{7}{8}
	\end{equation*}
	Hence, for every $s>\frac{9}{10}$ we have  $\E[X_{\vec{N}}^2]\lesssim N_1^{0-}$. Finally, the analysis for the other two terms  $\frkF_{s,\vec{N}}(G,V,\phi,\phi,U,U)$ and  $\frkF_{s,\vec{N}}(G,\phi,V,\phi,U,U)$ in $\frkF_{s,\vec{N}}^{(D1)}$ is identical, and in the end we have:
	\begin{equation*}
		|\frkF_{s,\vec{N}}^{(D1)}|\1_{B_R}(U) \lesssim_{s,R} \norm{V}_{H^s} X_{\vec{N}}
	\end{equation*}
	with $X_{\vec{N}}$ a non-negative random variable which satisfies the first bound in~\eqref{rv dyadic estimate Rs}. In particular, $\frkF_{s,\vec{N}}^{(D1)}$ satisfies the bound in~\eqref{estimate Fs dyadic block}. \\
	
	\textbf{Study of $\frkF^\omega_{s,\vec{N}}$:} This is the last term in~\eqref{expanding 3 hf FsN}; it is given by:
	\begin{equation*}
		\begin{split}
			\frkF^\omega_{s,\vec{N}} = \frkF_{s,\vec{N}}(G,\phi,\phi,\phi,U,U) &= \sum_{\substack{\linc \\k_3 \neq k_2,k_4}} \Psi_{2s}(\vec{k}) G_{k_1}^{N_1}\cjg{\phi^{N_2}_{k_2}}\phi^{N_3}_{k_3}\cjg{\phi^{N_4}_{k_4}}U^{N_5}_{k_5}\cjg{U^{N_6}_{k_6}} \\
			&=  \sum_{k_1,...,k_6}C_0(\vec{k}) \Psi_{2s}(\vec{k}) G_{k_1}^{N_1}\cjg{\phi^{N_2}_{k_2}}\phi^{N_3}_{k_3}\cjg{\phi^{N_4}_{k_4}}U^{N_5}_{k_5}\cjg{U^{N_6}_{k_6}}
		\end{split}
	\end{equation*}
	where $C_0(\vec{k})$ is the constraint term defined in~\eqref{C(k) constraints for Fs}. Here, we provide a similar analysis to that for $\frkF_{s,\vec{N}}^{(D1)}$. In particular, using the Cauchy-Schwarz inequality in the $k_1,k_5,k_6$ summation, and then the cut-off $\1_{B_R}(U)$ on $\hsigt$ (with $\sigma=s-\frac{1}{2}-$), leads to:
	\begin{equation*}
		\begin{split}
			|\frkF^\omega_{s,\vec{N}}|\1_{B_R}(U) &\leq \Big( \sum_{k_1,k_5,k_6} |G^{N_1}_{k_1}U^{N_5}_{k_5}U^{N_6}_{k_6}|^2 \Big)^{\frac{1}{2}} \Big( \sum_{k_1,k_5,k_6} \big| \sum_{k_2,k_3,k_4} C_0(\vec{k})\Psi_{2s}(\vec{k})  \phi^{N_2}_{k_2}\cjg{\phi^{N_3}_{k_3}}\phi^{N_4}_{k_4} \big|^2 \Big)^{\frac{1}{2}} \\
			&\lesssim_{s,R} Y_{\vec{N}}
		\end{split} 
	\end{equation*} 
	where $Y_{\vec{N}}$ is the non-negative random variable defined as:
	\begin{equation*}
		Y_{\vec{N}} := N_1^{n_s+}(N_{5}N_{6})^{\frac{1}{2}-s}\Big( \sum_{k_1,k_5,k_6} \big| \sum_{k_2,k_3,k_4} C_0(\vec{k})\Psi_{2s}(\vec{k}) \frac{g_{k_1}\cjg{g_{k_2}} g_{k_3}}{\langle k_1\rangle^s \langle k_2\rangle^s \langle k_3\rangle^s} \big|^2 \Big)^{\frac{1}{2}}
	\end{equation*}
	Now, our goal is to prove that $Y_{\vec{N}}$ satisfies the second bound in~\eqref{rv dyadic estimate Rs}. Note that the constraint term $C_0(\vec{k})$ contains the condition $k_3 \neq k_2,k_4$, so we are in a position to use both Lemma~\ref{lem L2 estimate with orthogonality} (with $m=3$) and the bound~\eqref{psi estimate + counting bound different signatures} from Corollary~\ref{cor psi estimate and counting bound}. Starting with the Cauchy-Schwarz inequality, we can therefore write:
	\begin{multline*}
			\E[Y_{\vec{N}}] \leq N_1^{n_s+}(N_{5}N_{6})^{\frac{1}{2}-s} \Big( \sum_{k_1,k_5,k_6} \E\big[ \big| \sum_{k_2,k_3,k_4} C_0(\vec{k})\Psi_{2s}(\vec{k}) \frac{g_{k_2}\cjg{g_{k_3}} g_{k_4}}{\langle k_2\rangle^s \langle k_3\rangle^s \langle k_4\rangle^s} \big|^2\big] \Big)^{\frac{1}{2}} \\
			 \lesssim N_1^{n_s+}(N_{5}N_{6})^{\frac{1}{2}-s} N_1^{-s}N_3^{-2s} \big(\sum_{k_1,...,k_6} C_0(\vec{k})\Psi_{2s}(\vec{k})^2 \big)^{\frac{1}{2}} \lesssim N_1^{n_s+s-1+}N_3^{\frac{3}{2}-2s} (N_{5}N_{6})^{1-s}
	\end{multline*}
	Reasoning exactly as we did for $X_{\vec{N}}$ in~\eqref{E[X_N^2] to estimate}, we deduce that for every $s> \frac{9}{10}$, the above inequality implies $E[Y_{\vec{N}}]\lesssim N_1^{0-}$. Hence, we have proven that:
	\begin{equation*}
		|\frkF^\omega_{s,\vec{N}}|\1_{B_R}(U)\lesssim_{s,R} Y_{\vec{N}}
	\end{equation*}
	with $Y_{\vec{N}}$ satisfying the second bound in~\eqref{rv dyadic estimate Rs}. In particular, $\frkF^\omega_{s,\vec{N}}\1_{B_R}(U)$ satisfies the bound in~\eqref{estimate Fs dyadic block}. \\
	To sum up, we have established that:
	\begin{equation*}
		|\frkF_{s,\vec{N}}\1_{B_R}(U)| \leq |\frkF^{(D2)}_{s,\vec{N}}\1_{B_R}(U)| + |\frkF^{(D1)}_{s,\vec{N}}\1_{B_R}(U)| + |\frkF^\omega_{s,\vec{N}}\1_{B_R}(U)| \leq \textnormal{RHS of~\eqref{estimate Fs dyadic block}}
	\end{equation*}
	Now we recall the non-pairing/pairing decomposition in~\eqref{Fs = non pairing + pairing}. Thus, in order to obtain~\eqref{estimate Fs dyadic block} for $\cF_{s,\vec{N}}$, it remains to study the paring contribution $\cF_{s,\vec{N}}^{\downarrow}$. \\
	\textbf{The pairing contribution :} We can rewrite $\cF_{s,\vec{N}}^{\downarrow}$ as:
	\begin{equation*}
		\cF_{s,\vec{N}}^{\downarrow} =  \sum_{\substack{\linc \\  k_3 = k_2 }} \Psi_{2s}^\downarrow(\vec{k}) G^{N_1}_{k_1}\cjg{U^{N_2}_{k_2}}...\cjg{U^{N_6}_{k_6}} +  \sum_{\substack{\linc  \\ k_3 = k_4, \hsp k_3\neq k_2 }} \Psi_{2s}^\downarrow(\vec{k}) G^{N_1}_{k_1}\cjg{U^{N_2}_{k_2}}...\cjg{U^{N_6}_{k_6}}
	\end{equation*}
	which degenerates as:
	\begin{multline}\label{Fs,N degenerated}
		\cF_{s,\vec{N}}^{\downarrow}	= \1_{N_2 = N_3} \norm{U^{N_2}}_{L^2}^2 \sum_{k_1-k_4+k_5-k_6=0 } \Psi_{2s}^\downarrow(\vec{k})G_{k_1}^{N_1} \cjg{U^{N_4}_{k_4}}U^{N_5}_{k_5}\cjg{U^{N_6}_{k_6}} \\ 
		+ \1_{N_3 = N_4} \sum_{k_1-k_2+k_5-k_6=0} \Psi_{2s}^\downarrow(\vec{k}) G^{N_1}_{k_1} W_{k_2}\cjg{U^{N_2}_{k_2}}U^{N_5}_{k_5}\cjg{U^{N_6}_{k_6}} 
	\end{multline}
	with $W_{k_2} = \sum_{k_3 \neq k_2} |U^{N_3}_{k_3}|^2$, and where $\Psi_{2s}^\downarrow \in \{ \Psi_{2s}^{(0),\downarrow},\Psi_{2s}^{(1),\downarrow}\}$ is the degenerated version of $\Psi_{2s} \in \{ \Psi_{2s}^{(0)},\Psi_{2s}^{(1)}\}$ defined in Definition~\ref{def Big Psi}. The important point here is that there is no pairing in the sums above. Indeed, when
	\begin{center}
		 $\Psi_{2s}^\downarrow(\vec{k}) = \Psi_{2s}^{(0),\downarrow}(\vec{k})= \1_{\Omega^\downarrow(\vec{k}) =0} \psi_{2s}^\downarrow(\vec{k}) $
	\end{center}
	the paring contributions have a zero-contribution,
	because if, say, in the first sum in~\eqref{Fs,N degenerated} we have $k_5=k_6$, then $k_1=k_4$ and we deduce that $\psi_{2s}^\downarrow(\vec{k})=0$, that is $\Psi_{2s}^\downarrow(\vec{k})=0$. Similarly, if, say, in the second sum in~\eqref{Fs,N degenerated} we have $k_2=k_5$, then $k_1=k_6$, and we deduce that $\psi_{2s}^\downarrow(\vec{k})=0$.  On the other hand, when
		\begin{center}
		$\Psi_{2s}^\downarrow(\vec{k}) = \Psi_{2s}^{(1),\downarrow}(\vec{k}) = \1_{\Omega^\downarrow(\vec{k}) \neq 0} \frac{\psi_{2s}^\downarrow(\vec{k})}{\Omega^\downarrow(\vec{k})} $,
	\end{center}
	 then, thanks to the factorization of $\Omega^\downarrow$ from Lemma~\ref{lem factorization dgn Omega}, and the constraint $\Omega^\downarrow(\vec{k}) \neq 0$, a pairing contribution cannot happen in the sums above. Hence, we can analyze the terms:
	\begin{align*}
		 &\sum_{k_1-k_4+k_5-k_6=0 } \Psi_{2s}^\downarrow(\vec{k})G_{k_1}^{N_1} \cjg{U^{N_4}_{k_4}}U^{N_5}_{k_5}\cjg{U^{N_6}_{k_6}}, & & \sum_{k_1-k_2+k_5-k_6=0} \Psi_{2s}^\downarrow(\vec{k}) G^{N_1}_{k_1} W_{k_2}\cjg{U^{N_2}_{k_2}}U^{N_5}_{k_5}\cjg{U^{N_6}_{k_6}} 
	\end{align*}
	with the same method as the one for $\frkF_{s,\vec{N}}$. The situation is even more favorable because in $\cF_{s,\vec{N}}^\downarrow$ we additionally benefit from extra negative powers of $N_1$ or $N_3$, since:
	\begin{align*}
		\norm{U^{N_2}}_{L^2}^2\1_{B_R}(U)& \lesssim_{R} N_1^{1-2s+}  & &\textnormal{and} &  |W_{k_2}|\1_{B_R}(U) \leq \norm{U^{N_3}}_{L^2}^2\1_{B_R}(U)\lesssim_{R} N_3^{1-2s+} 
	\end{align*}
	In the end, $\cF^\downarrow_{s,\vec{N}}$ also satisfies the bound in~\eqref{estimate Fs dyadic block}. To summarize, we have proven that:
	\begin{equation*}
		|\cF_{s,\vec{N}}(U)| \1_{B_R}(U) \leq |\frkF_{s,\vec{N}}(U)| \1_{B_R}(U) + |\cF^\downarrow_{s,\vec{N}}(U)| \1_{B_R}(U) \leq \textnormal{RHS of~\eqref{estimate Fs dyadic block}},
	\end{equation*} 
	and the proof of Lemma~\ref{lem exp integrability Rs} is complete.
\end{proof}

%% file: appendix_approximation_of_the_flow_by_the_truncated_flow.tex
\subsection{Construction of the flow in the Bourgain spaces}

We introduce the spaces we will use to construct the flow of~\eqref{NLS}. These are due to Bourgain~\cite{Bourgain1993}, see also~\cite{ginibre1995probleme}.
\begin{defn}[Bourgain spaces]
	Let $\sigma, b \in \R$. We define the space $X^{\sigma,b}(\R \times \T)$ as the subspace of the tempered distribution $\cS'(\R \times \T)$ given by
	\begin{equation*}
		X^{\sigma,b}(\R \times \T) := \{ u \in \cS'(\R \times \T) \hsp: \hsp \langle \tau + n^2 \rangle^b \langle n \rangle^\sigma \cF_{t,x\ra \tau,n }(u) \in L^2l^2(\R \times \Z ) \}
	\end{equation*}
		where $\cF_{t,x\ra \tau,n }$ is the space-time Fourier transform. Equipped with the norm:
	\begin{equation*}
		\norm{u}_{X^{\sigma,b}} := \big( \sum_{n\in \Z} \int_\R \langle \tau +n^2 \rangle^{2b} \langle n \rangle^{2\sigma}|\cF_{t,x\ra \tau,n }(u)|^2 d\tau \big)^{\frac{1}{2}}
	\end{equation*}
	$X^{\sigma,b}(\R \times \T)$ is a Banach space in which the Schwartz functions $\cS(\R \times \T)$ are dense. Moreover, if $I$ is an open interval, we define the restriction space $X^{\sigma,b}_{I}$ as the subspace of the distribution space $\cD'(I \times \T)$ given by:
	\begin{equation}\label{Xsigma,b I}
		X^{\sigma,b}_I := \{ u|_{I} \hsp : u \in X^{\sigma,b}(\R \times \T)\hsp \} 
	\end{equation}
	where $u|_{I}$ denotes the restriction of the distribution $u$ on $I \times \T$. This space, endowed with the restriction norm:
	\begin{equation*}
		\norm{u}_{X^{\sigma,b}_I} := \inf \{ \norm{\tld{u}}_{X^{\sigma,b}} \hsp : \hsp \tld{u} \in X^{\sigma,b}(\R \times \T), \hsp \tld{u}|_I = u\}
	\end{equation*}
	is a Banach space in which the subspace $S(I\times \T):=\{u|_I \hsp : \hsp u \in \cS(\R \times \T)\}$ is dense. \\
	Finally, we define the local Bourgain space $X^{\sigma,b}_{loc}(I \times \T)$ as the set: 
	\begin{equation*}
		X^{\sigma,b}_{loc}(I \times \T) := \{ u \in \cD'(I \times \T) \hsp : \hsp \forall J \Subset I, \hsp u|_J \in X^{\sigma,b}_J \}
	\end{equation*}
	where $J \Subset I$ means that $J$ is an open interval such that $\bar{J} \subset I$. We also denote:
	\begin{equation*}
		X^{\sigma,b}_{loc} := \bigcup_{I} \hsp X_{loc}^{\sigma,b}(I \times \T)
	\end{equation*}
	where the union is taken over the open intervals $I$ that contain 0.
\end{defn}

\begin{prop}\label{prop embedding X sigma b into CHsig}
	Let $\sigma \in \R$ and $b>\frac{1}{2}$. Let $I=(a,b)$ an open interval. Then, we have the continuous embedding:
	\begin{equation*}
		X^{\sigma,b}_{(a,b)} \hookrightarrow \cC([a,b];\hsigt)
	\end{equation*}
	where $\cC([a,b];\hsigt)$ is equipped with the sup-norm. As a consequence, we also have:
	\begin{equation*}
			X^{\sigma,b}_{loc}(I \times \T) \hookrightarrow \cC(I;\hsigt)
	\end{equation*}
\end{prop}

\begin{prop}[Linear estimate]\label{prop lin estimate} Let $\sigma,b \in \R$. Let $I$ an open interval and $t_0 \in I$. Then,
	\begin{equation*}
		\|e^{i(t-t_0)\p_x^2}u_0 \|_{X^{\sigma,b}_I} \lesssim \hsignorm{u_0}
	\end{equation*}
	
\end{prop}

\begin{prop}[Non-linear estimates, see for example \cite{book_tzirakis}]\label{prop non-lin estimates} Let $\sigma > 0$ and $b>\frac{1}{2}$ close to $\frac{1}{2}$. Let $I$ an open interval. Then,
	\begin{equation}\label{non-linear estimate}
		\norm{F(u)}_{X^{\sigma,b-1}_I} \lesssim |I|^{0+} \norm{u}_{X^{\sigma,b}_I}^5
	\end{equation}
	where $|I|$ is the length of $I$. As a consequence, for every $t_0 \in I$,
	\begin{equation*}
		\norm{\int_{t_0}^t e^{i(t-\tau)\p_x^2}F(u(\tau)) d\tau}_{X^{\sigma,b}_I} \lesssim |I|^{0+} \norm{u}_{X^{\sigma,b}_I}^5
	\end{equation*}
	Moreover, the same estimates  hold replacing $F$ by $F_N$, for every $N \in \N$ (and the constants in the inequalities are independent of $N \in \N$). 
\end{prop}

\begin{defn}[Linear and Duhamel operators]\label{def lin op and Duhamel op}
Let $I$ an open interval and $t_0 \in I$. Thanks to Propositions~\ref{prop lin estimate} and~\ref{prop non-lin estimates}, we can introduce the following operators\footnote{We purposely omit the dependence on the interval $I$, writing $S_{t_0}, \cD_{t_0}$ instead of $S_{t_0,I}, \cD_{t_0,I}$, to avoid unnecessary complexity in the notations.}:
\begin{align*}
	&\begin{split}
		S_{t_0} : \hsp & \hsigt \ra X^{\sigma,b}_I \\
		& u_0 \longmapsto e^{i(t-t_0)\p_x^2}u_0
	\end{split} & &\textnormal{and,} &
	& \begin{split}
		\cD_{t_0} : \hsp & X^{\sigma,b}_I \lra X^{\sigma,b}_I \\
		& u \longmapsto -i \int_{t_0}^t e^{i(t-\tau)\p_x^2}F(u(\tau)) d\tau 
	\end{split}
\end{align*}
Similarly, we define $\cD_{t_0}^N$ as $\cD_{t_0}$ replacing $F$ by $F_N$. Moreover, when $t_0 = 0$, we simply write $S$, $\cD$, and $\cD^N$ instead of $S_{t_0}$, $\cD_{t_0}$, and $\cD^N_{t_0}$.
\end{defn}

\begin{rem}
	When we consider $u \in \cC(I; \hsigt)$, with $\sigma>0$ close to $0$, it is not clear how to give a meaning to $|u|^4u$, even in the sense of distributions. Consequently, the notion of solution for~\eqref{NLS} becomes unclear for functions in the class $\cC_t \hsig$. For that reason, we look for solutions in the smaller class $X^{\sigma,b}_{loc}(I \times \T)$ ($b>\frac{1}{2}$), where the non-linearity is well-defined as a distribution thanks to~\eqref{non-linear estimate} (combined with the gluing principle for distributions).
\end{rem}
	
\begin{defn}
		We say that a pair $(u,I)$ is a solution to~\eqref{NLS} in the class $X^{\sigma,b}_{loc}$, where $I$ is an open interval that contains $0$, and $u \in X^{\sigma,b}_{loc}(I \times \T)$,  if:
		\begin{enumerate}
			\item $u(0) = u_0 \in \hsigt$
			\item $u$ satisfies the equation $i \p_t u + \p_x^2u = |u|^4u$ in the sense of distributions.
		\end{enumerate} 
		Equivalently, $(u,I)$ is a solution to~\eqref{NLS} in the class $X_{loc}^{\sigma,b}$, if $u \in X^{\sigma,b}_{loc}(I \times \T)$, and for all $J \Subset I$, the following equality:
		\begin{equation*}
			u = S(u_0) + \cD(u)
		\end{equation*}
		holds in $X^{\sigma,b}_J$.
\end{defn} 

\begin{rem}
	For the equation~\eqref{truncated equation}, the non-linearity $\pi_N (|\pi_N u|^4 \pi_N u)$ is well defined as an element of $\cC_t \hsig$ for any $u \in \cC_t \hsig$. The solutions of~\eqref{truncated equation} factorized as in Proposition~\ref{prop factorization of the truncated flow}; they are global thanks to the $L^2$-norm conservation, and unique in the class $\cC_t \hsig$. Hence, the use of the $X^{\sigma,b}$ spaces is not necessary in order to find solutions of~\eqref{truncated equation}. However, we will run a fixed point argument in the space $X^{\sigma,b}$ for both~\eqref{NLS} and~\eqref{truncated equation}. Indeed, this will first prove the existence and the uniqueness of solutions to~\eqref{NLS}, and secondly, it will help us to prove approximation properties of $\Phi$ (the flow of~\eqref{NLS}) by $\Phi^N$ (the flow of~\eqref{truncated equation}). 
	 Note also that the solutions of~\eqref{truncated equation} in the class $X^{\sigma,b}_{loc}$ coincide with the solutions of~\eqref{truncated equation} in the class $\cC_t \hsig$.
\end{rem}

From the linear and non-linear estimates in Propositions~\ref{prop lin estimate} and~\ref{prop non-lin estimates}, we classically obtain:

\begin{prop}[Local theory]\label{prop local theory} Let $\sigma>0$ and $b > \frac{1}{2}$ close to $\frac{1}{2}$. \\
	--For every $R_0 > 0$, there exists $R>R_0$ and $\delta >0$ such that for every $\hsignorm{u_0} \leq R_0$, every $t_0 \in \R$, and every $N \in \N$, the maps:
	\begin{align}\label{duhamel maps}
		&\begin{split}
			\Gamma^{t_0}_{u_0}: \hsp &B^{t_0}_R(\delta) \lra B^{t_0}_R(\delta) \\
			& u \longmapsto S_{t_0}(u_0) + \cD_{t_0}(u)
		\end{split} & &\textnormal{and} & &\begin{split}
		\Gamma^{t_0,N}_{u_0}: \hsp &B^{t_0}_R(\delta) \lra B^{t_0}_R(\delta) \\
		& u \longmapsto S_{t_0}(u_0) + \cD_{t_0}^N(u)
		\end{split}
	\end{align} 
	are contractions with a universal contraction coefficient $\gamma \in (0,1)$, and where $B^{t_0}_R(\delta)$ is the closed centered ball of radius $R$ in $X^{\sigma,b}_{(t_0 -\delta , t_0 + \delta)}$. \\
	
	-- As a consequence, for every $u_0 \in \hsigt$, there exists a unique maximal solution $u \in X^{\sigma,b}_{loc}(I_{max}(u_0) \times \T)$ to~\eqref{NLS}, in the class $X_{loc}^{\sigma,b}$, that we denote $\Phi(t)u_0$, and where $I_{max}(u_0)$ is an open interval containing 0. Moreover, if $I_{max}(u_0)$ is strictly included in $\R$, then:
	\begin{equation*}
		\hsignorm{\Phi(t)u_0} \lra + \infty
	\end{equation*}  
	as $t \ra \p I_{max}(u_0)$ (the boundary of the interval $I_{max}(u_0)$).
\end{prop}

In \cite{Li_Wu_Xu_global}, the authors proved that the $\hsig$-norm of the local solutions from this proposition do not blow up in finite time when $\sigma > \frac{2}{5}$. More precisely :

\begin{thm}[see \cite{Li_Wu_Xu_global}]\label{thm GWP} Let $\sigma > \frac{2}{5}$. There exists a continuous function $f : \R \ra \R^+$ such that for every $R>0$, there exists a constant $C>0$ such that for every $\hsignorm{u_0} \leq R$,
	\begin{equation*}
	\hsignorm{\Phi(t)u_0} \leq C f(t), \hspace{0.3cm} t \in I_{max}(u_0)
	\end{equation*}
	Consequently, $I_{max}(u_0) = \R$, that is the solutions of~\eqref{NLS} are global, and for every $R>0$, $T>0$, there exists $\Lambda(R,T)>0$ such that:
	\begin{equation}\label{Phi B included B}
		\Phi(t)(B_R) \subset B_{\Lambda(R,T)}
	\end{equation}
	for any $|t| \leq T$.
\end{thm}

\subsection{Approximation of the flow by the truncated flow} This paragraph is dedicated to the proof of Proposition~\ref{prop approx Phi by PhiN set version}. We recall that we denote $\Phi$ and $\Phi^N$ the flow of~\eqref{NLS} and~\eqref{truncated equation} respectively. To prove Proposition~\ref{prop approx Phi by PhiN set version}, we rely on the following lemma:

\begin{lem}\label{lem approx phi by phiN}
	Let $\sigma > \frac{2}{5}$. Let $T>0$ and $K \subset \hsigt$ a compact set. Then,
	\begin{equation}\label{approx phi by phiN}
		\sup_{|t| \leq T} \hsp  \sup_{u_0 \in K} \hsp \hsignorm{\Phi(t)u_0 - \Phi^N(t)u_0} \tendsto{N \ra \infty} 0
	\end{equation}
\end{lem}

First, we show how this lemma implies Proposition~\ref{prop approx Phi by PhiN set version}, and we prove it afterwards. 

\begin{proof}[Proof of Proposition~\ref{prop approx Phi by PhiN set version} assuming Lemma~\ref{lem approx phi by phiN}] Let $\eps>0$. \\
	-- Firstly, we prove~\eqref{PhiN K included PhiK + B}. Let $T>0$. By~\eqref{approx phi by phiN}, there exists $N_{\eps,K,T} \in \N$ such that for every $N \geq N_{\eps,K,T}$:
	\begin{equation*}
			\sup_{|t| \leq T} \hsp  \sup_{u_0 \in K} \hsp \hsignorm{\Phi(t)u_0 - \Phi^N(t)u_0} \leq \eps 
	\end{equation*}
	As a consequence, the equality
	\begin{equation*}
		\Phi^N(t)u_0 = \Phi(t)u_0 + (\Phi^N(t)u_0 - \Phi(t)u_0)
	\end{equation*}
	for all $u_0 \in K$ and $|t|\leq T$ implies~\eqref{PhiN K included PhiK + B}. \\
	
	-- Let us now prove~\eqref{PhiK included PhiN(K+B)}. By continuity of $\Phi(t)$, the set $\Phi(t) K$ is compact in $\hsigt$. Applying the first point~\eqref{PhiN K included PhiK + B} to this set, we know that there exists $N_{\eps,K,t} \in \N$ such that for every $N \geq N_{\eps,K,t}$:
	\begin{equation*}
		\Phi^N(-t) (\Phi(t)K) \subset \Phi(-t)\Phi(t)K + B_{\eps}^{\hsig} = K + B_{\eps}^{\hsig}
	\end{equation*}
	Applying now $\Phi^N(t) = \Phi^N(-t)^{-1}$ to this inclusion leads to~\eqref{PhiK included PhiN(K+B)}.
\end{proof}

In the remaining part of this paragraph, we prove Lemma~\ref{lem approx phi by phiN}. To that end, we first recall a classical result from functional analysis.

  \begin{lem}\label{lem cvgce on compact set of bdd lin map}
	Let $E,F$ two Banach spaces. Let $\{ T_N\}_{N \in \N}$ be a sequence of bounded linear maps with the uniform (in $N \in \N$) bound $\norm{T_N}_{E \ra F} \leq M$, for some constant $M>0$. If for every $u \in E$:
	\begin{equation*}
		\norm{T(u)-T_N(u)}_F \tendsto{N \ra \infty} 0
	\end{equation*}
	for some linear map $T : E \ra F$, then for every compact set $K \subset E$, we have:
	\begin{equation*}
		\sup_{u \in K} \norm{T(u)-T_N(u)}_F \tendsto{N \ra \infty} 0
	\end{equation*}
\end{lem}

\begin{proof}[Proof of Lemma \ref{lem approx phi by phiN}]  
	The proof proceeds in two steps. First we prove~\eqref{approx phi by phiN} locally in time, and then, by iteration, we establish the convergence up to time $T$. First of all, let $R_0>0$ such that $K \subset B_{R_0}$. Then, we set:
	\begin{equation}\label{Lambda0}
		\Lambda_0 := 1+ \sup_{|t| \leq T} \sup_{u_0 \in B_{R_0}} \hsignorm{\Phi(t) u_0} 
	\end{equation}
	which is finite from Theorem~\ref{thm GWP}.
	Then, tanks to Proposition~\ref{prop local theory}, we invoke $R> \Lambda_0$ and $\delta>0$ such that for any $u_0 \in B_{\Lambda_0}$, the maps in~\eqref{duhamel maps} are contractions, with a universal contraction coefficient $\gamma \in (0,1)$. Now, we fix $u_0 \in K$. We denote respectively $\Phi_{u_0}$ and $\Phi^N_{u_0}$ the solutions of~\eqref{NLS} and~\eqref{truncated equation} with initial data $u_0$; these functions evaluated at a time $t$ are interchangeably denoted $\Phi_{u_0}(t)$ or $\Phi(t)u_0$ (and $\Phi^N_{u_0}(t)$ or $\Phi^N(t)u_0$). Moreover we simply write $X_I$ instead of $X^{\sigma,b}_I$ for convenience. \\
	
	\underline{Local-in-time convergence:} Recall that $\Phi_{u_0}$ and $\Phi^N_{u_0}$ are respectively given by the fixed point of $\Gamma_{u_0} : B^{0}_R(\delta) \ra B^{0}_R(\delta)$ and $\Gamma^N_{u_0} : B^{0}_R(\delta) \ra B^{0}_R(\delta)$, defined in~\eqref{duhamel maps} (with $t_0 = 0$). Then, we can write:
	\begin{multline*}
		\| \Phi_{u_0} - \Phi^N_{u_0} \|_{X_{(0,\delta)}} = \| \Gamma_{u_0}(\Phi_{u_0}) - \Gamma_{u_0}^N(\Phi^N_{u_0}) \|_{X_{(0,\delta)}} \\ 
		\leq \| \Gamma_{u_0}(\Phi_{u_0}) - \Gamma_{u_0}^N(\Phi_{u_0}) \|_{X_{(0,\delta)}} + \| \Gamma^N_{u_0}(\Phi_{u_0}) - \Gamma_{u_0}^N(\Phi^N_{u_0}) \|_{X_{(0,\delta)}}
	\end{multline*}
	On the one hand, by contraction we have $\| \Gamma^N_{u_0}(\Phi_{u_0}) - \Gamma_{u_0}^N(\Phi^N_{u_0}) \|_{X_{(0,\delta)}} \leq \gamma 	\| \Phi_{u_0} - \Phi^N_{u_0} \|_{X_{(0,\delta)}} $, so that we obtain from the inequality above:
	\begin{equation*}
		\| \Phi_{u_0} - \Phi^N_{u_0} \|_{X_{(0,\delta)}} \leq C_\gamma \| \Gamma_{u_0}(\Phi_{u_0}) - \Gamma_{u_0}^N(\Phi_{u_0}) \|_{X_{(0,\delta)}} = C_\gamma \| \cD(\Phi_{u_0}) - \cD^N(\Phi_{u_0}) \|_{X_{(0,\delta)}} 
	\end{equation*}
	(with $C_\gamma = \frac{1}{1-\gamma}$). Next, from the fact that $F_N(u) = \pi_N F(\pi_N u)$, we have $\cD^N(u) = \pi_N \cD(\pi_N u)$, and we deduce that :
	\begin{equation*}
		\| \Phi_{u_0} - \Phi^N_{u_0} \|_{X_{(0,\delta)}} \leq C_{\gamma} (\| \pi_N^\perp \cD(\Phi_{u_0}) \|_{X_{(0,\delta)}} + \| \pi_N(\cD(\Phi_{u_0}) - \cD(\pi_N \Phi_{u_0})) \|_{X_{(0,\delta)}} )
	\end{equation*}
	Again by contraction,  
	\begin{equation*}
		 \| \pi_N(\cD(\Phi_{u_0}) - \cD(\pi_N \Phi_{u_0})) \|_{X_{(0,\delta)}} \leq  \| \cD(\Phi_{u_0}) - \cD(\pi_N \Phi_{u_0}) \|_{X_{(0,\delta)}} \leq \gamma \| \pi_N^\perp \Phi_{u_0} \|_{X_{(0,\delta)}}
	\end{equation*}
	so,
	\begin{equation*}
			\| \Phi_{u_0} - \Phi^N_{u_0} \|_{X_{(0,\delta)}} \leq C_\gamma (\| \pi_N^\perp \cD(\Phi_{u_0}) \|_{X_{(0,\delta)}} + \| \pi_N^\perp \Phi_{u_0} \|_{X_{(0,\delta)}})
	\end{equation*}
	Consequently, by Proposition~\ref{prop embedding X sigma b into CHsig}, we have (with $C_\gamma$ changing):
	\begin{equation}\label{intermediate bound on phi - phiN}
		\begin{split}
			\sup_{u_0 \in K} \sup_{t \in [0,\delta]} \hsignorm{\Phi_{u_0}(t) - \Phi^N_{u_0}(t)} & \lesssim \sup_{u_0 \in K} 	\| \Phi_{u_0} - \Phi^N_{u_0} \|_{X_{(0,\delta)}} \\
			& \leq C_\gamma \sup_{v \in K_1} \| \pi_N^\perp v \|_{X_{(0,\delta)}} + C_\gamma \sup_{w \in K_2} \| \pi_N^\perp w \|_{X_{(0,\delta)}}
		\end{split}
	\end{equation}
	where:
	\begin{align*}
		K_1 &:= \{ \cD (\Phi_{u_0})  \hsp : \hsp u_0 \in K\} & &\textnormal{and} & K_2 &:= \{ \Phi_{u_0} \hsp : \hsp u_0 \in K\}
	\end{align*}
	These sets are compact subset of $X_{(0,\delta)}$ because $K \subset \hsig$ is compact and the maps:
	\begin{align*}
		\cD :& \hsp X_{(0,\delta)} \ra X_{(0,\delta)} & &\textnormal{and} & \Phi :& \hsp  \hsig \ra X_{(0,\delta)} \hsp ; \hsp u_0 \longmapsto \Phi_{u_0}
	\end{align*} 
	are continuous. Furthermore, the projectors $\pi^\perp_N$ are linear maps on $X_{(0,\delta)}$ and they satisfy $\norm{\pi^\perp_N}_{X_{(0,\delta)} \ra X_{(0,\delta)} } \leq 1$, and they also converge pointwisely to 0. We are then in position to use Lemma~\ref{lem cvgce on compact set of bdd lin map}, and conclude from~\eqref{intermediate bound on phi - phiN} that:
	\begin{equation}\label{local approx cvgce}
			\sup_{u_0 \in K} \sup_{t \in [0,\delta]} \hsignorm{\Phi_{u_0}(t) - \Phi^N_{u_0}(t)} \tendsto{N \ra \infty} 0 
	\end{equation}
	This concludes the first step of the proof. In what follows, we propagate by iteration this result in order to obtain the convergence up to time $T$. \\
	
	\underline{Long-time convergence :} Let $k \in \N$ such that $k \leq \lfloor \frac{T}{\delta} \rfloor$, and suppose	that:
	\begin{equation}\label{induction hypothesis}
		\sup_{u_0 \in K} \sup_{t \in [0,k\delta]} \hsignorm{\Phi_{u_0}(t) - \Phi^N_{u_0}(t)} \tendsto{N \ra \infty} 0
	\end{equation}
	Our goal is to prove that the convergence holds on the interval $[0,(k+1)\delta]$. Indeed, if we do so, then from~\eqref{local approx cvgce}, and by induction, we will obtain the convergence on the interval $[0,T]$, and similarly on $[-T,0]$ by the same argument; it will then complete the proof of the proposition. \\
	By \eqref{induction hypothesis}, we have in particular, with $t_k := k\delta$,
	\begin{equation*}
		\sup_{u_0 \in K} \hsignorm{\Phi_{u_0}(t_k) - \Phi^N_{u_0}(t_k)} \tendsto{N \ra \infty} 0
	\end{equation*}
	Thus, for some $N_0 \in \N$, we have $N \geq N_0 \implies \hsignorm{\Phi^N_{u_0}(t_k)} \leq \Lambda_0$, where we recall that $\Lambda_0$ is defined in~\eqref{Lambda0}. As a consequence, denoting $u_k := \Phi_{u_0}(t_k)$ and $v_k := \Phi^N_{u_0}(t_k)$, the maps \begin{center}
		 $\Gamma_{u_k}^{t_k} : B_R^{t_k}(\delta) \ra B_R^{t_k}(\delta)$ and $\Gamma_{v_k}^{t_k,N} : B_R^{t_k}(\delta) \ra B_R^{t_k}(\delta)$ 
	\end{center}defined in~\eqref{duhamel maps} are contractions, with still the same universal contraction coefficient $\gamma \in (0,1)$. Furthermore, $\Phi_{u_0}$ and $\Phi^N_{u_0}$ coincide on $I_k := (k\delta, (k+1)\delta)$ with the fixed point of $\Gamma_{u_k}^{t_k}$ and $\Gamma_{v_k}^{t_k,N}$ respectively. Now, the argument is very similar to that in the first step. Indeed, we start by writing:
	\begin{equation}\label{u - uN on Ik}
		 \Phi_{u_0} - \Phi^N_{u_0} = \Gamma^{t_k}_{u_k}(\Phi_{u_0}) - \Gamma^{t_k,N}_{v_k}(\Phi^N_{u_0})  = S_{t_k}(u_k-v_k) + \Gamma^{t_k}_{u_k}(\Phi_{u_0}) - \Gamma^{t_k,N}_{u_k}(\Phi^N_{u_0}) 
 	\end{equation}  
 	On the one hand,  
 	\begin{equation*}
 	\|S_{t_k}(u_k-v_k) \|_{X_{I_k}} \lesssim \hsignorm{u_k-v_k} = \hsignorm{\Phi_{u_0}(t_k)-\Phi^N_{u_0}(t_k)} \tendsto{N \ra \infty} 0
 	\end{equation*} 
 	uniformly in $u_0 \in K$. On the other hand, we treat the term $\Gamma^{t_k}_{u_k}(\Phi_{u_0}) - \Gamma^{t_k,N}_{u_k}(\Phi^N_{u_0})$ exactly as we treated the term $\Gamma_{u_0}(\Phi_{u_0}) - \Gamma_{u_0}^N(\Phi^N_{u_0})$ in the first part of the proof. Then, from~\eqref{u - uN on Ik}, we indeed obtain:
 	\begin{equation*}
 		\sup_{u_0 \in K} \sup_{t \in I_k} \hsignorm{\Phi_{u_0}(t) - \Phi^N_{u_0}(t)} \lesssim \sup_{u_0 \in K} \| \Phi_{u_0} - \Phi^N_{u_0}\|_{X_{I_k}} \tendsto{ N \ra \infty} 0
 	\end{equation*}
 	Finally, combining this with~\eqref{induction hypothesis} yields:
 	\begin{equation*}
 			\sup_{u_0 \in K} \sup_{t \in [0,k(\delta +1)]} \hsignorm{\Phi_{u_0}(t) - \Phi^N_{u_0}(t)}  \tendsto{ N \ra \infty} 0
 	\end{equation*}
 	which completes the proof.
\end{proof}